\documentclass[11pt, reqno]{amsart}
\usepackage{amsxtra}
\usepackage{amsmath,amsthm,amssymb,enumerate}
\usepackage{setspace}
\usepackage{color}
\let\oldtocsection=\tocsection

\let\oldtocsubsection=\tocsubsection

\let\oldtocsubsubsection=\tocsubsubsection

\renewcommand{\tocsection}[2]{\hspace{0em}\oldtocsection{#1}{#2}}
\renewcommand{\tocsubsection}[2]{\hspace{2em}\oldtocsubsection{#1}{#2}}
\renewcommand{\tocsubsubsection}[2]{\hspace{4.5em}\oldtocsubsubsection{#1}{#2}}

\makeatletter
\def\subsection{\@startsection{subsection}{2}
 \z@{.5\linespacing\@plus.7\linespacing}{-.5em}
 {\normalfont\bfseries}}
\def\subsubsection{\@startsection{subsubsection}{3}
 \z@{.5\linespacing\@plus.7\linespacing}{-.5em}
 {\normalfont\bfseries}}
\makeatother

\newtheorem{theorem}{Theorem}[section]
\newtheorem*{theorem*}{Theorem}
\newtheorem{lemma}[theorem]{Lemma}
\newtheorem*{lemma*}{Lemma}
\newtheorem{example}[theorem]{Example}
\newtheorem*{example*}{Example}

\newtheorem*{corollary*}{Corollary}
\newtheorem{definition}[theorem]{Definition}
\newtheorem*{definition*}{Definition}
\newtheorem{remark}[theorem]{Remark}
\newtheorem*{remark*}{Remark}
\newtheorem{notation}[theorem]{Notation}
\newtheorem*{notation*}{Notation}
\newtheorem{lemma-notation}[theorem]{Lemma-Notation}
\newtheorem*{lemma-notation*}{Lemma-Notation}
\newtheorem{lemma-definition}[theorem]{Lemma-Definition}
\newtheorem*{lemma-definition*}{Lemma-Definition}

\newtheorem{proposition}[theorem]{Proposition}
\newtheorem*{proposition*}{Proposition}

\numberwithin{equation}{section}

\def \RR {{\mathbb R}}         
\def \CC {{\mathbb C}}

\def\KK{{\mathbb K}}

\def \TT {{\mathbb T}}

\def \cE {{\mathcal{E}}}

\def \O  {{\mathcal{O}}}

\def \al {\alpha}

\def \la {\lambda}

\def \ra  {\rightarrow}           

\def \la {\langle}
\def \ra {\rangle}

\def \g  {\mathfrak{g}}   
\def \h  {\mathfrak{h}}

\def \b  {\mathfrak{b}}

\def \a  {\mathfrak{a}}

\def \t  {\mathfrak{t}}

\newcommand{\beqa}{\begin{eqnarray*}}                     
\newcommand{\eeqa}{\end{eqnarray*}}
\def \hs {\hspace{.2in}}
\def \hhs {\hspace{.1in}}
\def \lara {\la \, , \, \ra}

\def \bfv {\underline{v}}

\def \lara {\la \, , \, \ra}

\def \bw {{\bf w}}

\def \KK {{\mathbb{K}}}
\def \bk {{\KK}}

\def \bfu {{\bf u}}
\def \bfv {{\bf v}}
\def \bfw {{\bf w}}

\def \bfv {{\bf v}}
\def \bfw {{\bf w}}

\def \pist {\pi_{\rm st}}

\def \sQ {{\scriptscriptstyle Q}}

\def \sG {{\scriptscriptstyle G}}
\def \sX {{\scriptscriptstyle X}}

\def \sM {{\scriptscriptstyle M}}
\def \sN {{\scriptscriptstyle N}}

\def \piX {{\pi_{\scriptscriptstyle X}}}

\def \lrw {\longrightarrow}
\def \Pist {\Pi_{\rm st}}

\def \cD {{\mathcal D}}

\def \sL {{\scriptscriptstyle L}}

\def \sF {{\scriptscriptstyle F}}

\def \T {\mathbb T}

\def \ow {\overline{w}}
\def \ou {\overline{u}}
\def \ov {\overline{v}}

\def \cR {{\mathcal{R}}}

\def \calQ {{\mathcal{Q}}}

\def \sG {{\scriptscriptstyle G}}
\def \sB {{\scriptscriptstyle B}}
\def \sQ {{\scriptscriptstyle Q}}
\def \sT {{\scriptscriptstyle T}}
\def \sTT {{\scriptscriptstyle \TT}}

\def \sBS {{\scriptscriptstyle BS}}

\def \cP {{\mathcal{P}}}
\def \Pos {{\rm Pos}}
\def \Poly {{\rm Poly}}

\def \cA {{\mathcal{A}}}

\def \sT {{\scriptscriptstyle T}}

\def \bphi {\boldsymbol{\phi}}

\def \bbeta {\boldsymbol{\beta}}

\def \bsigma {\boldsymbol{\sigma}}
\def \br {{\bf r}}
\def \tbfu {\tilde{\bfu}}

\def \sp {\mathfrak{sp}}
\def \SP {{\rm SP}}

\begin{document}

\setlength{\baselineskip}{1.1\baselineskip}
\title[Bott-Samelson atlases, total positivity, and Poisson structures]
{Bott-Samelson atlases, total positivity, and Poisson structures on some homogeneous spaces}
\author{Jiang-Hua Lu}
\address{
Department of Mathematics   \\
The University of Hong Kong \\
Pokfulam Road               \\
Hong Kong}
\email{jhlu@maths.hku.hk}
\author{Shizhuo Yu}
\email{yusz@connect.hku.hk}
\date{}
\begin{abstract} 
Let $G$ be a connected and simply connected complex semisimple Lie group.
For a collection of homogeneous $G$-spaces $G/Q$, we construct a finite atlas $\cA_{\sBS}(G/Q)$ on $G/Q$, called the {\it Bott-Samelson atlas},
and we prove that all of its coordinate functions are
positive with respect to the Lusztig positive structure on $G/Q$.
We also show that the standard Poisson structure
$\pi_{\sG/\sQ}$ on $G/Q$ is presented, in each of the coordinate charts of $\cA_{\sBS}(G/Q)$,
as a symmetric Poisson CGL extension (or a certain localization thereof) in the sense of Goodearl-Yakimov, making
$(G/Q, \pi_{\sG/\sQ}, \cA_{\sBS}(G/Q))$ into a {\it Poisson-Ore variety}. Examples of $G/Q$ include
$G$ itself, $G/T$, $G/B$, and $G/N$,  where $T \subset G$ is a maximal torus, $B \subset G$ a Borel subgroup, and
$N$ the uniradical of $B$.
\end{abstract}
\maketitle
\tableofcontents

\addtocontents{toc}{\protect\setcounter{tocdepth}{1}}
\section{Introduction and the main results}\label{sec:intro}
\subsection{Introduction}\label{ssec:intro}
Let $G$ be a connected and simply connected complex semi-simple Lie group with Lie algebra $\g$. Certain
geometric aspects of $G$ were revealed after the discovery of quantum groups,  among them the notion of {\it total positivity} on
$G$ and the {\it standard multiplicative Poisson structure}\footnote{The word ``standard" here refers to the fact that the Poisson structure $\pist$ is defined using the standard classical $r$-matrix on the Lie algebra of $G$, as opposed to the more general Belavin-Drinfeld ones \cite{BD}.}  $\pist$ on $G$. Indeed,
Lusztig introduced \cite{Lusz:pos, Lusz:intro, Lusz:survey} the {totally positive part $G_{>0}$ of $G$
 via his work on representations of the quantized universal enveloping algebra $U_q(\g)$, while
Drinfeld introduced \cite{dr:hamil, dr:quantum} the standard multiplicative
Poisson structure $\pist$ on $G$ as the semi-classical limit of the (standard)
quantum coordinate ring
$\CC_q[G]$ of $G$.  The pair $(G, \pist)$ is the prototypical example of a complex
Poisson Lie group (see, for example, \cite{dr:hamil, etingof-schiffmann} and $\S$\ref{ssec:GBcs}).

Both the totally positive part $G_{>0}$ of $G$ defined by Lusztig and the standard Poisson structure on $G$
can be extended to certain homogeneous spaces of $G$.
In this paper, for a special collection of homogeneous spaces $G/Q$, we construct a finite atlas on $G/Q$, called the {\it Bott-Samelson atlas}, and we
prove some remarkable properties of both the Lusztig total positivity and
the standard Poisson structure on $G/Q$, expressed through the Bott-Samelson atlas.
The homogeneous spaces considered in this paper are precisely all the diagonal $G$-orbits in the double flag variety $(G/B) \times (G/B)$ and in
$(G/ B) \times (G/N)$, where $B$ is a Borel subgroup of $G$ and $N$ the unipotent radical of $B$, and particular examples include
$G$ itself, $G/T, G/B$ and $G/N$,
where $T \subset B$ is a maximal torus of $G$. See
\eqref{eq:Q-list} for the precise descriptions of all the subgroups $Q$. 

To give the precise statements of the main results of the paper, we first briefly review some background on total positivity and Poisson structures.


The geometric structures underlying total positivity are the so-called {\it positive structures} on varieties, as defined in \cite{BK:CrystalsII, FG:IHES}.
Here we briefly recall (see $\S$\ref{ssec:pos} for details)
that a positive structure on an irreducible rational complex variety $X$ is a {\it positive equivalence class} $\cP^{\sX}$
of toric charts on $X$ (see Definition \ref{de:equi-charts}).
A positive structure $\cP^{\sX}$ on $X$ gives rise to a well-defined
 {\it totally positive part} $X_{>0}$ of $X$, and thus the notion of total positivity, by setting all the coordinates in one (equivalently, any)
toric chart in $\cP^{\sX}$ to be positive. The positive structure $\cP^{\sX}$ also gives rise to the semi-field $\Pos(X, \cP^\sX)$ of {\it positive functions},
which, by definition, are non-zero rational functions on $X$ that have subtraction-free expressions in the coordinates of one (equivalently, any)
toric chart in $\cP^\sX$.

For a connected and simply connected complex semi-simple Lie group $G$, positively equivalent toric charts on
$G$ giving rise to the totally positive part $G_{>0}$ of Lusztig were described in \cite{Lusz:pos}
(although the term ``positive structure" was not used), and it follows from \cite[Theorem 1.11 and Theorem 1.12]{FZ:double}
that all generalized minors on $G$ are positive functions, and that
an element $g \in G$ lies in $G_{>0}$ if and only if $\Delta(g) > 0$ for every
generalized minor $\Delta$.
Positive structures on double Bruhat cells and Schubert varieties induced by the Lusztig positive structure on $G$ have been considered
in \cite{ABHY1, ABHY2, BZ:total, FZ:double}. See also \cite{FG:IHES} for related positive structures
in higher Teichm${\rm \ddot{u}}$ller theory.

Extending the case for $G$ (and some other cases already existing in the literature),
we define in $\S$\ref{ssec:pos-GQ} the Lusztig positive structure $\cP^{\sG/\sQ}_{\rm Lusztig}$ on $G/Q$ for each $Q$ in
 \eqref{eq:Q-list}.

Turning to Poisson structures, it is easy to prove (see $\S$\ref{ssec:piGQ} for detail)
that for each of the homogeneous spaces $G/Q$ from \eqref{eq:Q-list}, the standard Poisson
structure $\pist$ on $G$ projects to a well-defined Poisson structure on $G/Q$, which we will denote by $\pi_{\sG/\sQ}$ and refer to as the
{\it standard Poisson structure} on $G/Q$. The pair
$(G/Q, \pi_{\sG/\sQ})$ is then a prototypical example of {\it Poisson homogeneous spaces} \cite{dr:homog} of the Poisson Lie group $(G, \pist)$.


It has been noticed, see, for example \cite{GY:PNAS, GY:AMSbook}, that the quantum coordinate rings of many spaces related to
the complex semi-simple Lie group $G$ can be presented as iterated Ore extensions, a notion from the theory of non-commutative rings
\cite{GW}. Correspondingly, their semi-classical limits, which are now the
(classical) coordinate rings with Poisson brackets, are iterated Poisson-Ore extensions.
Here recall \cite{Oh} that for a Poisson algebra $(A, \{\, , \, \})$ over $\CC$, an {\it Ore extension}
of $(A, \{\, , \, \})$ is the $\CC$-algebra $A[z]$ with a Poisson bracket $\{\,, \, \}$ which extends the Poisson bracket on $A$ and
satisfies $\{z,  A\} \subset zA + A$.
A polynomial Poisson algebra $A = (\CC[z_1,  \ldots, z_n], \{\, , \, \})$
is called an {\it iterated Poisson-Ore extension} if
\begin{equation}\label{eq:zzz-1}
\{\bk[z_{1}, \ldots, z_{i-1}], \; z_i\} \subset z_i \bk[z_{1}, \ldots, z_{i-1}] + \bk[z_{1}, \ldots, z_{i-1}], \hs 2 \leq i \leq n.
\end{equation}
Such an iterated Poisson-Ore extension is said to be {\it symmetric} if it also satisfies
\begin{equation}\label{eq:zzz-0}
\{z_i, \; \bk[z_{i+1}, \ldots, z_n]\} \subset z_i \bk[z_{i+1}, \ldots, z_n] + \bk[z_{i+1}, \ldots, z_n], \hs \forall \; 1 \leq i \leq n-1.
\end{equation}
Iterated Poisson-Ore extensions are Poisson analogs of iterated Ore extensions.

With additional assumptions on the existence of compatible rational actions by split complex tori,
K. Goodearl and M. Yakimov introduced in \cite{GY:PNAS, GY:Poi} a special class of symmetric iterated Poisson-Ore extensions, called
{\it symmetric Poisson CGL extensions} (named after G. Cauchon, K. Goodearl, and E. Letzter) and developed an extensive theory on such extensions
in connection with cluster algebras. In particular, one of the main results of \cite{GY:PNAS, GY:Poi} says that a presentation of a Poisson
algebra $(A, \{\, , \, \})$ as a symmetric Poisson CGL extension naturally gives rise to a cluster algebra structure on $A$ compatible with the Poisson
bracket $\{\,, \, \}$ in the sense of \cite{GSV:book}, i.e., all the cluster variables from the same cluster have log-canonical Poisson brackets.
See \cite{GY:PNAS, GY:BZ-conjecture, GY:AMSbook, GY:double} for applications of the Goodearl-Yakimov theory to
classical and quantum cluster structures on the coordinate rings of double Bruhat cells for complex semisimple Lie groups and of Schubert cells for
symmetrizable Kac-Moody groups.

In this paper,  for an irreducible rational $\TT$-Poisson variety $(X, \piX)$, where $\TT$ is a split $\CC$-torus,
we define
a {\it $\TT$-Poisson-Ore atlas} for $(X, \piX)$ to be an atlas $\cA_\sX$ on $X$ consisting of $\TT$-invariant coordinate charts, in {\it each} one of which
$\piX$ is presented as a symmetric Poisson CGL extension or a localization thereof by homogeneous Poisson prime elements
(Definition \ref{de:Poi-CGL}).
By a {\it $\TT$-Poisson-Ore variety}, we mean a triple $(X, \pi_\sX, \cA_\sX)$, where
$(X, \piX)$ is an irreducible rational $\TT$-Poisson variety,  and $\cA_{\sX}$ is a
{\it $\TT$-Poisson-Ore atlas} for $(X, \piX)$.  A $\TT$-Poisson-Ore variety is also simply called a {\it Poisson-Ore variety}.
 See $\S$\ref{ssec:de-CGL} for detail.

By  \cite{Bog-B:uniformly-rational}, an $n$-dimensional irreducible rational complex variety $X$ is said to be {\it uniformly rational} if it admits
a cover by Zariski open subsets  of $\CC^n$.
A Poisson-Ore variety is thus uniformly rational.
Given that the requirements on a Poisson algebra to be an iterated Poisson-Ore
extension are very restrictive,
and that the changes of coordinates between different coordinate charts are in general highly non-trivial birational
transformations, it is a remarkable feature of any Poisson variety if it admits a Poisson-Ore atlas.
We prove that this is the case for all the homogeneous spaces $(G/Q, \pi_{\sG/\sQ})$ considered in this paper.

What serves as a Poisson-Ore atlas for each $(G/Q, \pi_{\sG/\sQ})$ is a finite atlas $ \cA_{\sBS}(G/Q)$ on $G/Q$
constructed in this paper, which we call the
{\it Bott-Samelson atlas} on $G/Q$. The Bott-Samelson atlas is canonical in the sense that its construction depends only on
the choice of a pinning
$\{T \subset B, \{e_\al\}_{\al \in \Gamma}\}$ for $G$ (see \cite{Lusz:pos} and Notation \ref{nota:intro0}).
In this paper we set up the Bott-Samelson atlas as a bridge connecting the Lusztig positive structure on $G/Q$ and the
standard Poisson structure $\pi_{\sG/\sQ}$. Further connections to cluster algebras and total positivity will be given in
\cite{Lu:cover}.

\subsection{Statements of main results and organization of the paper}\label{ssec:statements}
We first set up some notation to be used throughout the paper.

\begin{notation}\label{nota:intro0}
{\rm
We will fix a connected and simply-connected complex semisimple Lie group $G$ together with a  {\it pinning} \cite{Lusz:pos}, i.e., a collection
 $\{T \subset B, \{e_\al\}_{\al \in \Gamma}\}$, where $T$ is a maximal torus of $G$, $B$ is a Borel subgroup of $G$ containing
$T$, $\Gamma$ is the set of simple roots corresponding to the pair $(T, B)$, and $e_\al$ is a root vector for $\al \in \Gamma$.

Let $W = N_\sG(T)$ be the corresponding Weyl group, where $N_\sG(T)$ is the normalizer subgroup of $T$ in $G$.
Let $l: W \to {\mathbb{N}}$ be the length function on $W$.
Let $w_0$ be the longest element of $W$, and let $l_0 = l(w_0)$.

For any closed subgroup $P$ of $G$, denote  the image of $g \in G$ in $G/P$ by $g_\cdot P$.
The identity element of a group will always be denoted
as $e$.

Let $N$ unipotent radicals of $B$.
For $v \in W$, let
\[
B(v)\;\stackrel{{\rm def}}{=}\; B \cap \dot{v} B \dot{v}^{-1} \;\;\mbox{and}\;\;   N(v) \; \stackrel{{\rm def}}{=}\; N \cap \dot{v} N \dot{v}^{-1} \subset B(v),
\]
where for $v \in W$, $\dot{v}$ is any representative of $v$ in  $N_\sG(T)$.
\hfill $\diamond$
}
\end{notation}

In this paper, we
consider the homogeneous spaces $G/Q$ of $G$, where
\begin{equation}\label{eq:Q-list}
Q = B(v) \;\; \; \mbox{or}\;\;\; N(v),
\;\;\;  v \in W.
\end{equation}
Consider the diagonal $G$-action on $(G/B) \times (G/B)$ and on $(G/B) \times (G/N)$. It is well-known that the sets of $G$-orbits in both
$(G/B) \times (G/B)$ and $(G/B) \times (G/N)$  are indexed by $W$ via $v \mapsto G(e_\cdot B, \dot{v}_\cdot B)$
and $v \mapsto G(e_\cdot B, \dot{v}_\cdot N)$. As the stabilizer subgroup of $G$ at $(e_\cdot B, \dot{v}_\cdot B)$ and at
$(e_\cdot B, \dot{v}_\cdot N)$ are respectively $B(v)$ and $N(v)$, we can thus identify the collections
$\{G/B(v): v \in W\}$ and $\{G/N(v): v \in W\}$ as the diagonal $G$-orbits in $(G/B) \times (G/B)$ and in $(G/B) \times (G/N)$ respectively.
In particular,
\[
B(w_0) = T, \hs N(w_0) = \{e\}, \hs B(e) = B, \hs N(e) = N,
\]
and we have $G/T$ and $G/B$ as the respective open and closed $G$-orbits in
$(G/B) \times (G/B)$  and $G$ and $G/N$ as the respective open and closed $G$-orbits in $(G/B) \times (G/N)$.

\medskip
The paper consists of three parts.

In the first part of this paper, for each $Q$ in \eqref{eq:Q-list},
we construct the Bott-Samelson atlas $\cA_{\sBS}(G/Q)$ on $G/Q$.
More precisely,
for each $Q$ in \eqref{eq:Q-list},  we consider the open cover of $G/Q$ by {\it shifted big cells}, i.e.,
\begin{equation}\label{eq:cover-1}
G/Q = \bigcup_{w \in W} wB^-B/Q,
\end{equation}
where $B^-$ is the Borel subgroup of $G$ such that $B^- \cap B = T$.
For $u \in W$, let $\cR(u)$ be the set of all reduced words of $u$. For $v \in W$, let
\[
\cR_{v}  = \bigcup_{w \in W} \cR(w_0w^{-1}) \times \cR(w) \times \cR(v).
\]
Using the fixed pinning for $G$, we construct, for each
$\br \in \cR(w_0w^{-1}) \times \cR(w) \times \cR(v)$, the
{\it Bott-Samelson parametrizations}
\[
\bsigma^{\br}_{\sB(v)}:\;\; \CC^l \,\stackrel{\sim}{\longrightarrow} \, wB^-B/B(v) \hs
\mbox{and} \hs \bsigma^{\br}_{\sN(v)}:\;\; \CC^l \times (\CC^\times)^d\,\stackrel{\sim}{\longrightarrow} \, wB^-B/N(v),
\]
where $l  = l(w_0) + l(v) = \dim G/B(v)$, and $d = \dim T$ (see \eqref{eq:bga-Bv} and \eqref{eq:bga-Nv}).
The Bott-Samelson atlas $\cA_{\sBS}(G/Q)$, for $Q = B(v)$ or $N(v)$, is then defined to be
\[
\cA_{\sBS}(G/Q) = \{\bsigma^{\br}_{\sQ}: \; \br \in \cR_v\}.
\]
We refer to any coordinate chart in $\cA_{\sBS}(G/Q)$ as a {\it Bott-Samelson coordinate chart} on $G/Q$, and the resulting coordinates
(on a shifted big cell)
{\it Bott-Samelson coordinates}.  A Bott-Samelson coordinate, being a regular function on some shifted big cell in $G/Q$, can thus be also
regarded as a rational function on $G/Q$.

An outline of the construction of the Bott-Samelson atlas is given in $\S$\ref{ssec:outline} and details, including explicit formulas for
all the Bott-Samelson coordinates, are given in $\S$\ref{sec:ABS}.

In the second part of the paper, extending the Lusztig positive structure on $G$,
we first define  in $\S$\ref{ssec:pos-GQ} the Lusztig positive structure $\cP^{\sG/\sQ}_{\rm Lusztig}$ on each $G/Q$, and we prove
the following Theorem A in $\S$\ref{ssec:BSpos-GQ}.

\medskip
\noindent
{\bf Theorem A.}
For any $v \in W$ and $Q = B(v)$ or $N(v)$, all the Bott-Samelson coordinates on $G/Q$,
when regarded as rational functions on $G/Q$,
are positive with respect to $\cP^{\sG/\sQ}_{\rm Lusztig}$.

\medskip
As a consequence of Theorem A, all Bott-Samelson coordinates on $G/Q$ take positive values at every point in $(G/Q)_{>0}$, the totally positive part of
$G/Q$ defined by $\cP^{\sG/\sQ}_{\rm Lusztig}$.

We remark that the Bott-Samelson coordinate charts are not to be confused with toric charts in $\cP^{\sG/\sQ}_{\rm Lusztig}$. In particular, given
any toric chart in $\cP^{\sG/\sQ}_{\rm Lusztig}$ with coordinates $(c_1, \ldots, c_{d(\sQ)})$
and any Bott-Samelson chart with coordinates $(z_1, \ldots, z_{d(\sQ)})$,  where $d(Q) = \dim (G/Q)$,
while Theorem A implies that each $z_j$ has a subtraction-free expression in $(c_1, \ldots, c_{d(\sQ)})$,
the $c_j$'s, in general,  do not have subtraction-free expressions in $(z_1, \ldots, z_{d(\sQ)})$.  This already happens for
$G/Q = SL(2, \CC)$ (see Remark \ref{re:SL2}). Furthermore, the changes of coordinates between two Bott-Samelson coordinate charts
are in general not positive. See Example \ref{ex:SL3-0} for the case of $G/Q = G =SL(3, \CC)$, where two out of the $16$ Bott-Samelson
coordinate charts on $G$, as well as the coordinate transformations between them, are given.

\medskip
The third part of the paper concerns the  standard Poisson structure $\pi_{\sG/\sQ}$ on $G/Q$.
The following Theorem B summarizes
Theorem \ref{th:CGL-Bv} (for $Q = B(v)$) and Theorem \ref{th:CGL-Nv} (for $Q = N(v)$). See Example \ref{ex:SL4-GB} for an illustration of Theorem B
for $SL(4, \CC)/B$.

\medskip
\noindent
{\bf Theorem B.} {\it For any $v \in W$ and $Q = B(v)$ or $N(v)$,  the Bott-Samelson atlas $\cA_{\sBS}(G/Q)$
 is a $T$-Poisson-Ore atlas for the Poisson structure $\pi_{\sG/\sQ}$, making
$(G/Q, \pi_{\sG/\sQ}, \cA_{\sBS}(G/Q))$ into a $T$-Poisson-Ore variety.}

\medskip
For another property of the Bott-Samelson coordinates with respect to the Poisson structure $\pi_{\sG/\sQ}$,
recall that for a smooth affine complex Poisson variety $(X, \piX)$, each $f \in \CC[X]$ has Hamiltonian vector field $H_f \in {\mathfrak{X}}^1(X)$
given by $H_f(f_1) = \piX(df, df_1)$ for $f_1 \in \CC[X]$. We say that $f$ has complete Hamiltonian flow if all the integral curves of
the holomorphic vector field $H_f$ on $X$ are defined over the entire $\CC$.  As a direct consequence of Theorem B and a general property of
Poisson-Ore varieties proved in $\S$\ref{ssec:hamil},
we have the following Theorem C.

\medskip
\noindent
{\bf Theorem C.} {\it For any $v \in W$ and $Q = B(v)$ or $N(v)$,  and for any $w \in W$,
all the Bott-Samelson coordinates on $wB^-B/Q$
have complete Hamiltonian flows in $wB^-B/Q$ with respect to the Poisson structure $\pi_{\sG/\sQ}$}

\medskip
The integral curve of any Bott-Samelson coordinate on $wB^-B/Q$ through any point $q \in wB^-B/Q$
lies entirely in the symplectic leaf $\Sigma_q$ of $\pi_{\sG/\sQ}$ through $q$, and thus also in the {\it $T$-leaf}
of $\pi_{\sG/\sQ}$ through $q$, defined as $T\Sigma_q = \bigcup_{t \in T} t\Sigma_q$.
Generalizing the case of $G$ in \cite{hodges, reshe-4, Kogan-Z}, where the $T$-leaves of $\pist$ in $G$ are shown to be precisely the double Bruhat cells,
 and the case of $G/B$ in \cite{GY:GP}, where the $T$-leaves of $\pi_{\sG/\sB}$ are shown to be precisely the open Richardson varieties, we
also determine the $T$-leaves  of $\pi_{\sG/\sQ}$ for $Q = B(v)$ or $N(v)$ for all $v \in W$.  See $\S$\ref{ssec:D} for detail.

\medskip
Theorem A and Theorem B together suggest further more direct connections between the Lusztig positive structure and the
standard Poisson structure on $G/Q$. One such connection will be established in \cite{Lu:cover} through cluster algebras. More precisely,
it will be proved in \cite{Lu:cover},
again for $v \in W$ and $Q = B(v)$ or $N(v)$,
that the Bott-Samelson atlas on $G/Q$ gives rise to a {\it cluster open cover of $G/Q$ compatible with the Lusztig positive structure}, i.e.,
the cluster structures on all the shifted big cells arising from the symmetric Poisson CGL presentations of the Poisson structure $\pi_{\sG/\sQ}$ via
the Goodearl-Yakimov theory
\cite{GY:PNAS, GY:Poi}, while not in general mutation equivalent,
are all compatible with the Lusztig positive structure in the sense that {\it every one} of their clusters defines a
toric chart in
$\cP^{\sG/\sQ}_{\rm Lusztig}$.

While this paper and \cite{Lu:cover} explore connections between Poisson structures and total positivity via the notion of
Poisson-Ore varieties, we
believe that the latter will also find applications to other areas such as integrable systems, as hinted by Theorem C,
and  tropicalization of positive varieties \cite{ABHY1, ABHY2, BK:CrystalsII}. Such topics will be investigated elsewhere.

\subsection{Outlines of the construction of the Bott-Samelson atlas and proofs of main results}\label{ssec:outline}
The key ingredients in our construction of the Bott-Samelson atlas $\cA_{\sBS}(G/Q)$ are  the so-called {\it generalized Bruhat cells}
which we now briefly recall:  for any positive integer $r$, consider the right action of
product group $B^r$ on $G^r$ by
\[
(g_1, \, g_2, \, \ldots, \, g_r) \cdot (b_1, \, b_2, \ldots, b_r) =
(g_1b_1, \, b_1^{-1}g_2b_2, \, \ldots, \, b_{r-1}^{-1}g_rb_r), \hs g_j\in G, \, b_j \in B.
\]
Let $F_r$ be the quotient space of $G^r$ by $B^r$, which is also denoted as
\begin{equation}\label{eq:Fn}
F_r =\, G \times_B G \times_B \cdots \times_B G/B.
\end{equation}
For $\bfu = (u_1, \ldots, u_r) \in W^r$, the {\it generalized Bruhat cell} $\O^\bfu$ is defined in \cite{Lu-Mou:flags}
as
\begin{equation}\label{eq:Ou-de}
\O^\bfu = Bu_1B\times_B \cdots \times_B Bu_rB/B
\stackrel{{\rm def}}{=} \varpi_r((Bu_1B) \times \cdots \times (Bu_rB)) \subset F_r,
\end{equation}
where $\varpi_r: G^r \to F_r$ is the projection.
Note that when $r = 1$, generalized Bruhat cells
are just the $B$-orbits in $G/B$,  which we will refer to as {\it Bruhat cells} (also called Schubert cells in the literature) and denote as
\[
\O^w = BwB/B \subset G/B, \hs w \in W.
\]
It easy to see that $\dim \O^\bfu = l(\bfu):=l(u_1) + \cdots +l(u_r)$.  Fixing a one-parameter subgroup for each simple root,
we will see in $\S$\ref{ssec:GBcs-coor} that any choice
of $\tilde{\bfu} \in \cR(u_1) \times \cdots \times \cR(u_r)$ naturally gives rise to a {\it Bott-Samelson parametrization}
$\bbeta^\br: \CC^{l(\bfu)} \to \O^\bfu$, and we call the resulting coordinates on $\O^\bfu$ the
 {\it Bott-Samelson coordinates} on $\O^\bfu$ defined by $\tilde{\bfu}$.

The core of our construction of the Bott-Samelson atlases on $G/B(v)$ and $G/N(v)$, where $v \in W$, consists of two explicit isomorphisms
\begin{align*}
&J^w_{\sB(v)}:\;\; wB^-B/B(v) \longrightarrow \O^{w_0w^{-1}} \times \O^{(w, v)},\\
&J^w_{\sN(v)}:\;\; wB^-B/N(v) \longrightarrow \O^{w_0w^{-1}} \times \O^{(w, v)} \times T,
\end{align*}
given in \eqref{eq:Jw-Bv-0} - \eqref{eq:Jw-Nv}, where $w \in W$ is arbitrary.
Given any
\[
\br = (\bw^0, \bw, \bfv) \in \cR(w_0w^{-1}) \times \cR(w) \times \cR(v),
\]
by composing
$J^w_{\sB(v)}$ and $J^w_{\sN(v)}$ with the Bott-Samelson parametrizations of $\O^{w_0w^{-1}}$ and $\O^{(w, v)}$ defined
respectively by $\bw^0$ and $(\bw, \bfv)$ (and any isomorphism $T \cong (\CC^\times)^d$ in the case of $N(v)$),
we obtain the desired {\it Bott-Samelson parametrizations} (see \eqref{eq:bga-Bv} and \eqref{eq:bga-Nv})
\[
\bsigma_{\sB(v)}^\br: \,\CC^{l} \longrightarrow wB^-B/B(v) \hs \mbox{and} \hs
\bsigma_{\sN(v)}^\br: \,\CC^{l} \times (\CC^\times)^d\longrightarrow wB^-B/N(v),
\]
where again $l = l(w_0) + l(v) = \dim G/B(v)$ and $d = \dim T$.  Note that
each shifted big cell $wB^-B/B(v)$ or $wB^-B/N(v)$ may have more than one Bott-Samelson parametrization,
the precise number being $|\cR(w_0w^{-1})|
+ |\cR(w)| + |\cR(v)|$.

\begin{remark}\label{re:KLKWY}
{\rm
When $Q = B$ (and thus $v = e$) and $w \in W$, the isomorphism
\[
J^w_\sB:\;\;\; wB^-B/B \; \stackrel{\sim}{\longrightarrow} \;\O^{w_0w^{-1}} \times \O^w
\]
was  stated  in \cite[Lemma A.4]{KL:79},
and the explicit formula  was given
in \cite{KWY}.
Kazhdan and Lusztig used the isomorphism $J^w_\sB$
in the proof of
\cite[Theorem A2]{KL:79}, which expresses singularities of Schubert varieties (the closures of Bruhat cells)
in terms of Kazhdan-Lusztig polynomials, while A. Knutson, A. Woo, and A. Yong \cite{KWY}
used $J^w_\sB$
to show that singularities (and some other invariants) of Richardson
varieties in $G/B$ are determined by that of Schubert varieties. Here recall that
a Richardson variety in $G/B$ is the intersection of a Schubert variety with an opposite Schubert variety (the closure of a $B^-$-orbit).
The covering of $G/B$ by the collection $\{\O^{w_0w^{-1}} \times \O^w: w \in W\}$ via the
isomorphisms $J^w_\sB$ is called a {\it Bruhat atlas} on $G/B$ in \cite{HKL:Bruhatlas}.
\hfill $\diamond$
}
\end{remark}

To prove Theorem A, we use the explicit formulas for the Bott-Samelson coordinates in Proposition \ref{pr:zj-Bv} and
Proposition \ref{pr:zj-Nv}, but in a crucial way we also use \cite[Theorem 2.12]{FZ:double}, which describes certain collections of generalized
minors as positive transcendental bases for
the fields of rational functions on double Bruhat cells.   See $\S$\ref{ssec:BSpos-GQ} for detail.

To prove Theorem B, we first give a Poisson geometrical interpretation of the isomorphisms
$J^w_{\sB(v)}$ and $J^w_{\sN(v)}$ and then apply symmetric Poisson CGL extensions associated to generalized Bruhat cells established in
\cite{Elek-Lu:BS}.

More specifically, it is shown in \cite{Lu-Mou:flags} that for every $\bfu \in W^r$, the generalized Bruhat cell $\O^\bfu$ carries
a so-called {\it standard Poisson structures} $\pi_r$, also defined using the Poisson structure $\pist$ on $G$,
and it is proved in \cite{Elek-Lu:BS} that the Poisson algebra $(\CC[\O^\bfu], \pi_r)$ is a symmetric Poisson CGL extension in any of the
Bott-Samelson parametrizations of $\O^\bfu$ (see $\S$\ref{ssec:GBcs}).
For any $v, w \in W$, let $\pi_{1,2}$ be the unique Poisson structure on
$\O^{w_0w^{-1}} \times \O^{(w, v)}$ and $\pi_{1, 2, 0}$ the unique Poisson structure on
$\O^{w_0w^{-1}} \times \O^{(w, v)} \times T$ such that
\begin{align*}
J^w_{\sB(v)}: \;\;& \; (wB^-B/B(v), \; \pi_{\sG/\sB(v)}) \longrightarrow \left(\O^{w_0w^{-1}} \times \O^{(w, v)}, \;\pi_{1,2}\right) \;\; \; \mbox{and}\\
J^w_{\sN(v)}: \;\;& \; (wB^-B/N(v), \; \pi_{\sG/\sN(v)}) \longrightarrow \left(\O^{w_0w^{-1}} \times \O^{(w, v)} \times T, \;
\pi_{1,2,0}\right)
\end{align*}
are Poisson isomorphisms. Our Theorem \ref{th:JwQ-poi}, proved in Appendix A, says that $\pi_{1, 2}$ is a {\it mixed product} of
$\pi_1$ and $\pi_2$, i.e.,
\[
\pi_{1, 2} = (\pi_1, 0) + (0, \pi_2) + \mu,
\]
where $\mu$ is a certain {\it mixed term} expressed using the
$T$-actions on $\O^{w_0w^{-1}} \times \O^{(w, v)}$.  A similar statement holds for $\pi_{1, 2, 0}$.
Applying a general construction (see Lemma \ref{le:product}) on mixed products of symmetric Poisson CGL extensions, we immediately
prove that the Poisson structure $\pi_{\sG/\sQ}$ is presented as a symmetric Poisson CGL extension in every
Bott-Samelson parametrization of $wB^-B/Q$.
 Details of the presentations are given in
Theorem \ref{th:CGL-Bv} for $Q = B(v)$ and in Theorem \ref{th:CGL-Nv} for $Q = N(v)$.

Theorem C is a direct consequence of Theorem B and a general fact on Poisson-Ore varieties. More precisely,
in $\S$\ref{ssec:hamil}, we prove a general fact (see Proposition \ref{pr:hamil}) on the completeness of the Hamiltonian flows of all the
CGL generators for any
symmetric Poisson CGL extension. Consequently (see Theorem \ref{th:hamil}),  for any Poisson-Ore variety
$(X, \piX, \cA_\sX)$, all the coordinate functions in any coordinate chart in $\cA_\sX$ have complete Hamiltonian flows in that coordinate chart.
Theorem C then follows as a special case.

We finish this section by setting up more notation to be used for the rest of the paper.

\begin{notation}\label{nota:intro}
{\rm
Continuing with Notation \ref{nota:intro0},
let $\g$ and $\h$ be the respective Lie algebras of $G$ and $T$.
Recall that $\Gamma\subset \h^*$  is the set of all simple roots, and for each $\al \in \Gamma$, we have fixed a root vector $e_\al$ of $\al$ as
part of the pinning. For $\al \in \Gamma$, let
$e_{-\al}$ be the unique root vector for $-\alpha$  such that $h_\al :=[e_\al, e_{-\al}] \in \h$ satisfies $\al(h_\al) = 2$, and let
$x_\al: \CC \to G$ and $x_{-\al}: \CC \to G$ be the one-parameter subgroups given by
\[
x_\al(c) = \exp(c \,e_\al), \hs x_{-\al}(c) = \exp(c\,e_{-\al}), \hs c \in \CC.
\]
Correspondingly, one also has the co-character
$\al^\vee: \CC^\times \to T$
such that $\frac{d}{dc}|_{c=1} \al^\vee(c) = h_\al$.

For $\al \in \Gamma$, let $s_\alpha \in W$ be the corresponding simple reflection,
and choose the
representative $\overline{s_\alpha}$  of $s_\alpha \in W$ in $N_\sG(T)$
by
\cite[(1.8)]{FZ:double}
\[
\overline{s_\alpha} = x_\alpha(-1) x_{-\alpha}(1) x_\alpha(-1).
\]
Recall that $l: W \to {\mathbb{N}}$ is the length function on $W$, and that
for $w \in W$.
\[
\cR(w) = \{\bfw = (s_{\al_1},\, s_{\al_2}, \, \ldots, \, s_{\al_l}): \; \al_j \in \Gamma \;
\mbox{and} \;  s_{\al_1} s_{\al_2} \cdots s_{\al_l} = w\; \mbox{is reduced}\}
\]
Denote $\cR(e) = \emptyset$. By \cite[$\S$1.4]{FZ:double}, for any $\bfw = (s_{\al_1}, s_{\al_2}, \ldots, s_{\al_l})\in \cR(w)$,
\[
\overline{w} \;\stackrel{\rm def}{=} \;\overline{s_{\alpha_1}} \;\overline{s_{\alpha_2}} \;\cdots \; \overline{s_{\alpha_l}}
\]
is a representative of $w$ in $N_\sG(T)$ independent of the choice of $\bfw \in \cR(w)$. Moreover, recall that the weak order on $W$ is defined to be
$w_1 \preceq w$ if $w = w_1w_2$ and $l(w) = l(w_1) + l(w_2)$, and in such a case $\ow = \overline{w_1} \, \overline{w_2}$.
Denote
the (right) action of $W$ on $T$ by
\[
t^w = \ow^{\,-1}t\ow, \hs t \in T, \, w \in W.
\]
For a character $\lambda$ on $T$, the evaluation of $\lambda$  at $t \in T$ is denoted as $t^\lambda \in \CC^\times$.

Let $N^-$ be the unipotent radical of $B^-$.
For $g \in B^-B = N^-TN$, let $[g]_- \in N^-$, $[g]_0 \in T$, and $[g]_+ \in N$ be the unique elements such that $g = [g]_-[g]_0[g]_+$.
For a simple root $\alpha$, let $\omega_\al$ be the corresponding fundamental weight and let
$\Delta^{\omega_\al}$ be the corresponding principal minor on $G$, so that the restriction of $\Delta^{\omega_\al}$ to
$B^-B$ is given by
\[
\Delta^{\omega_\al}(g) = [g]_0^{\omega_\al}, \hs g \in B^-B.
\]
For $u, v \in W$ and $\al \in \Gamma$, let
$\Delta_{u\omega_\al, v\omega_\al}$ be the regular function on $G$ defined by
\[
\Delta_{u\omega_\al, v\omega_\al}(g) = \Delta^{\omega_\al}(\ou^{\, -1} g \ov), \hs g \in G.
\]
The functions $\Delta_{u\omega_\al, v\omega_\al}$ are called {\it generalized minors} on $G$ (see \cite{FZ:double}).
We will need the following property of generalized minors (see the proof of \cite[Proposition 2.7]{FZ:double}):
\begin{equation}\label{eq:Delta-omega-ast}
\Delta^{\omega_\al}(\overline{w_0}^{\, -1} g^{-1} \overline{w_0})
= \Delta^{\omega_\al}(\overline{w_0} g^{-1} \overline{w_0}^{\, -1})
=\Delta^{\omega_{\al^*}}(g), \hs g \in G,
\end{equation}
where for $\alpha \in \Gamma$, $\al^* = -w_0(\al)$. Note that as
$s_{\al^*} w_0 = w_0 s_\al$, one has
\[
\overline{w_0} = \overline{s_{\al^*}s_{\al^*} w_0} = \overline{s_{\al^*}}\; \overline{s_{\al^*} w_0} =  \overline{s_{\al^*}}\; \overline{w_0 s_\al}=
\overline{s_{\al^*}}\; \overline{w_0}\; \overline{s_\al}^{\, -1},
\]
and thus $\overline{s_{\al^*}} = \overline{w_0} \, \overline{s_\al} \, \overline{w_0}^{\, -1} =  \overline{w_0}^{\, -1} \, \overline{s_\al} \, \overline{w_0}$.
We also recall from \cite[$\S$2.1]{FZ:double} the involutive anti-automorphisms  $\tau$ (denoted as $T$ in
\cite[$\S$2.1]{FZ:double}) and $\iota$ of $G$ given by
\begin{equation}\label{eq:tau-iota}
t^\tau = t, \;\; x_\alpha(c)^\tau = x_{-\alpha}(c),\;\; t^\iota = t^{-1}, \;\;
x_\alpha(c)^\iota = x_{\alpha}(c), \;\; x_{-\alpha}(c)^\iota= x_{-\alpha}(c),
\end{equation}
where $t \in T$, $\alpha \in \Gamma$, and $c \in \CC$.
By \cite[(2.15)]{FZ:double}, one has
\begin{equation}\label{eq:Delta-all}
\Delta^{\omega_\al}((g^{-1})^\iota) = \Delta^{\omega_\al}(g^\tau) = \Delta^{\omega_\al}(g), \hs \al \in \Gamma, \; g \in G.
\end{equation}
The following facts are from \cite[Proposition 2.1]{FZ:double}: for any $w \in W$,
\begin{equation}\label{eq:ow-tau}
 \ow^{\, \tau} =\ow^{\, -1} \hs \mbox{and} \hs \ow^{\, \iota} =\overline{w^{-1}}
\end{equation}

For an integer $n \geq 1$, let $[1, n] =\{1, 2, \ldots, n\}$. All varieties in this paper are assumed to be smooth.
If an algebraic $\CC$-torus $\TT$ acts  an affine variety $X$, define the induced $\TT$-action on $\CC[X]$ by
\[
(t\cdot f)(x) = f(t\cdot x), \hs t \in \TT, \, f \in \CC[X], \, x \in X.
\]
}
\end{notation}

\medskip
\noindent{\bf Acknowledgments}: The earliest motivation for the work in this paper came from discussions with Xuhua He and
Alan Knutson on trying to understand the Poisson geometry behind the notion of {\it Bruhat atlas} proposed in \cite{HKL:Bruhatlas}.
We also thank Yipeng Mi and Yanpeng Li for helpful discussions.
The research in this paper was partially supported by
the Research Grants Council of the Hong Kong SAR, China (GRF 17304415 and GRF 17307718).
A version of Theorem B is contained in the University of Hong Kong PhD thesis \cite{Yu:thesis} of the second author.

\section{Construction of the Bott-Samelson atlas}\label{sec:ABS}
\subsection{Bott-Samelson coordinates on Bruhat cells}\label{ssec:Bcs-coor} We continue to use the set-up in Notation \ref{nota:intro0} and
 Notation \ref{nota:intro}.
Recall that for $u \in W$, the $B$-orbit  $\O^u = BuB/B$
 and $B^-$-orbit $B^-uB/B$ in $G/B$ are respectively called the Bruhat (or Schubert) cell
and opposite Bruhat (or Schubert) cell corresponding to $u$.
Set
\begin{equation}\label{eq:Nu-Num}
N_u = N \cap \ou N^- \ou^{\, -1}  \hs \mbox{and} \hs N_u^- = N^- \cap \ou N^- \ou^{\, -1}.
\end{equation}
It follows from the unique decompositions $BuB = N_u \ou B$ and $B^-uB = N_u^- \ou B$ that
\begin{align}\label{eq:Nu-Ou}
&N_u \longrightarrow BuB/B, \;\; n \longmapsto n\ou_\cdot B, \hs n \in N_u,\\
\label{eq:Num-Oum}
&N_u^- \longrightarrow  B^-uB/B, \;\; m \longmapsto m\ou_\cdot B, \hs m \in N_u^-,
\end{align}
are isomorphisms.
For future use,  note that
if $u = u_1u_2$ and $l(u) = l(u_1) + l(u_2)$, then
\begin{equation}\label{eq:Nu-Nu-1}
N_u \ou = (N_{u_1} \overline{u_1}) (N_{u_2} \overline{u_2})
\end{equation}
is direct product decomposition,
from which it follows that
\begin{equation}\label{eq:Nu-Nu-2}
N_{u_1} \subset N_u,  \;\;\;
\overline{u_1} N_{u_2} \overline{u_1}^{\, -1} \subset N_u, \;\;\; \mbox{and} \;\;\;
\overline{u_1}^{\, -1} N_u^- \,\overline{u_1} \subset N^-.
\end{equation}

 Let $u \in W$ and $\bfu = (s_{\al_1}, \ldots, s_{\al_k}) \in \cR(u)$. Recall from that for each $\al \in
\Gamma$ we have the one-parameter subgroup $x_\al: \CC \to N$ and  $\overline{s_\al} \in N_\sG(T)$. For
$z = (z_1, \ldots, z_k) \in \CC^k$, set
\begin{equation}\label{eq:guz}
g_\bfu(z) := x_{\al_1}(z_1)\overline{s_{\al_1}} \,\cdots \,x_{\al_{k}}(z_k)\overline{s_{\al_k}}.
\end{equation}
By \eqref{eq:Nu-Nu-1}, $g_\bfu(z) \in N_u\ou$ for every $z \in \CC^k$ and that
\begin{equation}\label{eq:z-gz}
\CC^k \ni z = (z_1, \ldots, z_k) \longmapsto g_\bfu(z) \in N_u \ou = N \ou \cap \ou N^-
\end{equation}
is an isomorphism from $\CC^k$ to $N_u \ou$. One thus has the {\it Bott-Samelson parametrization}
\begin{equation}\label{eq:z-Ou}
\CC^k \ni z = (z_1, \ldots, z_k) \longmapsto g_\bfu(z)_\cdot B  \in \O^u.
\end{equation}
The coordinates $(z_1, \ldots, z_k)$ on $\O^u$ via \eqref{eq:z-Ou} are called {\it Bott-Samelson coordinates} on $\O^u$ defined by the reduced word
$\bfu = (s_{\al_1}, \ldots, s_{\al_k})$ of $u$.

\begin{lemma}\label{le:zj-n} Let $u \in W$, $\bfu = (s_{\al_1}, \ldots, s_{\al_k}) \in \cR(u)$, and set
$n_\bfu(z) = g_\bfu(z) \ou^{\, -1} \in N_u$ for $z = (z_1, \ldots, z_k) \in \CC^k$. Then
\[
z_j = \Delta_{s_{\al_1} \cdots s_{\al_{j-1}} \omega_{\al_j}, \, s_{\al_1} \cdots s_{\al_j} \omega_{\al_j}}(n_\bfu(z)), \hs j \in [1, k].
\]
\end{lemma}

\begin{proof} For $i \in [1, k]$, let $g_{\al_i}(z_i) = x_{\al_i}(z_i)\overline{s_{\al_i}}$. Let $j \in [1, k]$. By \eqref{eq:Nu-Nu-1},
\[
\overline{s_{\al_1} \cdots s_{\al_{j-1}}} ^{\,-1} g_{\al_1}(z_1) \cdots g_{\al_{j-1}}(z_{j-1})\in N^-, \;\;
g_{\al_{j+1}}(z_{j+1}) \cdots g_{\al_k}(z_k) \overline{s_{\al_{j+1}} \cdots s_{\al_k}}^{\, -1} \in N.
\]
It follows that
\begin{align*}
\Delta_{s_{\al_1} \cdots s_{\al_{j-1}} \omega_{\al_j}, \, s_{\al_1} \cdots s_{\al_j} \omega_{\al_j}}(n_{\bfu}(z))
 &= \Delta^{\omega_{\al_j}}(\overline{s_{\al_1} \cdots s_{\al_{j-1}}} ^{\,-1} g_{\bfu}(z)
\overline{s_{\al_{j+1}} \cdots s_{\al_k}}^{\,-1})\\
&= \Delta^{\omega_{\al_j}}(g_{\al_j}(z_j)) = z_j.
\end{align*}
\end{proof}

\subsection{Bott-Samelson coordinates on generalized Bruhat cells}\label{ssec:GBcs-coor}
For an integer $r \geq 1$, recall from \eqref{eq:Fn} the quotient variety $F_r$ of $G^r$ by $B^r$, and recall that
associated to each $\bfu = (u_1, \ldots, u_r) \in W^r$ one has the
generalized Bruhat cell $\O^\bfu \subset F_r$ given in \eqref{eq:Ou-de}.
 For $(g_1, g_2, \ldots, g_r) \in G^r$, write
$[g_1, g_2, \ldots, g_r]_{\sF_r} = \varpi_r(g_1, \ldots, g_r) \in F_r$,
where $\varpi_r: G^r \to F_r$ is again the projection. On then has the disjoint union
\begin{equation}\label{eq:Fn-decomp}
F_r = \bigsqcup_{\bfu \in W^r} \O^\bfu,
\end{equation}
generalizing the decomposition $G/B = \bigsqcup_{u \in W} \O^u$.

Equip $F_r$ with the $T$-action by
\begin{equation}\label{eq:T-Fr}
t \cdot [g_1, \, g_2, \, \ldots, \g_r]_{\sF_r} = [tg_1, \, g_2, \, \ldots, g_r]_{\sF_r}, \hs t \in T, \,  g_1, g_2, \ldots, g_r \in G.
\end{equation}
It is clear that each generalized Bruhat cell $\O^\bfu \subset F_r$ is $T$-invariant.

Let $\bfu = (u_1, \ldots, u_r) \in W^r$.
By \eqref{eq:Nu-Ou},
one has the isomorphism
\begin{equation}\label{eq:O-bfu-para}
N_{u_1} \times N_{u_2} \times \cdots \times N_{u_r} \longrightarrow \O^{\bfu}, \;\;
(n_1, \, n_2, \, \ldots, \, n_r) \longmapsto [n_1\overline{u_1}, \, n_2\overline{u_2}, \, \ldots, \, n_r \overline{u_r}]_{\sF_r},
\end{equation}
where $n_i \in N_{u_i}$ for $i \in [1, r]$. In particular, $\dim \O^\bfu = l(u_1) + \cdots + l(u_r)$.

Let  now $\tilde{\bfu} = (\bfu_1, \; \bfu_2, \; \ldots, \; \bfu_r)  \in \cR(u_1) \times \cdots \times \cR(u_r)$.
Let $l = l(u_1) + \cdots + l(u_r)$ and  let $l_i = l(u_1) + \cdots + l(u_i)$ for $i \in [1, r]$. Also write
\begin{equation}\label{eq:tilde-bfu}
\tilde{\bfu} = (\bfu_1, \; \bfu_2, \; \ldots, \; \bfu_r) = (s_{\al_1}, \; s_{\al_2}, \; \ldots, \; s_{\al_{l}})
\in  W^{l},
\end{equation}
By \eqref{eq:z-gz}, one then has the {\it Bott-Samelson parametrization} $\bbeta^{\tilde{\bfu}}:   \CC^{l} \rightarrow\O^\bfu$ given by
\begin{equation}\label{eq:beta-Ou}
\bbeta^{\tilde{\bfu}} (z_1, \ldots, z_{l}) =  [g_{\bfu_1}(z_1, \ldots, z_{l_1}), \;\,g_{\bfu_2}(z_{l_1+1}, \ldots z_{l_2}), \, \ldots, \,
g_{\bfu_r}(z_{l_{r-1}+1}, \ldots, z_{l})]_{\sF_r}.
\end{equation}

\begin{definition}\label{de:BScoor-Ou} \cite{Elek-Lu:BS}
{\rm The coordinates $(z_1, \ldots, z_{l})$  via the Bott-Samelson parametrization
$\bbeta^{\tilde{\bfu}}: \CC^{l} \to \O^\bfu$ are called the
{\it Bott-Samelson coordinates on $\O^\bfu$} defined by  $\tilde{\bfu}$.
\hfill $\diamond$
}
\end{definition}

The following lemma follows immediately from Lemma \ref{le:zj-n}.

\begin{lemma}\label{le:zj-Ou}
For $\bfu = (u_1, \ldots, u_r) \in W^r$ and $\tbfu =(s_{\al_1}, \ldots, s_{\al_l}) \in \cR(u_1) \times \cdots \times \cR(u_r)$ as in \eqref{eq:tilde-bfu}, the Bott-Samelson coordinates
$(z_1, \ldots, z_l)$ on $\O^\bfu$ defined by $\tbfu$ are given as follows: for $i \in [1, r]$, $j \in [l_{i-1}+1, \; l_i]$ and
$n_i \in N_{u_i}$,
\[
z_j([n_1\overline{u_1}, \, n_2\overline{u_2}, \, \ldots, \, n_r \overline{u_r}]_{\sF_r})
 = \Delta_{s_{l_{i-1}+1} \cdots s_{j-2}s_{j-1} \omega_{\al_j}, \; s_{l_{i-1}+1} \cdots s_{j-1}s_j\omega_{\al_j}} (n_i),
\]
where $s_p = s_{\al_p}$ for $p \in [1, l]$. Furthermore, with respect to the $T$-action on $\CC[\O^\bfu]$ induced by the $T$-action on $\O^\bfu$ in
\eqref{eq:T-Fr} (see end of Notation \ref{nota:intro}), one has
\begin{equation}\label{eq:T-xi}
t\cdot z_j = t^{s_{\al_1} s_{\al_2} \cdots s_{\al_{j-1}}(\al_j)} z_j, \hs t \in T, \; j \in [1, l].
\end{equation}
\end{lemma}

\subsection{Construction of the Bott-Samelson atlas}\label{ssec:ABS}
Throughout $\S$\ref{ssec:ABS}, we fix $v\in W$ and let $Q = B(v)$ or $N(v)$. We will first construct decompositions of
the shifted big cells
$wB^-B/Q \subset G/Q$, $w \in W$, using generalized Bruhat cells, and we will then introduce Bott-Samelson coordinates on $wB^-B/Q$ using those
on generalized Bruhat cells.

Let $w \in W$. Note that every element in $wB^-B$ can be uniquely written as $a \ow b$, where $a \in \ow N^- \ow^{\, -1}$ and $b \in B$.
On the other hand, recall that $N_w = N \cap \ow N^- \ow^{\, -1}$ and $N_w^- = N^- \cap \ow N^- \ow^{\, -1}$. Using the
direct product decompositions
\begin{equation}\label{eq:Nw-nn}
\ow N^- \ow^{\,-1} = N^-_w N_w = N_w N^-_w,
\end{equation}
any $a \in \ow N^- \ow^{\, -1}$ can be further decomposed uniquely as
\begin{equation}\label{eq:ggg-000}
a = a_+ a_- = a^\prime_- a^\prime_+ \hs \mbox{with} \hs a_+, \,a^\prime_+ \in N_w, \; \;a_-, \,a^\prime_- \in N_w^-.
\end{equation}
Using $a_+^\prime a_-^{-1} = (a_-^\prime)^{-1} a_+$, one has
(see Notation \ref{nota:intro})
\begin{equation}\label{eq:ggg-inverse}
a = [a_+^\prime a_-^{-1}]_+a_- = [a_+^\prime a_-^{-1}]_-^{-1}a_+^\prime.
\end{equation}
It follows that the map
\begin{equation}\label{eq:xpm}
\ow N^- \ow^{\, -1} \longrightarrow N_w^- \times N_w, \;\; a \longmapsto (a_-, \, a_+^\prime),
\end{equation}
is an isomorphism, where again $a \in \ow N^- \ow^{\, -1}$ is decomposed as in \eqref{eq:ggg-000}. Consequently,
for any closed subgroup $Q$ of $B$, one has the isomorphism
\begin{equation}\label{eq:IwQ}
I^w_{\sQ}: \;\; wB^-B/Q \longrightarrow (B^-wB/B) \times (BwB/Q),\;\; a\ow b_\cdot Q \longmapsto (a_-\ow_\cdot B, \; a_+^\prime \ow b_\cdot Q),
\end{equation}
where $a \in \ow N^- \ow^{\, -1}$, decomposed as in \eqref{eq:ggg-000}, and $b \in B$.

Let now $Q = B(v)$ or $N(v)$, where $v \in W$, and note that one has the unique decomposition $B = N_vN(v)T$.
Thus every element in $BwB/B(v)$ or in $BwB/N(v)$ is
uniquely written as $n_1\ow {n_2}_\cdot B(v)$ or $n_1\ow n_2 t_\cdot N(v)$, where $n_1 \in N_w, n_2 \in N_v$, and $t \in T$.
It follows that one has the isomorphisms
\begin{align}\label{eq:BBQ-1}
\zeta^{(w, v)}_{\sB(v)}: \;\; & BwB/B(v) \longrightarrow \O^{(w, v)}: \;\; n_1\ow {n_2}_\cdot B(v)\longmapsto [n_1\ow, \, n_2\ov]_{\sF_2}, \\
\label{eq:BBQ-2}
\zeta^{(w, v)}_{\sN(v)}: \;\; &BwB/N(v) \longrightarrow \O^{(w, v)} \times T: \;\; n_1\ow n_2 t_\cdot N(v) \longmapsto ([n_1\ow, \, n_2\ov]_{\sF_2}, \; t),
\end{align}
where again $n_1 \in N_w, \, n_2 \in N_v$, and $t \in T$.  On the other hand,  using the identity
\[
\overline{w_0w^{-1}} \, N_w^- = N_{w_0w^{-1}}  \, \overline{w_0w^{-1}},
\]
we introduce the isomorphism
\begin{equation}\label{eq:zeta-w}
\zeta^w: \;\; B^-wB/B \lrw \O^{w_0w^{-1}}, \;\; m\ow_\cdot B \longmapsto {\overline{w_0w^{-1}} \, {m^{-1}}}_\cdot B,
\hs m \in N_w^-.
\end{equation}
Combining the isomorphisms $I^w_\sQ$ in \eqref{eq:IwQ} with the isomorphisms in \eqref{eq:BBQ-1}, \eqref{eq:BBQ-2}, and \eqref{eq:zeta-w}, we get
our desired decompositions
\begin{align}\label{eq:Jw-Bv-0}
J^w_{\sB(v)}:\;\; & \; wB^-B/B(v) \; \stackrel{\sim}{\longrightarrow}\; \O^{w_0w^{-1}} \times \O^{(w, v)},\\
\label{eq:Jw-Nv-0}
J^w_{\sN(v)}: \;\;& \; wB^-B/N(v) \; \stackrel{\sim}{\longrightarrow}\; \O^{w_0w^{-1}} \times \O^{(w, v)} \times T,
\end{align}
explicitly given as
\begin{align}\label{eq:Jw-Bv}
&J^w_{\sB(v)} (a\ow n_\cdot B(v)) = \left ({\overline{u} \, a_-^{-1}}_\cdot B, \;\, [a_+^\prime \ow, \; n\ov]_{\sF_2}\right), \\
\label{eq:Jw-Nv}
&J^w_{\sN(v)}(a\ow nt_\cdot N(v)) =\left({\overline{u} \, a_-^{-1}}_\cdot B, \;\, [a_+^\prime \ow, \; n\ov]_{\sF_2},\, t\right),
\end{align}
where $a \in \ow N^- \ow^{\, -1}$, decomposed as in \eqref{eq:ggg-000}, $n \in N_v$, $t \in T$, and $u = w_0w^{-1}$.

It is straightforward to prove the following $T$-equivariance of $J^w_{\sB(v)}$ and $J^w_{\sN(v)}$.

\begin{lemma}\label{le:JwQ-T}
The isomorphisms $J^w_{\sB(v)}$ and $J^w_{\sN(v)}$ are $T$-equivariant, where $t_1 \in T$ acts on $wB^-B/B(v)$ and $wB^-B/N(v)$ by
left translation by $t_1$ and on $\O^{w_0w^{-1}} \times \O^{(w, v)}$ and $\O^{w_0w^{-1}} \times \O^{(w, v)} \times T$ respectively by
\begin{align*}
&t_1 \cdot \left({n_1}\ou_\cdot B, \; [n_2\ow, \; n_3\ov]_{\sF_2}\right) = \left(t_1^{u^{-1}} {n_1}\ou_\cdot B, \; [t_1n_2\ow, \; n_3\ov]_{\sF_2}\right), \\
&t_1 \cdot \left({n_1} \ou_\cdot B, \; [n_2\ow, \; n_3\ov]_{\sF_2}, \; t\right) = \left(t_1^{u^{-1}} {n_1}\ou_\cdot B, \; [t_1n_2\ow, \; n_3\ov]_{\sF_2},\; t_1^w t\right),
\end{align*}
where $u = w_0w^{-1} \in W$,  $n_1 \in N_u, \, n_2 \in N_w, \, n_3 \in N_v$ and $t \in T$.
\end{lemma}

It is also straightforward to prove the following for the inverses of $J^w_{\sB(v)}$ and $J^w_{\sN(v)}$.

\begin{lemma}\label{le:JwQ-inverse} With $u = w_0w^{-1}$, the inverses of $J^w_{\sB(v)}$ and $J^w_{\sN(v)}$ are respectively given by
\begin{align*}
\left(J^w_{\sB(v)}\right)^{-1}\!\!\left(n_1\ou_\cdot B, \, [n_2\ow, \; n_3\ov]_{\sF_2}\right) &=
 [n_2 (\ou^{\, -1} \, n_1 \ou)]_-^{-1} n_2 \ow {n_3}_\cdot B(v)\\
\nonumber
& =[n_2 (\ou^{\, -1} \, n_1 \ou)]_+(\ou^{\, -1} n_1^{-1}\ou)\ow {n_3}_\cdot B(v),\\
\left(J^w_{\sN(v)}\right)^{-1}\!\!\left(n_1\ou_\cdot B, \, [n_2\ow, \, n_3\ov]_{\sF_2},  t\right) &=
[n_2 (\ou^{\, -1} \, n_1 \ou)]_-^{-1} n_2 \ow n_3t_\cdot N(v)\\
\nonumber
& =[n_2 (\ou^{\, -1} \, n_1 \ou)]_+(\ou^{\, -1} n_1^{-1}\ou)\ow n_3t_\cdot N(v),
\end{align*}
where $n_1 \in N_u$, $n_2 \in N_w$, $n_3 \in N_v$, and $t \in T$.
\end{lemma}

We can now use the isomorphisms $J^w_{\sQ}$ in \eqref{eq:Jw-Bv-0} and \eqref{eq:Jw-Nv-0} to introduce
Bott-Samelson coordinates on $wB^-B/Q$.
Let first $Q = B(v)$.

\begin{notation}\label{nota:br}
{\rm  Let $l_0 = l(w_0)$, and for $v, w \in W$, let $k =l(w_0w^{-1}) = l_0-l(w)$ and  $l = l_0 + l(v)$. We write an element
$\br  \in \cR(w_0w^{-1}) \times \cR(w) \times \cR(v)$ as $\br= (\bfw^0, \bfw, \bfv)$ with
\begin{equation}\label{eq:br}
\bfw^0 = (s_{\al_1}, \, \ldots, \, s_{\al_k}), \;\;\; \bfw = (s_{\al_{k+1}}, \, \ldots, s_{\al_{l_0}}), \;\;\; \bfv = (s_{\al_{l_0+1}}, \, \ldots, \, s_{\al_{l}}),
\end{equation}
where  $\al_j \in \Gamma$ for each $j \in [1, l]$.
\hfill $\diamond$
}
\end{notation}

Recall from \eqref{eq:beta-Ou} that associated to  $\bfw^0 \in \cR(w_0w^{-1})$ and $(\bfw, \bfv) \in \cR(w) \times \cR(v)$
one has the parametrizations $\bbeta^{\bfw^0}: \CC^k \rightarrow \O^{w_0w^{-1}}$ and
$\bbeta^{(\bfw, \bfv)}: \CC^{l-k} \rightarrow  \O^{(w, v)}$ given by
\begin{align*}
&\bbeta^{\bfw^0}(z_1, \ldots, z_k) = g_{\bfw^0}(z_1, \ldots, z_k)_\cdot B, \\
&\bbeta^{(\bfw, \bfv)}(z_{k+1}, \ldots, z_l) =\left([g_{\bw}(z_{k+1}, \ldots,  z_{l_0}), \; g_{\bfv}(z_{l_0+1}, \ldots, z_l)]_{\sF_2}\right).
\end{align*}
Consequently, one has the parametrization
\begin{align}\label{eq:bga-1}
&\bsigma^{\br} = \bbeta^{\bfw^0} \times \bbeta^{(\bfw, \bfv)}: \;\;\CC^l \longrightarrow \O^{w_0w^{-1}} \times \O^{(w, v)},\\
\nonumber
&\bsigma^\br(z) =  \left( g_{\bfw^0}(z_1, \ldots, z_k)_\cdot B, \; [g_{\bw}(z_{k+1}, \ldots,  z_{l_0}), \; g_{\bfv}(z_{l_0+1}, \ldots, z_l)]_{\sF_2}\right).
\end{align}
Combining with $(J^w_{\sB(v)})^{-1}: \O^{w_0w^{-1}} \times \O^{(w, v)} \to wB^-B/B(v)$,  we have the isomorphism
\begin{equation}\label{eq:bga-Bv}
\bsigma^{\br}_{\sB(v)} \stackrel{\rm def}{=} (J^w_{\sB(v)})^{-1} \circ \bsigma^{\br}: \;\; \CC^l \longrightarrow wB^-B/B(v).
\end{equation}

\begin{definition}\label{de:BS-shifted}
{\rm
For $w \in W$ and $\br = (\bfw^0, \bfw, \bfv) \in \cR(w_0w^{-1}) \times \cR(w) \times \cR(v)$, the
map $\bsigma^{\br}_{\sB(v)}: \CC^l \to wB^-B/B(v)$ in \eqref{eq:bga-Bv}
 is called the {\it Bott-Samelson parametrization} of $wB^-B/B(v)$ defined by $\br$, and the induced
coordinates $(z_1, z_2, \ldots, z_l)$
 are called {\it the Bott-Samelson coordinates} on $wB^-B/B(v)$ defined by $\br$. The collection
\[
\cA_{{\sBS}}(G/B(v)) = \{\bsigma^{\br}_{\sB(v)}: \; w \in W, \; \br \in \cR(w_0w^{-1}) \times \cR(w) \times \cR(v)\}
\]
is called the {\it Bott-Samelson atlas} on $G/B(v)$, and each $\bsigma^{\br}_{\sB(v)}$ in $\cA_{{\sBS}}(G/B(v))$ is called a {\it Bott-Samelson coordinate chart}
on $G/B(v)$.
\hfill $\diamond$
}
\end{definition}

Turning to $Q = N(v)$, fix any listing $\omega_1, \ldots, \omega_d$ of all the fundamental weights, and let
\begin{equation}\label{eq:Xi}
\sigma :\;\;  (\CC^\times)^d \longrightarrow T,
\end{equation}
be the inverse of the isomorphism $T \to (\CC^\times)^d, t \mapsto (t^{\omega_1}, \, \ldots, \, t^{\omega_d})$.
One then has the parametrization
\begin{equation}\label{eq:bga-2}
\bsigma^\br \times \sigma: \;\; \CC^l \times (\CC^\times)^d \longrightarrow \O^{w_0w^{-1}} \times \O^{(w, v)} \times T.
\end{equation}
Combining with  $(J^w_{\sN(v)})^{-1}: \O^{w_0w^{-1}} \times \O^{(w, v)} \times T \to wB^-B/N(v)$, one has the isomorphism
\begin{equation}\label{eq:bga-Nv}
\bsigma^{\br}_{\sN(v)} = (J^w_{\sN(v)})^{-1} \circ (\bsigma^{\br} \times \sigma): \;\; \CC^l \times (\CC^\times)^d \longrightarrow wB^-B/N(v).
\end{equation}

\begin{definition}\label{de:BS-shifted-Nv}
{\rm
For $w \in W$ and $\br = (\bfw^0, \bfw, \bfv) \in \cR(w_0w^{-1}) \times \cR(w) \times \cR(v)$, the
map $\bsigma^{\br}_{\sN(v)}: \CC^l \times (\CC^\times)^d \to wB^-B/N(v)$ in \eqref{eq:bga-Nv}
 is called the {\it Bott-Samelson parametrization} of $wB^-B/N(v)$ defined by $\br$, and the induced
coordinates $(z_1, z_2, \ldots, z_{l+d})$
 are called {\it the Bott-Samelson coordinates on $wB^-B/N(v)$ defined by $\br$}. The collection
\[
\cA_{{\sBS}}(G/N(v)) = \{\bsigma^{\br}_{\sN(v)}: \; w \in W, \; \br \in \cR(w_0w^{-1}) \times \cR(w) \times \cR(v)\}
\]
is called the {\it Bott-Samelson atlas} on $G/N(v)$, and each $\bsigma^{\br}_{\sN(v)}$ in $\cA_{{\sBS}}(G/N(v))$ is called a {\it Bott-Samelson coordinate
chart} on $G/N(v)$.
\hfill $\diamond$
}
\end{definition}

\begin{example}\label{ex:two-charts-G}
{\rm
Let $v = w_0$ so $G/N(v) = G$. For $w = e$, the Bott-Samelson parametrization
$\bsigma^\br_\sG: \!\CC^{2l_0} \!\times \!(\CC^\times)^d \stackrel{\sim}{\rightarrow} B^-B$ for
$\br\! =\!(\bw^0, \emptyset, \bw_0)\! \in \!\cR(w_0)\!\times \!\cR(e) \!\times\! \cR(w_0)$ is given by
\[
\bsigma^{\br}_\sG(z) =\left (\overline{w_0}^{\, -1} g_{\bw^0}(z_1, \ldots, z_{l_0})\right)^{-1}
\left(g_{\bw_0}(z_{l_0+1}, \ldots, z_{2l_0})\overline{w_0}^{\, -1} \right)
\sigma(z_{2l_0+1}, \ldots, z_{2l_0+d})
\]
for $z = (z_1, \ldots, z_{2l_0+d}) \in \CC^{2l_0} \times (\CC^\times)^d$.
Similarly, for $w = w_0$ and for each $\br =(\emptyset, \bw_0, \bw_0^\prime)\in \cR(e) \times \cR(w_0) \times \cR(w_0)$,
we have the Bott-Samelson parametrization $\bsigma^\br_\sG: \CC^{2l_0} \times (\CC^\times)^d \stackrel{\sim}{\rightarrow}
Bw_0B$ given by
\begin{align*}
\bsigma^{\br}_\sG(z_1, \ldots, z_{2l_0+d}) &=  g_{\bw_0}(z_1, \ldots, z_{l_0}) g_{\bw'_0}(z_{l_0+1}, \ldots, z_{2l_0})\overline{w_0}^{\, -1}
\sigma(z_{2l_0+1}, \ldots, z_{2l_0+d}).
\end{align*}
\hfill $\diamond$
}
\end{example}

In the remainder of $\S$\ref{ssec:ABS}, we express
the Bott-Samelson on $G/B(v)$ and $G/N(v)$ using generalized minors on $G$.

Consider again the case of $Q = B(v)$ first. Fix  $w \in W$ and let again $u = w_0w^{-1}$.
Let $\br =(\bfw^0, \bfw, \bfv) \in \cR(w_0w^{-1}) \times \cR(w) \times \cR(v)$ be as in \eqref{eq:br}, i.e.,
\[
\bfw^0 = (s_{\al_1}, \, \ldots, \, s_{\al_k}), \;\;\; \bfw = (s_{\al_{k+1}}, \, \ldots, s_{\al_{l_0}}), \;\;\; \bfv = (s_{\al_{l_0+1}}, \, \ldots, \, s_{\al_{l}}).
\]
Note then that $(s_{\al_1}, \ldots, s_{\al_k}, s_{\al_{k+1}}, \ldots, {\al_{l_0}}) \in \cR(w_0)$ and that
\begin{equation}\label{eq:i-w}
l(s_{\al_i} \cdots s_{\al_k}w) = l(s_{\al_i} \cdots s_{\al_k}) + l(w), \hs i \in [1, k].
\end{equation}
Recall from Notation \ref{nota:intro} that for $\alpha \in \Gamma$ we have $\al^* = -w_0(\al)$ and that
$\overline{s_{\al^*}} =
\overline{w_0}^{\, -1} \, \overline{s_\al}\, \overline{w_0}$.

\begin{proposition}\label{pr:zj-Bv}
For $w \in W$, write an element in $wB^-B/B(v)$ as $\ow mn_\cdot B(v)$ for unique $m \in N^-$ and $n \in N_v$. Then
for $\br \in \cR(w_0w^{-1}) \times \cR(w) \times \cR(v)$ as in \eqref{eq:br},
the Bott-Samelson coordinates $(z_1, \ldots, z_l)$ on $wB^-B/B(v)$ defined by $\br$ are given by
\begin{equation}\label{eq:zj-Bv}
z_j(\ow mn_\cdot B(v)) =
 \begin{cases} \Delta_{s_{\al_1^*} \cdots s_{\al_j^*} \omega_{\al_j^*}, \, s_{\al_1^*} \cdots s_{\al_{j-1}^*} \omega_{\al_j^*}}(m),
&  \;\;\; j \in [1, \,k],\\
\Delta_{s_{\al_{k+1}} \cdots s_{\al_{j-1}}\omega_{\al_j}, \; s_{\al_{k+1}} \cdots s_{\al_j}\omega_{\al_j}}(\ow \, m\, \ow^{\, -1}), &  \;\;\; j \in [k+1, \, l_0],\\
\Delta_{s_{\al_{l_0+1}} \cdots s_{\al_{j-1}}\omega_{\al_j}, \; s_{\al_{l_0+1}} \cdots s_{\al_j}\omega_{\al_j}}(n), &  \;\;\; j \in [l_0+1, \, l].\end{cases}
\end{equation}
Furthermore, with respect to the $T$-action on $\CC[wB^-B/B(v)]$ induced from the $T$-action on $wB^-B/B(v)$ by left translation, each $z_j$ is a $T$-weight vector, with
\[
t\cdot z_j = \begin{cases} t^{s_{\al_k} s_{\al_{k-1}} \cdots s_{\al_j}(\al_j)} z_j, & \;\;\; j \in [1, k],\\
t^{s_{\al_{k+1}} s_{\al_2}\cdots s_{\al_{j-1}}(\al_j)} z_j, & \;\;\; j \in [k+1, l],\end{cases} \hs t \in T.
\]
\end{proposition}

\begin{proof}
Fix $m \in N^-$ and $n \in N_v$ and let $a  = \ow m \ow^{\, -1} \in \ow N^- \ow^{\, -1}$.
Decompose $a$ again as in \eqref{eq:ggg-000}, i.e.,
$a = a_+ a_- = a^\prime_- a^\prime_+$, where  $a_+, \,a^\prime_+ \in N_w$ and $a_-, \,a^\prime_- \in N_w^-$.
Let $n_1 = \ou a_-^{-1} \ou^{\, -1} \in N_u$ and $n_2 = a_+^\prime \in N_w$.
By the definition of $J^w_{\sB(v)}$,
\[
J^w_{\sB(v)}(\ow mn_\cdot B(v)) = J^w_{\sB(v)}(a\ow {n}_\cdot B(v)) = \left(n_1 \ou_\cdot B, \; [n_2\ow, \, n\ov]_{\sF_2}\right).
\]
Write $z_j = z_j(\ow mn_\cdot B(v))$ for $j \in [1, l]$.
Let first $j \in [1, k]$. By Lemma \ref{le:zj-n},
\begin{align*}
z_j &= \Delta^{\omega_{\al_j}}(\overline{s_{\al_1} \cdots s_{\al_{j-1}}}^{\, -1} n_1 \overline{s_{\al_1} \cdots s_{\al_j}})
= \Delta^{\omega_{\al_j}}(\overline{s_{\al_1} \cdots s_{\al_{j-1}}}^{\, -1} (\ou a_-^{-1} \ou^{\, -1}) \overline{s_{\al_1} \cdots s_{\al_j}})\\
&=\Delta^{\omega_{\al_j}}(\overline{s_{\al_j} \cdots s_{\al_{a}}} \,a_-^{-1} \,\overline{s_{\al_{j+1}} \cdots s_{\al_k}}^{\, -1}).
\end{align*}
By \eqref{eq:i-w} and \eqref{eq:Nu-Nu-2}, $\overline{s_{\al_{j+1}} \cdots s_{\al_k}} \, a_+ \, \overline{s_{\al_{j+1}} \cdots s_{\al_k}}^{\, -1} \in N$. Thus
\begin{align*}
z_j & = \Delta^{\omega_{\al_j}} (\overline{s_{\al_{j}} \cdots s_{\al_k}} \, a_-^{-1} a_+^{-1} \, \overline{s_{\al_{j+1}} \cdots s_{\al_k}}^{\, -1} \;
 \overline{s_{\al_{j+1}} \cdots s_{\al_k}} \, a_+ \, \overline{s_{\al_{j+1}} \cdots s_{\al_k}}^{\, -1})\\
& = \Delta^{\omega_{\al_j}} (\overline{s_{\al_{j}} \cdots s_{\al_k}} \, a^{-1} \, \overline{s_{\al_{j+1}} \cdots s_{\al_k}}^{\, -1})
= \Delta^{\omega_{\al_j}} (\overline{s_{\al_{j}} \cdots s_{\al_k} w}\, m^{-1}  \, \overline{s_{\al_{j+1}} \cdots s_{\al_k}w}^{\, -1}),
\end{align*}
where in the last step we use again \eqref{eq:i-w} for $i = j$ and $j+1$.
 By \eqref{eq:Delta-omega-ast}, one has
\[
z_j = \Delta^{\omega_{\al_j^*}}(\overline{w_0}^{\, -1}\, \overline{s_{\al_{j+1}} \cdots s_{\al_k}w}\, m \,
\overline{s_{\al_j} \cdots s_{\al_k}w}^{\, -1}\,\overline{w_0}).
\]
Note now that for $i = j$ or $j+1$, one has  $\overline{s_{\al_i} \cdots s_{\al_k}w}= \overline{s_{\al_1} \cdots s_{\al_{i-1}}}^{\, -1} \, \overline{w_0}$. It follows that
\begin{align*}
z_j &=\Delta^{\omega_{\al_j^*}}(\overline{w_0}^{\, -1}\, \overline{s_{\al_1} \cdots s_{\al_j}}^{\, -1} \, \overline{w_0} \, m \,
\overline{w_0}^{\, -1}\, \overline{s_{\al_1} \cdots s_{\al_{j-1}}} \, \overline{w_0})\\
& = \Delta^{\omega_{\al_j^*}}(\overline{s_{\al_1^*} \cdots s_{\al_j^*}}^{\, -1} \, m \, \overline{s_{\al_1^*} \cdots s_{\al_{j-1}^*}})=
\Delta_{s_{\al_1^*} \cdots s_{\al_j^*} \omega_{\al_j^*}, \; s_{\al_1^*} \cdots s_{\al_{j-1}^*} \omega_{\al_j^*}}(m).
\end{align*}
Assume now that $j \in [k+1, \, l_0]$. By Lemma \ref{le:zj-n},
\begin{align*}
z_j &= \Delta^{\omega_{\al_j}}(\overline{s_{\al_{k+1}} \cdots s_{\al_{j-1}}}^{\, -1}\, n_2\, \overline{s_{\al_{k+1}} \cdots s_{\al_j}})
=  \Delta^{\omega_{\al_j}}(\overline{s_{\al_{k+1}} \cdots s_{\al_{j-1}}}^{\, -1} \,a_+^\prime\, \overline{s_{\al_{k+1}} \cdots s_{\al_j}})\\
&= \Delta^{\omega_{\al_j}}(\overline{s_{\al_{k+1}} \cdots s_{\al_{j-1}}}^{\, -1}\, (a_-^\prime)^{-1}a\, \overline{s_{\al_{k+1}} \cdots s_{\al_j}}).
\end{align*}
By \eqref{eq:Nu-Nu-2}, $\overline{s_{\al_{k+1}} \cdots s_{\al_{j-1}}}^{\, -1}\, (a_-^\prime)^{-1}\, \overline{s_{\al_{k+1}} \cdots s_{\al_{j-1}}}\subset N^-$.
It follows that
\begin{align*}
z_j &= \Delta^{\omega_{\al_j}}(\overline{s_{\al_{k+1}} \cdots s_{\al_{j-1}}}^{\, -1}\, a\, \overline{s_{\al_{k+1}} \cdots s_{\al_j}}) =
\Delta_{s_{\al_{k+1}} \cdots s_{\al_{j-1}}\omega_{\al_j}, \; s_{\al_{k+1}} \cdots s_{\al_j}\omega_{\al_j}}(a)\\
& = \Delta_{s_{\al_{k+1}} \cdots s_{\al_{j-1}}\omega_{\al_j}, \; s_{\al_{k+1}} \cdots s_{\al_j}\omega_{\al_j}}(\ow \, m \, \ow^{\, -1}).
\end{align*}
Finally, assume that $j \in [l_0+1, \, l]$. By Lemma \ref{le:zj-n},
\[
z_j = \Delta^{\omega_{\al_j}}(\overline{s_{\al_{l_0+1}} \cdots s_{\al_{j-1}}}^{\, -1}\, n\, \overline{s_{\al_{l_0+1}} \cdots s_{\al_j}})
=\Delta_{s_{\al_{l_0+1}} \cdots s_{\al_{j-1}}\omega_{\al_j}, \; s_{\al_{l_0+1}} \cdots s_{\al_j}\omega_{\al_j}}(n).
\]
The statement on the $T$-weight for each $z_j$ follows either from a direct calculation using \eqref{eq:zj-Bv} or from the
$T$-equivariance of the isomorphism $J^w_{\sB(v)}$ and Lemma \ref{le:zj-Ou}.
\end{proof}

Turning to $Q = N(v)$, recall that we have fixed a listing $\omega_1, \ldots \omega_d$ of all the fundamental weights to define the
isomorphism $\sigma: (\CC^\times)^d \to T$ in \eqref{eq:Xi}, which is in turn used in defining the Bott-Samelson coordinate charts on $G/N(v)$.
The following proposition follows directly from Proposition \ref{pr:zj-Bv}.

\begin{proposition}\label{pr:zj-Nv}
For $w \in W$, write an element in $wB^-B/N(v)$ as $\ow mnt_\cdot N(v)$ for unique $m \in N^-$, $m \in N_v$, and $t \in T$. Then
for $\br \in \cR(w_0w^{-1}) \times \cR(w) \times \cR(v)$ as in \eqref{eq:br}, the Bott-Samelson coordinates $(z_1, \ldots, z_{l+d})$ on $wB^-B/N(v)$
defined by $\br$ are given by
\[
z_j(\ow mnt_\cdot N(v)) =
 \begin{cases} \Delta_{s_{\al_1^*} \cdots s_{\al_j^*} \omega_{\al_j^*}, \, s_{\al_1^*} \cdots s_{\al_{j-1}^*} \omega_{\al_j^*}}(m),
& \; j \in [1, \,k],\\
\Delta_{s_{\al_{k+1}} \cdots s_{\al_{j-1}}\omega_{\al_j}, \; s_{\al_{k+1}} \cdots s_{\al_j}\omega_{\al_j}}(\ow \, m\, \ow^{\, -1}), & \; j \in [k+1, \, l_0],\\
\Delta_{s_{\al_{l_0+1}} \cdots s_{\al_{j-1}}\omega_{\al_j}, \; s_{\al_{l_0+1}} \cdots s_{\al_j}\omega_{\al_j}}(n), & \;j \in [l_0+1, \, l],\\
t^{\omega_{j-l}}, & \; j \in [l+1, \, l+d].\end{cases}
\]
Furthermore, with respect to the $T$-action on $\CC[wB^-B/N(v)]$ induced from the $T$-action on $wB^-B/N(v)$ by left translation, each $z_j$ is a $T$-weight vector, with
\[
t\cdot z_j = \begin{cases} t^{s_{\al_k} s_{\al_{k-1}} \cdots s_{\al_j}(\al_j)} z_j, & \;\;\; j \in [1, k],\\
t^{s_{\al_{k+1}} s_{\al_2}\cdots s_{\al_{j-1}}(\al_j)} z_j, & \;\;\; j \in [k+1, l],\\
t^{w\omega_{j-l}}z_j, & \hs j \in [l+1, l+d],\end{cases}  \;\;\; t \in T.
\]
\end{proposition}

\begin{example}\label{ex:B}
{\rm
Consider the case of $Q = B$ so that $v = e$.  For any $w \in W$ and $\br = (\bw^0, \bw, \emptyset) \in \cR(w_0w^{-1}) \times \cR(w) \times \cR(e)$ with
\[
\bw^0 = (s_{\al_1}, \ldots, s_{\al_k}) \hs \mbox{and} \hs \bw = (s_{\al_{k+1}}, \ldots, s_{\al_{l_0}}),
\]
the Bott-Samelson coordinates on $wB^-B/B$ defined by $\br$ are given by
\begin{equation}\label{eq:zj-B}
z_j(\ow m_\cdot B) =
 \begin{cases} \Delta_{s_{\al_1^*} \cdots s_{\al_j^*} \omega_{\al_j^*}, \, s_{\al_1^*} \cdots s_{\al_{j-1}^*} \omega_{\al_j^*}}(m),
&  \;\;\; j \in [1, \,k],\\
\Delta_{s_{\al_{k+1}} \cdots s_{\al_{j-1}}\omega_{\al_j}, \; s_{\al_{k+1}} \cdots s_{\al_j}\omega_{\al_j}}(\ow \, m\, \ow^{\, -1}), &  \;\;\; j \in [k+1, \, l_0].
\end{cases}
\end{equation}
\hfill $\diamond$
}
\end{example}

\begin{example}\label{ex:SL3-0}
{\rm
Consider $G = SL(3, \CC)$ and $Q = \{e\}$ (so $v = w_0$), with the standard choices of $B$ and $B^-$ being the respectively the subgroups consisting of
upper-triangular and lower triangular elements, and denote $s_1=s_{\al_1}$ and $s_2=s_{\al_2}$ for the standard choice of $\al_1$ and $\al_2$.
There are a total of $16$ Bott-Samelson coordinate charts on $G$, corresponding to the $16$ elements in the  set
$\bigcup_{w \in W} \cR(w_0w^{-1}) \times \cR(w) \times \cR(w_0)$.   Write  $g \in SL(3)$ as
$\displaystyle g = \left(\begin{array}{ccc} a_{11} & a_{12} & a_{13}\\ a_{21} & a_{22} & a_{23}\\ a_{31} & a_{32} & a_{33}\end{array}\right)$ and
let $\displaystyle \Delta_{ij, kl} =
{\rm det}\left(\begin{array}{cc} a_{ik} & a_{il}\\ a_{jk} & a_{jl}\end{array}\right)$ for $i<j$ and $k < l$.

As the first example, let $w = e$ and choose $\br_1 = ((s_1,s_2,s_1), \emptyset, (s_1, s_2, s_1))$. Then for the  corresponding Bott-Samelson parametrization
$\bsigma^{\br_1}=\bsigma^{\br_1}_{\sG}$ of $B^-B$, one maps
$\xi =(\xi_1, \ldots, \xi_8) \in \CC^6 \times (\CC^\times)^2$ to
\begin{align*}
\bsigma^{\br_1}(\xi) &= \left(\begin{array}{ccc} 1 & 0 & 0\\ \xi_3 & 1 & 0 \\ \xi_2 & \xi_1 & 1\end{array}\right)
\left(\begin{array}{ccc} 1 & \xi_4 & \xi_4\xi_6-\xi_5\\ 0 & 1 & \xi_6\\ 0 & 0 & 1\end{array}\right)
\left(\begin{array}{ccc} \xi_7 & 0 & \\ 0 & \xi_8/\xi_7 & 0\\ 0 & 0 & 1/\xi_8\end{array}\right)\\
&=\left(\begin{array}{ccc} \xi_7 & \xi_4\xi_8/\xi_7 & (\xi_4\xi_6-\xi_5)/\xi_8\\
\xi_3\xi_7 & (\xi_3\xi_4+1)\xi_8/\xi_7 & (\xi_3\xi_4\xi_6-\xi_3\xi_5+\xi_6)/\xi_8\\
\xi_2\xi_7 & (\xi_2\xi_4+\xi_1)\xi_8/\xi_7 & (\xi_2\xi_4\xi_6-\xi_2\xi_5+\xi_1\xi_6+1)/\xi_8\end{array}\right),
\end{align*}
and the corresponding Bott-Samelson coordinates $(\xi_1, \ldots, \xi_8)$ on $B^-B$ are
\begin{align*}
&\xi_1 = \frac{\Delta_{13,12}}{\Delta_{12,12}}, \hs \xi_2 = \frac{a_{31}}{a_{11}}, \hs \xi_3 = \frac{a_{21}}{a_{11}}, \hs \xi_4 = \frac{a_{11}a_{12}}{\Delta_{12,12}},\hs
 \xi_5 = a_{11} \Delta_{12,23}, \\
&\xi_6 = \frac{\Delta_{12,12} \Delta_{12,13}}{a_{11}}, \hs \xi_7 = a_{11}, \hs \xi_8 = \Delta_{12,12}.
\end{align*}

As the second example, let $w = w_0 = s_1s_2s_1$ and $\br_2 = (\emptyset, (s_2, s_1, s_2), (s_1, s_2, s_1))$.
The  corresponding Bott-Samelson parametrization $\bsigma^{\br_2} =\bsigma^{\br_2}_{\sG}$ of $w_0B^-B=Bw_0B$ maps
$z =(z_1, \ldots, z_8) \in \CC^6 \times (\CC^\times)^2$ to
\begin{align*}
\bsigma^{\br_2}(z) &= \overline{w_0}  \left(\begin{array}{ccc} 1 & 0 & 0 \\ -z_1 & 1 & 0\\ z_2 & -z_3 & 1\end{array}\right)
\left(\begin{array}{ccc} 1 & z_4 & z_4z_6-z_5\\ 0 & 1 & z_6\\ 0 & 0 & 1\end{array}\right)
\left(\begin{array}{ccc} z_7 & 0 & \\ 0 & z_8/z_7 & 0\\ 0 & 0 & 1/z_8\end{array}\right)\\
&=\left(\begin{array}{ccc} z_2z_7 & (z_2z_4-z_3)z_8/z_7 & (z_2z_4z_6-z_2z_5-z_3z_6+1)/z_8\\
z_1z_7 & (z_1z_4-1)z_8/z_7 & (z_1z_4z_6-z_1z_5-z_6)/z_8\\
z_7 &z_4z_8/z_7 & (z_4z_6-z_5)/z_8\end{array}\right),
\end{align*}
and the Bott-Samelson coordinates $(z_1, \ldots, z_8)$ on $w_0B^-B$ are
\begin{align*}
&z_1 = \frac{a_{21}}{a_{31}}, \hs z_2 =  \frac{a_{11}}{a_{31}}, \hs z_3= \frac{\Delta_{13,12}}{\Delta_{23,12}}, \hs
z_4 = \frac{a_{31}a_{32}}{\Delta_{23,12}},\\
&z_5 =a_{31}\Delta_{23,23}, \hs
z_6 = \frac{1}{a_{31}} \Delta_{23,12}\Delta_{23,13},
 \hhs z_7 = a_{31}, \hs z_8=\Delta_{23,12}.
\end{align*}

The changes between the  coordinates $(\xi_1, \ldots, \xi_8)$ and $(z_1, \ldots, z_8)$ are given by
\begin{align*}
&\xi_1 = z_3/(z_1z_3-z_2), \hs \xi_2 = 1/z_2, \hs \xi_3 = z_1/z_2, \\
&\xi_4 = \frac{z_2(z_2z_4-z_3)}{z_1z_3-z_2}, \hs
\xi_5 = z_2(1-z_1z_4 + (z_1z_3-z_2)z_5),\\
&\xi_6 = \frac{1}{z_2}(z_1z_3-z_2)(z_1z_3z_6-z_2z_6-z_1), \hs
\xi_7 = z_2z_7, \hs \xi_8 = z_8(z_1z_3-z_2),\\
&z_1=\xi_3/\xi_2, \hs z_2=1/\xi_2, \hs
z_3 = \xi_1/(\xi_1\xi_3-\xi_2),\\
& z_4 = \frac{\xi_2(\xi_2\xi_4+\xi_1)}{\xi_1\xi_3-\xi_2},\hs z_5 = \xi_2(\xi_1\xi_3\xi_5-\xi_2\xi_5+\xi_3\xi_4+1),\\
&z_6 = \frac{1}{\xi_2}(\xi_1\xi_3-\xi_2)(\xi_1\xi_3\xi_6-\xi_2\xi_6+\xi_3), \hs z_7 = \xi_2\xi_7, \hs z_8 = \xi_8(\xi_1\xi_3-\xi_2).
\end{align*}
\hfill $\diamond$
}
\end{example}

\section{Positivity of the Bott-Samelson coordinates}\label{sec:BS-pos}
In $\S$\ref{ssec:pos} we recall from \cite{ABHY1, BK:CrystalsII, FG:IHES} the notion of positive structures on complex varieties.
In $\S$\ref{ssec:pos-GQ} we first recall  the
Lusztig positive structure on $G$ and then extend it to $G/Q$ for $Q = B(v)$ or $N(v)$ for all $v \in W$.
Some results from \cite{BZ:tensor, FZ:double} on generalized minors and double Bruhat cells are reviewed in $\S$\ref{ssec:auxi},
which are then used in $\S$\ref{ssec:BSpos-GQ} to prove
positivity of all Bott-Samelson coordinates  on $G/Q$ with respect to
the Lusztig positive structure.

\subsection{Positive varieties}\label{ssec:pos}
We first recall the notion of positive varieties.

\begin{notation}\label{nota:Pos}
{\rm
1) For an integer $m \geq 1$, let $\Poly_m^{>0}$ be the set of all non-zero polynomials in $m$ variables with non-negative integer coefficients.
Elements in $\Poly_m^{>0}$ will also be called {\it positive integral polynomials in $m$ variables}.

2) For an irreducible rational complex variety $X$ with the field $\CC(X)$ of rational functions, and for
a subset $\bphi = \{\phi_1, \ldots, \phi_m\}$ of  algebraically independent elements in $\CC(X)$,
denote by $\Pos(\bphi)$ the set of elements $f \in \CC(X)$ such that $f = p(\phi_1, \ldots, \phi_m)/q(\phi_1, \ldots, \phi_m)$ for some $p, q \in \Poly_m^{>0}$.
Elements in $\Pos(\bphi)$ are also said to have {\it subtraction-free expressions in $\bphi$}.

3) A rational map $F$ from $(\CC^\times)^k$ to $(\CC^\times)^k$ is said to be positive if the components of $F$
are in $\Pos(c_1, \ldots, c_k)$, where $(c_1, \ldots, c_k)$ are the coordinates on $(\CC^\times)^k$.
\hfill $\diamond$
}
\end{notation}

\begin{definition}\label{de:equi-charts} \cite{BK:CrystalsII, FG:IHES}
{\rm Let $X$ be an $n$-dimensional irreducible rational complex variety.

1) A {\it toric chart} in $X$ is an open embedding
$\rho: (\CC^\times)^n \to X$. Two toric charts
\[
\rho_1: \;\;(\CC^\times)^n \, {\longrightarrow}\,  X\hs \mbox{and} \hs
\rho_2: \;\; (\CC^\times)^n\, {\longrightarrow}\, X
\]
are said to be {\it positively equivalent} if both $\rho_2^{-1} \circ \rho_1$ and $\rho_1^{-1} \circ \rho_2$ are positive rational maps from $(\CC^\times)^n$
to $(\CC^\times)^n$.
The collection of all toric charts positively equivalent
to a given toric chart $\rho$ is called the {\it positive equivalence class} of $\rho$ and is denoted as $[\rho]$.

2) A {\it positive structure} on
$X$ is a positive equivalence class $\cP^\sX$ of toric charts in $X$.
\hfill $\diamond$
}
\end{definition}

\begin{definition}\label{de:posvar} \cite{BK:CrystalsII, FG:IHES}
{\rm 1)  A {\it positive variety} is a pair $(X, \cP^\sX)$,
where $X$ is an irreducible
rational complex variety and $\cP^\sX$ a positive structure on $X$.  A
toric chart $\rho$ in $X$ such that $[\rho] = \cP^\sX$ is also called a {\it toric chart in $\cP^\sX$}.

2) Given a positive variety $(X, \cP^\sX)$ and a toric chart $\rho: (\CC^\times)^n \to X$ in $\cP^\sX$,
 the subset
\[
(X, \cP^\sX)_{>0} \, \stackrel{\rm def}{=}\, \rho((\RR_{>0})^n)
\]
of $X$, which is independent of the choice of  $\rho$ in $\cP^\sX$, is called the {\it totally positive part}
of $(X, \cP^\sX)$. Here $\RR_{>0}$ is the set of all positive real numbers.
When the positive structure $\cP^\sX$ is clearly indicated, we denote $(X, \cP^\sX)_{>0}$ simply
by $X_{>0}$ and call
it the totally positive part of $X$.

3) Given a positive variety $(X, \cP^\sX)$,  a rational function $f \in \CC(X)$ is said to be {\it positive with respect to $\cP^\sX$} if there exists
a toric chart $\rho: (\CC^\times)^n \to X$ in $\cP^\sX$ such that $f \in \Pos(c_1, \ldots, c_n)$, where $(c_1, \ldots, c_n)$
are the local coordinates on $X$ defined by the toric chart $\rho$.
Note that the definition of $f$ being positive is independent on the choice of the toric chart $\rho$ in $\cP^\sX$.  Denote by
$\Pos(X, \cP^\sX)$, or simply by $\Pos(X)$ when the positive structure $\cP^\sX$ is clearly understood,
 the set of all rational functions on $X$ that are positive with respect to $\cP^\sX$.
\hfill $\diamond$
}
\end{definition}

\begin{remark}\label{re:Pos-X}
{\rm
It is clear that for any integer $m\geq 1$, the set of $\Poly_m^{>0}$ of positive integral polynomials in $m$ variables is closed under addition and
multiplication. Consequently, for any positive variety $(X, \cP^{\sX})$, $\Pos(X, \cP^\sX)$ is a semi-field in the sense that it is closed under addition, multiplication, and division but with no zero element.
In particular, for any finite subset $\{f_1, \ldots, f_m\}$ of $\Pos(X, \cP^\sX)$, not necessarily algebraically independent,
one has $p(f_1, \ldots, f_m) \in \Pos(X, \cP^\sX)$ for any
$p \in \Poly_m^{>0}$.
\hfill $\diamond$
}
\end{remark}

\subsection{The Lusztig positive structure on $G/Q$}\label{ssec:pos-GQ}
Returning to the connected and simply connected
complex semisimple Lie group $G$, we first recall the Lusztig positive structure on $G$ and then extend it to $G/B(v)$ and
$G/N(v)$ for all $v \in W$.

For $w \in W$, $\bw = (s_{\al_1}, \ldots, s_{\al_k}) \in \cR(w)$, and $c = (c_1, \ldots, c_k) \in \CC^k$,
 define
\begin{equation}\label{eq:gcw0}
x_\bw^-(c) = x_{-\al_1}(c_1) x_{-\al_2}(c_2) \cdots x_{-\al_k}(c_k), \hs
x_\bw^+(c) = x_{\al_1}(c_1) x_{\al_2}(c_2) \cdots x_{\al_k}(c_k).
\end{equation}
As $\bw$ is a reduced word for $w$, one has  (see, for example, \cite[(d) of Proposition 2.7]{Lusz:pos})
\begin{equation}\label{ex:xwc}
x_\bw^-((\CC^\times)^k) \subset N^- \cap BwB \hs \mbox{and} \hs
x_\bw^+((\CC^\times)^k) \subset N \cap B^-wB^-.
\end{equation}
The following facts are proved in \cite[Proposition 2.7]{Lusz:pos} and \cite[Theorem 3.1]{BZ:total}).

\begin{lemma}\label{le:ccc-N}
For any $w \in W$ and $\bw = (s_{\al_1}, \ldots, s_{\al_k}) \in \cR(w)$,
\begin{align}\label{eq:ccc-N}
x_\bw^-: &\; (\CC^\times)^k \longrightarrow N^- \cap BwB, \;\; (c_1, \ldots, c_k) \longmapsto x_\bw^-(c_1, \ldots, c_k),\\
\label{eq:ccc-Nm}
x_\bw^+:& \; (\CC^\times)^k \longrightarrow N \cap B^-wB^-, \;\; (c_1, \ldots, c_k) \longmapsto x_\bw^+(c_1, \ldots, c_k),
\end{align}
are toric charts, respectively in $N^- \cap BwB$ and $N \cap B^-wB^-$, and their positive equivalence classes are independent of the choice of
$\bw \in \cR(w)$.
\end{lemma}

Let again $\sigma: (\CC^\times)^d \to T$ be the isomorphism in \eqref{eq:Xi}. For $\bw_0, \bw_0^\prime \in \cR(w_0)$,
define $\rho_{(\bw_0, \bw_0^\prime)}:\;  (\CC^\times)^{2l_0+d} \longrightarrow G$ by
\begin{equation}\label{eq:rho-ww}
\rho_{(\bw_0, \bw_0^\prime)}(c_1, \ldots, c_{2l_0+d})
= x_{\bw_0}^-(c_1, \ldots, c_{l_0}) x_{\bw_0^\prime}^+(c_{l_0+1}, \ldots, c_{2l_0})\sigma(c_{2l_0+1}, \ldots, c_{2l_0+d}).
\end{equation}
By the unique decomposition $B^-B = N^- N T$ and by Lemma \ref{le:ccc-N},
$\rho_{(\bw_0, \bw_0^\prime)}$ is
a toric chart in $G$,
and the positive equivalence class $[\rho_{(\bw_0, \bw_0^\prime)}]$ of toric charts in $G$ is independent of the choices of
$(\bw_0, \bw_0^\prime) \in \cR(w_0)\times \cR(w_0)$.

\begin{definition}\label{de:posG}
{\rm
 We will set
\[
\cP^{\sG}_{\rm Lusztg} = [\rho_{(\bw_0, \bw_0^\prime)}]
\]
for any $(\bw_0, \bw_0^\prime) \in \cR(w_0)\times \cR(w_0)$ and call it the {\it Lusztig positive structure on $G$}.
\hfill $\diamond$
}
\end{definition}

Denote by $G_{>0}$ the totally positive part of $G$ defined by $\cP^{\sG}_{\rm Lusztg}$. Then, using any reduced words
$\bw_0, \bw_0^\prime \in \cR(w_0)$, one has
\[
G_{>0} = \{x_{\bw_0}^-(c_1, \ldots, c_{l_0})x_{\bw_0^\prime}^+(c_{l_0+1}, \ldots, c_{2l_0}) t: \; c_1, \ldots, c_{2l_0} \in \RR_{>0},
t^{\omega_{\al}} > 0\ , \forall \, \al \in \Gamma\},
\]
which coincides with the totally positive part of $G$ defined by Lusztig \cite{Lusz:pos}.

To define the Lusztig positive structures on $G/B(v)$ and $G/N(v)$ for $v \in W$, we first prove the following lemma.

\begin{lemma}\label{le:nn-v}
For any $v \in W$, the following maps are all open embeddings:
\begin{align*}
&\delta_v: \;\; B^-v^{-1}B^- \longrightarrow G/N(v), \;\; g \longmapsto g_\cdot N(v),\\
&\epsilon_v: \;\; (N^- \cap Bw_0B) \times (N \cap B^-v^{-1}B^-) \times T\longrightarrow G/N(v), \;\; (m, n, t) \longmapsto mnt_\cdot N(v),\\
&\epsilon_v^\prime: \;\; (N^- \cap Bw_0B) \times (N \cap B^-v^{-1}B^-) \longrightarrow G/B(v), \;\; (m, n) \longmapsto mn_\cdot B(v).
\end{align*}
\end{lemma}

\begin{proof} Consider the Zariski open subset $B^-B\ov^{\, -1}$ of $G$ and its unique decomposition
\[
B^-B\ov^{-1} = B^-\ov^{\, -1} (\ov N \ov^{\, -1}).
\]
Recall that $\ov N \ov^{\, -1}$ has the unique decomposition
\[
\ov N \ov^{\, -1} = (N^- \cap \ov N \ov^{\, -1})(N \cap \ov N \ov^{\, -1}) = (N^- \cap \ov N \ov^{\, -1})N(v),
\]
and note that $B^-\ov^{\, -1} (N^- \cap \ov N \ov^{\, -1})= B^-v^{-1}B^-$. One then has the unique decomposition
$B^-B\ov^{\, -1} = (B^-v^{-1}B^-) N(v)$, from which it follows that
\begin{equation}\label{eq:BBN}
B^-v^{-1}B^- \times N(v) \longrightarrow B^-B\ov^{\, -1}, \;\; (g, n) \longmapsto gn, \;\;\; g \in B^-v^{-1}B^- , \, n \in N(v),
\end{equation}
is an isomorphism.  Thus $\delta_v$ is an open embedding.

As both $(N^-\cap Bw_0B) \times (N \cap B^-v^{-1}B^-) \times T$ and $G/N(v)$ are smooth,  irreducible, and of the same dimension,
by the Grothendieck-Zariski factorization theorem \cite[Theorem 8.12.6]{EGAIV},
to show that $\epsilon_v$ is an open embedding, it is enough to show that it is injective  (see also proof of \cite[Theorem 1.2]{FZ:double}).

Suppose that $m, m' \in N^- \cap Bw_0B$, $n, n' \in N \cap B^-v^{-1}B^-$, and $t, t' \in T$  are such that $mnt_\cdot N(v) = m'n't'_\cdot N(v)$.
Since $mnt, m'n't' \in B^-v^{-1}B^-$, the injectivity of $\delta_v$ implies that
$mnt = m'n't'$, and thus $m=m', n=n'$ and $t = t'$. This shows that $\epsilon_v$ is injective and thus an embedding.

Similarly one shows that $\epsilon_v^\prime$ is an embedding.
\end{proof}

Let now $v \in W$ and $l =l_0+l(v) = \dim G/B(v)$, where recall that $l_0 = l(w_0)$.
For any
\[
(\bw_0, \bfv)  = (s_{\al_1}, \, \ldots, s_{\al_{l_0}}, \, s_{l_0+1}, \, \ldots, s_{\al_l}) \in \cR(w_0) \times \cR(v),
\]
and for $c = (c_1, \ldots, c_{l+d}) \in (\CC^\times)^{l+d}$, write
 \[
c_{(1)} = (c_1, \ldots, c_{l_0}) \in (\CC^\times)^{l_0}, \hs
c_{(2)} = (c_{l_0+1}, \ldots, c_{l}) \in (\CC^\times)^{l(v)},
\]
 and $c_{(3)} =(c_{l+1}, \ldots, c_{l+d}) \in (\CC^\times)^{d}$, and recall that
\begin{align}\label{eq:gcw1}
&x_{\bw_0}^-(c_{(1)}) = x_{-\al_1}(c_1) x_{-\al_2}(c_2)\cdots x_{-\al_{l_0}}(c_{l_0}) \in N^- \cap Bw_0B,\\
\label{eq:gcw2}
&x_{\bfv^{-1}}^+(c_{(2)}) = x_{\al_{l}}(c_{l}) x_{\al_{l-1}}(c_{l-1})  \cdots x_{\al_{l_0+1}}(c_{l_0+1}) \in N \cap B^-v^{-1} B^-,
\end{align}
and $\sigma(c_{(3)}) \in T$, where $\bfv^{-1} = (s_{\al_l}, \ldots, s_{\al_{l_0}}) \in \cR(v^{-1})$.
Introduce
\begin{align}\label{eq:rho-c-Bv}
&\rho_{(\bw_0, \bfv)}^{\sG/\sB(v)}:\;\; (\CC^\times)^l \longrightarrow G/B(v), \;\; c \longmapsto
x_{\bw_0}^-(c_{(1)})x_{\bfv^{-1}}^+(c_{(2)})_\cdot B(v),\\
\label{eq:rho-c-Nv}
&\rho_{(\bw_0, \bfv)}^{\sG/\sN(v)}:\;\; (\CC^\times)^{l+d}\longrightarrow G/N(v), \;\; c \longmapsto
x_{\bw_0}^-(c_{(1)})x_{\bfv^{-1}}^+(c_{(2)}) \sigma(c_{(3)})_\cdot N(v).
\end{align}

\begin{lemma}\label{le:c-chart}
For $Q = B(v)$ or $N(v)$, and for any $(\bw_0, \bfv) \in \cR(w_0) \times \cR(v)$,
$\rho_{(\bw_0, \bfv)}^{\sG/\sQ}$ is a toric chart in $G/Q$, and the positive structure it defines on $G/Q$ is independent of the choice of
$(\bw_0, \bfv)$.
\end{lemma}

\begin{proof}
By Lemma \ref{le:ccc-N} and
Lemma \ref{le:nn-v},
${\rho}_{(\bw_0, \bfv)}^{\sG/\sQ}$ is an open embedding and thus a toric chart in $G/Q$.

For any other choice $(\bw_0^\prime, \bfv^\prime) \in \cR(w_0) \times \cR(v)$, it again follows from Lemma \ref{le:nn-v} that
\[
{\rho}_{(\bw_0, \bfv)}^{\sG/\sB(v)}\left(c_{(1)}, c_{(2)}\right) = {\rho}_{(\bw_0^\prime, \bfv^\prime)}^{\sG/\sB(v)}\left(c_{(1)}^\prime, c_{(2)}^\prime\right)
\]
implies that $x_{\bw_0}^-(c_{(1)}) = x_{\bw_0^\prime}^-(c_{(1)}^\prime)$ and $x_{\bfv^{-1}}^+(c_{(2)})=x_{(\bfv^\prime)^{-1}}^+(c_{(2)}^\prime)$.
By Lemma \ref{le:ccc-N} again, $\Pos(c_{(1)}) = \Pos(c_{(1)}^\prime)$ and $\Pos(c_{(2)}) = \Pos(c_{(2)}^\prime)$.
Thus the two toric charts ${\rho}_{(\bw_0, \bfv)}^{\sG/\sB(v)}$ and ${\rho}_{(\bw_0^\prime, \bfv^\prime)}^{\sG/\sB(v)}$ on $G/B(v)$ are
positively equivalent. The case for $Q = N(v)$ is proved similarly.
\end{proof}

\begin{definition}\label{de:Luszpos-GQ}
{\rm
For $v \in W$ and $Q = B(v)$ or $N(v)$, define
\[
\cP^{\sG/\sQ}_{\rm Lusztig} =\left[{\rho}_{(\bw_0, \bfv)}^{\sG/\sQ}\right]
\]
 for any
$(\bw_0, \bfv)\in \cR(w_0) \times \cR(v)$ and call it
the {\it Lusztig positive structure}  on $G/Q$. Denote by $\Pos(G/Q)$ the set of all rational functions on $G/Q$ that are positive with respect to
$\cP^{\sG/\sQ}_{\rm Lusztig}$. Denote by $(G/Q)_{>0}$ the totally positive part of $G/Q$ defined by $\cP^{\sG/\sQ}_{\rm Lusztig}$.
\hfill $\diamond$
}
\end{definition}

\begin{example}\label{ex:GBpos}
{\rm For $G/B$, it follows from the definition of $\cP^{\sG/\sB}_{\rm Lusztig}$ that for any
$\bw_0 \in \cR(w_0)$, one has
\[
(G/B)_{>0} = \{x_{\bw_0}^-(c_1, \ldots, c_{l_0})_\cdot B: \; c_1, \ldots, c_{l_0} \in \RR_{>0}\},
\]
which coincides with
the totally positive part of $G/B$ defined by  Lusztig in \cite[$\S$8]{Lusz:pos}.
\hfill $\diamond$
}
\end{example}

\begin{remark}\label{re:posGNv}
{\rm
Let $v \in W$ and recall that double Bruhat cell $G^{w_0, v^{-1}}$ is defined as
\[
G^{w_0, v^{-1}} = Bw_0B \cap B^-v^{-1} B^- \subset G.
\]
A toric chart $x_{\bf i}$ in $G^{w_0, v^{-1}}$ (in fact for any
double Bruhat cell in $G$) is defined in \cite[(1.3)]{FZ:double} using
 any {\it double reduced word} ${\bf i}$ of $(w_0, v^{-1})$ as (we refer to \cite{FZ:double} for the notation)
\[
x_{\bf i}(a; t_1, \ldots, t_l) = a x_{i_1}(t_1) \cdots x_{i_l}(t_l), \hs a \in T, \, (t_1, \ldots, t_l) \in (\CC^\times)^l,
\]
and \cite[(2.9) and (2.11)]{FZ:double} show that the positive equivalence class $[x_{\bf i}]$ of toric charts is independent of the choice of ${\bf i}$.
Modifying the toric chart $x_{\bf i}$ to $x_{\bf i}^\prime$ by
\[
x_{\bf i}^\prime(a; t_1, \ldots, t_l) = x_{i_1}(t_1) \cdots x_{i_l}(t_l)a, \hs a \in T, \, (t_1, \ldots, t_l) \in (\CC^\times)^l,
\]
one has $[x_{\bf i}^\prime]=[x_{\bf i}]$ by \cite[(2.5)]{FZ:double}.
We will refer to the positive equivalence class $[x_{\bf i}]$, for any double reduced word ${\bf i}$ of $(w_0, v^{-1})$,
as the {\it Lusztig positive structure on $G^{w_0, v^{-1}}$}.

On the other hand, as $G^{w_0, v^{-1}}$ is a Zariski open subset of $B^-v^{-1}B^-$, the open embedding $\delta_v: B^-v^{-1}B^- \to G/N(v)$ restricts to
an open embedding,
\begin{equation}\label{eq:delta-vp}
\delta_{w_0, v}: \;\; G^{w_0, v^{-1}} \longrightarrow G/N(v), \;\; g \longmapsto g_\cdot N(v).
\end{equation}
Taking ${\bf i} = (\bw_0, \bfv^{-1})$ to be the double reduced word of $(w_0, v^{-1})$ with the simple reflections in
$\bw_0$ as in the negative alphabet and those in $\bfv^{-1}$ as in the positive alphabet,  one has
$\rho_{(\bw_0, \bfv)}^{\sG/\sN(v)} = \delta_v \circ x_{\bf i}^\prime$. Thus $\delta_{w_0, v}: G^{w_0, v^{-1}} \to G/N(v)$ is a
{\it positive open embedding} in the sense that for any toric chart $\rho$ in $G^{w_0, v^{-1}}$, $\rho$ is in the
Lusztig positive structure on $G^{w_0, v^{-1}}$ if and only if $\delta_{w_0, v} \circ \rho$ is a toric chart in Lusztig positive structure on $G/N(v)$.

For any $w \in W$, the restriction of $\delta_v$ to the double Bruhat cell
\[
G^{w, v^{-1}} = BwB \cap B^-v^{-1}B^- \subset B^-v^{-1}B^-
\]
gives an embedding of $G^{w, v^{-1}}$ to $G/N(v)$. It will be explained in Example \ref{ex:Tleaves-G} that the image of $G^{w, v^{-1}}$ in $G/N(v)$
is a $T$-leaf of the standard Poisson structure on $G/N(v)$.
\hfill $\diamond$
}
\end{remark}

\subsection{Some auxiliary facts on generalized minors}\label{ssec:auxi}  In this section, we first recall some facts from
\cite{FZ:double} on flag minors, and, for $v \in V$, we give examples of {\it regular} functions on $G/N(v)$ that are in $\Pos(G/N(v))$.

\begin{definition}\label{de:f-minor}
{\rm
A generalized minor of the form $\Delta_{w\omega_\al, \omega_\al}$ or $\Delta_{\omega_\al, w\omega_\al}$,
where $w \in W$ and $\al \in \Gamma$, is called a {\it flag minor}.  For $w \in W$ and $\bw = (s_{\al_1}, \ldots, s_{\al_k}) \in \cR(w)$,  set
\begin{equation}\label{eq:Delta-w}
\Delta_{\bw, j} = \Delta_{s_{\al_1} \cdots s_{\al_j} \omega_{\al_j}, \, \omega_{\al_j}}  \hs \mbox{and} \hs
\Delta_{j, \bw} = \Delta_{\omega_{\al_j}, \, s_{\al_1} \cdots s_{\al_j} \omega_{\al_j}},\;\;\; j \in [1, k].
\end{equation}
\hfill $\diamond$
}
\end{definition}

Clearly a flag minor of the form $\Delta_{w \omega_\al, \omega_\al}$ is invariant under the right translation by elements in $N$.
Similarly, a flag minor of the form $\Delta_{\omega_\al, w \omega_\al}$ is invariant under the left translation by elements in $N^-$.
Flag minors  of $g \in SL(n, \CC)$ of size $i \in [1, n]$
are the determinants of the submatrices of $g$ formed by the first $i$ columns and
any $i$ rows, or the first $i$ rows and any columns. See \cite{BFZ:para}.

Recall from Notation \ref{nota:intro} the weak order $\preceq$ on $W$: $w_1 \preceq w$ if $l(w) =
l(w_1) + l(w_1^{-1}w)$.
For $w \in W$, let $N_-^w = N^- \cap \ow N \ow^{\, -1}$, and recall that $N_w = N \cap \ow N^- \ow^{\, -1}$. Let
\begin{align*}
F(w) &= \{\Delta_{w_1 \omega_{\al}, \,w_2\omega_{\al}}|_{\sN_-^w}: \,
\al \in \Gamma, \, w_2 \preceq w_1\preceq  w\} \subset \CC[N_-^w],\\
F^\prime(w)& = \{\Delta_{w_1 \omega_{\al}, \,w_2\omega_{\al}}|_{\sN_w}: \,
\al \in \Gamma, \, w_1 \preceq w_2\preceq  w\} \subset \CC[N_w],
\end{align*}
and for
$\bw = (s_{\al_1}, \ldots. s_{\al_k}) \in \cR(w)$, let
\begin{align*}
F_1(\bw) &= \{\Delta_{\bw, j}|_{\sN_-^w}: \;  j \in [1, k]\} \subset F(w),\\
F_1^\prime(\bw) &= \{\Delta_{j,\bw}|_{\sN_w}: \;  j \in [1, k]\} \subset F^\prime(w).
\end{align*}

\begin{lemma}\label{le:w-f-minor}\cite[Theorem 2.22]{FZ:double}
For any $w \in W$ and $\bw \in \cR(w)$,

1) $F_1(\bw)$ is a transcendental basis of $\CC(N_-^w)$, and $F(w) \subset \Pos(F_1(\bw))\subset \CC(N_-^w)$;

2) $F_1^\prime(\bw)$ is a transcendental basis of $\CC(N_w)$, and $F^\prime(w) \subset \Pos(F_1^\prime(\bw))\subset \CC(N_w)$.
\end{lemma}

\begin{proof}
1) is part of \cite[Theorem 2.22]{FZ:double}. 2) follows from 1), the fact that
$(N_w)^\tau = N_-^{w}$, and the following identity from \cite[Proposition 2.7]{FZ:double}:
\begin{equation}\label{eq:gg-tau}
\Delta_{u \omega_\al,\, v \omega_\al}(g) = \Delta_{v \omega_\al,\, u \omega_\al}(g^\tau), \hs g \in G,
\end{equation}
where $\tau$ is the involutive anti-automorphism of $G$ in \eqref{eq:tau-iota}.
\end{proof}

Lemma \ref{le:w-f-minor} can be extended as follows.

\begin{lemma}\label{le:ww-minor}
For any $w \in W$ and $\bw = (s_{\al_1}, \ldots, s_{\al_k}) \in \cR(w)$, $\{\Delta_{\bw, j}|_{\sN^-}:  j \in [1, k]\}$ is a set of algebraically independent
regular functions on $N^-$, and
\[
\{\Delta_{w_1 \omega_{\al}, \, w_2\omega_\al}|_{\sN^-}:\; w_2 \preceq w_1 \preceq w\}
\subset \Pos(\Delta_{\bw, 1}|_{\sN^-}, \ldots, \Delta_{\bw, k}|_{\sN^-}).
\]
Similarly, $\{\Delta_{j, \bw}|_{\sN}: j \in [1, k]\} \subset \CC[N]$ is algebraically independent, and
\[
\{\Delta_{w_1 \omega_{\al}, \, w_2\omega_\al}|_{\sN}:\; w_1 \preceq w_2 \preceq w\} \subset \Pos(\Delta_{1, \bw}|_{\sN}, \ldots, \Delta_{k, \bw}|_{\sN}).
\]
\end{lemma}

\begin{proof} Let $m \in N^-$ and write $m$ uniquely as $m = m_1m_2$, where $m_1 \in N^- \cap \ow N^- \ow^{-1}$ and
$m_2 \in N^- \cap \ow N \ow^{-1}$. Then, for any $w_2 \preceq w_1 \preceq w$  and $\al \in \Gamma$, since
$\overline{w_1}^{-1} m_1 \overline{w_1} \in N^-$ by \eqref{eq:Nu-Nu-2},  one has
\begin{equation}\label{eq:pqD00}
\Delta_{w_1 \omega_\al, \, w_2\omega_\al}(m) = \Delta^{\omega_\al}(\overline{w_1}^{-1} m_1 \overline{w_1}\,
\overline{w_1}^{\, -1} m_2 \overline{w_2}) = \Delta_{w_1 \omega_\al, \, w_2\omega_\al}(m_2).
\end{equation}
The statement on $\{\Delta_{\bw, j}|_{\sN^-}:  j \in [1, k]\}$ now follows from
\eqref{eq:pqD00} and 1) of Lemma \ref{le:w-f-minor} . The statement on $\{\Delta_{j, \bw}|_{\sN}:  j \in [1, k]\}$ is proved using
\eqref{eq:gg-tau}.
\end{proof}

We now turn to examples of regular functions on $G/N(v)$, for $v \in W$,  that are positive with respect to $\cP^{\sG/\sN(v)}_{\rm Lusztig}$.
We identify $\CC[G/N(v)]$ with $\CC[G]^{\sN(v)}$, the algebra of right $N(v)$-invariant regular functions on $G$.

\begin{proposition}\label{pr:minor-pos}
For any $w, v_1 \in W$ such that $v_1 \preceq v$ and for all $\al \in \Gamma$,
$\Delta_{w\omega_\al, v_1\omega_\al}\in \CC[G]^{\sN(v)} \cong \CC[G/N(v)]$ and lies in
$\Pos(G/N(v))$.
\end{proposition}

\begin{proof} By \eqref{eq:Nu-Nu-2},
$\overline{v_1}^{\, -1} N(v) \overline{v_1} \subset N$. It follows that
$\Delta_{w\omega_\al, v_1\omega_\al}\in \CC[G]^{\sN(v)}$ for
any $w \in W$ and $\al \in \Gamma$.

Choose any $\bw_0 = (s_{\al_1}, \ldots, s_{\al_{l_0}}) \in \cR(w_0)$ and $\bfv = (s_{\al_{l_0+1}}, \ldots, s_{\al_l}) \in \cR(v)$.
For any rational function $f$ on $G/N(v)$, define $\tilde{f} \in \CC(c_1, \ldots, c_{l+d})$ by
\[
\tilde{f}:=f \circ \rho^{\sG/\sN(v)}_{(\bw_0, \bfv)} =
f \left(x_{\bw_0}^-(c_1, \ldots, c_{l_0}) x_{\bfv^{-1}}^+(c_{l_0+1}, \ldots, c_l) \sigma(c_{l+1}, \ldots, c_{l+d})\right).
\]
Note that as $\rho^{\sG/\sN(v)}_{(\bw_0, \bfv)}$ is an open embedding,
$\tilde{f} \neq 0$ if $f \neq 0$. By the definition of $\cP^{\sG/\sN(v)}_{\rm Lusztig}$,
$f \in \Pos(G/N(v))$ if and only if $\tilde{f} \in \Pos(c_1, \ldots, c_{l+d})$.
Write $t= \sigma(c_{l+1}, \ldots, c_{l+d})$, and note that
for any $\al \in \Gamma$, $t^{\omega_\al}$ is one of the coordinates in $(c_{l+1}, \ldots, c_{l+d})$.

{\bf Case 1.} Assume first that $v_1 = e$. Then for any $w \in W$ and $\al \in \Gamma$, one has
\[
\tilde{\Delta}_{w\omega_\al, \omega_\al} = t^{\omega_\al}\Delta_{w\omega_\al, \omega_\al}(x_{\bw_0}^-(c_1, \ldots, c_{l_0})).
\]
By \cite[Theorem 5.8]{BZ:tensor}, $\Delta_{w\omega_\al, \omega_\al}(x_{\bw_0}^-(c_1, \ldots, c_{l_0}))$ is a polynomial in $(c_1, \ldots, c_{l_0})$
with non-negative integer coefficients. As $\Delta_{w\omega_\al, \omega_\al} \neq 0$, ${\Delta}_{w\omega_\al, \omega_\al}\in \Pos(G/N(v))$.

{\bf Case 2.} Suppose now that $w = e$. Then for any $v_1 \preceq v$ and $\al \in \Gamma$, one has
\[
\tilde{\Delta}_{\omega_\al, v_1\omega_\al}= t^{v_1\omega_\al}\Delta_{\omega_\al, v_1\omega_\al}(x_{\bfv^{-1}}^+(c_{l_0+1}, \ldots, c_l)).
\]
By \cite[Theorem 5.8]{BZ:tensor} again, $\Delta_{\omega_\al, v_1\omega_\al}(x_{\bfv^{-1}}^+(c_{l_0+1}, \ldots, c_l))$
is a polynomial in $(c_{l_0}, \ldots, c_l)$
with non-negative integer coefficients. Thus ${\Delta}_{\omega_\al, v_1\omega_\al}\in \Pos(G/N(v))$.

In the notation of \cite{FZ:double}, consider now the collection $F({\bf i}) = \{\Delta_{k, {\bf i}^*}: k \in [1, l+d]\}$ of generalized minors on $G$
as defined in \cite[(1.22)]{FZ:double},  where we take $u = w_0$,
\[
{\bf i}^*=(\bw_0, \bfv) = (s_{\al_1}, \,\ldots, \,s_{\al_{l_0}}, \,s_{l_0+1}, \,\ldots, \,s_{\al_{l}})
\]
as a double reduce word of $(w_0, v)$ with the first set of $l_0$ simple roots regarded as in the negative alphabet
and the second set of $l-l_0 = l(v)$ simple roots
in the positive alphabet, and ${\bf i}$ is the word ${\bf i}^*$ read backwards (the double reduced word ${\bf i}^*$ is
 {\it unmixed} in the terminology in the proof of \cite[Proposition 2.29]{FZ:double}).
More precisely, by \cite[(1.16)]{FZ:double},
\[
\Delta_{k, {\bf i}^*} = \begin{cases} \Delta_{s_{l_0} \cdots s_{\al_k} \omega_{\al_k}, \, \omega_{\al_k}},& \hs k \in [1,\, l_0],\\
\Delta_{\omega_{\al_k}, \, s_{\al_{l_0+1}} \cdots s_{\al_{k-1}} \omega_k}, & \hs k \in [l_0+1,\, l],\\
\Delta_{\omega_{k-l}, \, v \omega_{k-l}}, & \hs k \in [l+1, \, l+d].\end{cases}
\]
As  each generalized minor in $F({\bf i})$ is a flag minor of the types in the two cases discussed above, one knows that
$F({\bf i}) \subset \Pos(G/N(v))$.
By \cite[Theorem 1.12]{FZ:double}, the set
\[
F({\bf i})|_{\sG^{w_0, v^{-1}}}= \{\Delta_{k, {\bf i}^*}|_{\sG^{w_0, v^{-1}}}: \, k \in [1, l+d]\}
\]
 is a transcendental basis for the field $\CC(G^{w_0, v^{-1}})$ of rational functions on $G^{w_0, v^{-1}}$, and that
every element in
\[
F(w_0, v^{-1})|_{\sG^{w_0, v^{-1}}} := \{\Delta_{w\omega_\al, \, v_1\omega_\al}|_{\sG^{w_0, v^{-1}}}: \, w \in W, v_1 \preceq v\}
\]
has a subtraction-free expression in the elements in $F({\bf i})|_{\sG^{w_0, v^{-1}}}$. Since the
image of the map $(\CC^\times)^{l+d} \to G$ given by
\[
(c_1,\ldots, c_{l+d}) \longmapsto x_{\bw_0}^-(c_1, \ldots, c_{l_0}) x_{\bfv^{-1}}^+(c_{l_0+1}, \ldots, c_l) \sigma(c_{l+1}, \ldots, c_{l+d})
\]
lies in $G^{w_0, v^{-1}}$, $\Delta_{w\omega_\al, v_1\omega_\al} \in \Pos(G/N(v))$ for all $w, v_1 \in W$ with $v_1 \preceq v$ and $\al \in \Gamma$.
\end{proof}

\begin{remark}\label{re:GQpos-inter}
{\rm
We note that  any $v \in W$ and $Q = B(v)$ or $N(v)$, one has
\[
(G/Q)_{>0} \subset \bigcap_{w \in W} wB^-B/Q.
\]
Indeed, for any $w \in W$ and $g \in G$,  $g_\cdot Q \in wB^-B/Q$ if and only if $\Delta_{w\omega_\al, \omega_\al}(g) \neq 0$ for all $\al \in \Gamma$.
Take any $(\bw_0, \bfv) \in \cR(w_0) \times \cR(v)$, $w \in W$, $\al \in \Gamma$, and $c_j \in \RR_{>0}$ for $j \in [1,l+d]$.
Then by \cite[Theorem 5.8]{BZ:tensor},
\begin{align*}
&\Delta_{w\omega_\al, \omega_\al}\left(\rho_{(\bw_0, \bfv)}^{\sG/\sB(v)}(c_1, \ldots, c_l)\right)
= \Delta_{w\omega_\al, \omega_\al}(x_{\bw_0}^-(c_{1}, \ldots, c_{l_0}))>0,\\
&\Delta_{w\omega_\al, \omega_\al}\left(\rho_{(\bw_0, \bfv)}^{\sG/\sN(v)}(c_1, \ldots, c_{l+d})\right) = \sigma(c_{l+1}, \ldots, c_{l+d})^{\omega_\al}
\Delta_{w\omega_\al, \omega_\al}(x_{\bw_0}^-(c_{1}, \ldots, c_{l_0}))>0.
\end{align*}
Thus $(G/Q)_{>0} \subset wB^-B/Q$ for both $Q = B(v)$ and $Q = N(v)$.
\hfill $\diamond$
}
\end{remark}

\subsection{Positivity of the Bott-Samelson coordinates on $G/Q$}\label{ssec:BSpos-GQ}
Let again $v \in W$ and $Q = B(v)$ or $N(v)$. We now prove Theorem A stated in the Introduction, namely that all the Bott-Samelson coordinates on
$G/Q$ are positive (rational) functions with respect to the Lusztig positive structure. We prove the following more detailed restatement of Theorem A.
For $v \in W$ and $Q = B(v)$ or $N(v)$, let
\[
d(Q) = \dim(G/Q) = \begin{cases} l = l_0 + l(v), & \;\;\; Q = B(v),\\ l+d, &\;\;\; Q = N(v).\end{cases}
\]

\begin{theorem}\label{th:GNv-pos}
Let $v \in W$ and $Q = B(v)$ or $N(v)$.
For any $w \in W$ and any $\br \in \cR(w_0w^{-1}) \times \cR(w) \times \cR(v)$, all the Bott-Samelson coordinates
$(z_1, \ldots, z_{d(\sQ)})$ on $wB^-B/Q$ defined by $\br$, when regarded as rational functions on $G/Q$, are in $\Pos(G/Q)$.
\end{theorem}

\begin{proof} Fix $w \in W$ and
$\br = (\bw^0, \bw, \bfv) \in \cR(w_0w^{-1}) \times \cR(w)
\times \cR(v)$, where
\[
\bw^0 = (s_{\al_1}, \ldots, s_{\al_k}),  \;\; \bw = (s_{\al_{k+1}}, \ldots, s_{\al_{l_0}}), \;\; \mbox{and} \;\;
 \bfv = (s_{\al_{l_0+1}}, \ldots, s_{\al_{l}}).
\]
We first consider the case of $Q = N(v)$.

Write an element in $wB^-B/N(v)$ uniquely as $g_\cdot N(v)$, where $g = \ow m n t_\cdot N(v)$ for unique $m \in N^-, n \in N_v$, and $t \in T$.
The values of the Bott-Samelson coordinates $z_1, \ldots, z_{l+d}$ at $g_\cdot N(v) \in wB^-B/N(v)$ are given in Proposition \ref{pr:zj-Nv}.

Assume first that $j \in [1, k]$. By Proposition \ref{pr:zj-Nv},
\[
z_j(g_\cdot N(v)) = \Delta_{s_{\al_1^*} \cdots s_{\al_j^*} \omega_{\al_j^*}, \, s_{\al_1^*} \cdots s_{\al_{j-1}^*} \omega_{\al_j^*}}(m).
\]
Let $\bfu^* = (s_{\al_1^*},  \ldots s_{\al_k^*}) \in \cR(w^{-1}w_0)$ and $\bw_0^\prime =
(s_{\al_{k+1}}, \ldots, s_{\al_{l_0}}, s_{\al_1^*},  \ldots s_{\al_k^*}) \in \cR(w_0)$.
By Lemma \ref{le:ww-minor}, there exist  $p, q \in \Poly_j^{>0}$ (see Notation \ref{nota:Pos}) such that
\begin{equation}\label{eq:pqm}
q(\Delta_{\bfu^*, 1}, \ldots, \Delta_{\bfu^*, j})(m) = z_j(g_\cdot N(v)) p(\Delta_{\bfu^*, 1}, \ldots, \Delta_{\bfu^*, j})(m).
\end{equation}
On the other hand, for each $i \in [1, j]$,
\begin{equation}\label{eq:mi}
\Delta_{\bfu^*,i}(m) = \Delta_{\bfu^*, i}(\ow^{\, -1}g t^{-1} n^{-1}) =
t^{-\omega_{\al_i^*}} \Delta_{ws_{\al_1^*} \cdots s_{\al_i^*} \omega_{\al_i^*}, \, \omega_{\al_i^*}}(g)
=\frac{\Delta_{\bw_0^\prime, \,l(w)+i}(g)}{\Delta_{w\omega_{\al_i^*}, \omega_{\al_i^*}}(g)}.
\end{equation}
For $i \in [1, j]$, set $f_i \in \CC[wB^-B/N(v)]$ by
\[
f_i (g_\cdot N(v)) = \Delta_{\bw_0^\prime, \,l(w)+i}(g)/\Delta_{w\omega_{\al_i^*}, \omega_{\al_i^*}}(g), \hs g \in \ow N^- N_vT.
\]
By \eqref{eq:pqm} and \eqref{eq:mi}, $q(f_1, \ldots, f_j) = z_jp(f_1, \ldots, f_j) \in \CC[wB^-B/N(v)]$.
By Proposition \ref{pr:minor-pos}, $f_i \in \Pos(G/N(v))$ for each $i \in [1, j]$. Thus $z_j \in \Pos(G/N(v))$ by Remark \ref{re:Pos-X}.

Suppose now that $j \in [k+1, l_0]$. By Proposition \ref{pr:zj-Nv} and using the involutive automorphism $\iota$ of $G$ given in \eqref{eq:tau-iota}
and the identities in
\eqref{eq:Delta-all} and
\eqref{eq:ow-tau}, one has
\begin{align*}
z_j(g_\cdot N(v)) &
=\Delta^{\omega_{\al_j}}\left(\overline{s_{\al_j} \cdots s_{\al_{l_0}}} \, m \, \overline{s_{\al_{j+1}} \cdots s_{\al_{l_0}}}^{\, -1}\right)\\
& = \Delta^{\omega_{\al_{j}}}\left(\overline{s_{\al_{l_0}} \cdots s_{\al_{j}}}^{\, -1}  \, (m^{-1})^\iota \, \overline{s_{\al_{l_0}} \cdots s_{\al_{j+1}}}\right)\\
 &= \Delta_{s_{\al_{l_0}} \cdots s_{\al_{j}} \omega_{\al_j}, \, s_{\al_{l_0}} \cdots s_{\al_{j+1}}\omega_{\al_j}}((m^{-1})^\iota).
\end{align*}
Let $\bw^{-1} = (s_{\al_{l_0}}, \ldots, s_{\al_{k+1}}) \in \cR(w^{-1})$.
By Lemma \ref{le:ww-minor}, there exist
$p', q' \in \Poly_{l_0-j+1}^{>0}$ such that
\[
q'(\Delta_{\bw^{-1}, 1}, \ldots, \Delta_{\bw^{-1}, l_0-j+1})((m^{-1})^\iota) = z_j(g_\cdot N(v))
p'(\Delta_{\bw^{-1}, 1}, \ldots, \Delta_{\bw^{-1}, l_0-j+1})((m^{-1})^\iota).
\]
On the other hand, for $i \in [1, l_0-j+1]$ and $i' := l_0-i+1 \in [j, l_0]$,
\begin{align*}
\Delta_{\bw^{-1}, i}((m^{-1})^\iota)&  = \Delta^{\omega_{\al_{i'}}}(\overline{s_{\al_{l_0}} \cdots s_{\al_{i'}}}^{\, -1}(m^{-1})^\iota)
=\Delta^{\omega_{\al_{i'}}}(\overline{s_{\al_{i'}} \cdots s_{\al_{l_0}}}\, m)\\
& = \Delta^{\omega_{\al_{i'}}}(\overline{s_{\al_{i'}} \cdots s_{\al_{l_0}}}\, \ow^{\, -1} g t^{-1} n)
=\frac{\Delta_{s_{\al_{k+1}} \cdots s_{\al_{i'-1}}\omega_{\al_{i'}}, \omega_{\al_{i'}}}(g)}{\Delta_{w\omega_{\al_{i'}}, \,\omega_{\al_{i'}}}(g)}.
\end{align*}
For $i \in [1, l_0-j+1]$, set $f^\prime_i \in \CC[wB^-B/N(v)]$ by
\[
f_j^\prime(g_\cdot N(v)) =
\Delta_{s_{\al_{k+1}} \cdots s_{\al_{i'-1}}\omega_{\al_{i'}}, \omega_{\al_{i'}}}(g)/\Delta_{w\omega_{\al_{i'}}, \, \omega_{\al_{i'}}}(g),
\hs g \in \ow N^-N_vT.
\]
One then has $q'(f_1^\prime, \ldots, f_{l_0-j+1}^\prime) = z_jp'(f_1', \ldots, f_{l_0-j+1}^\prime) \in \CC[wB^-B/N(v)]$, which, by
Proposition \ref{pr:minor-pos} and Remark \ref{re:Pos-X}, implies that
$z_j\in\Pos(G/N(v))$.

Assume now that $j \in [l_0+1, l]$. By Proposition \ref{pr:zj-Nv},
\[
z_j(g_\cdot N(v)) = \Delta_{s_{\al_{l_0+1}} \cdots s_{\al_{j-1}}\omega_{\al_j}, \; s_{\al_{l_0+1}} \cdots s_{\al_j}\omega_{\al_j}}(n).
\]
By Lemma \ref{le:ww-minor}, there exist non-zero $p^{\prime\prime}, q^{\prime\prime}
\in \Poly_{j-l_0}^{>0}$ such that
\[
q^{\prime\prime}(\Delta_{1, \bfv}, \ldots, \Delta_{j-l_0, \bfv})(n) = z_j(g_\cdot N(v))
p^{\prime \prime}(\Delta_{1, \bfv}, \ldots, \Delta_{j-l, \bfv})(n).
\]
On the other hand, for $i \in [1, j-l_0]$,
\begin{align*}
\Delta_{i, \bfv}(n) & = \Delta^{\omega_{\al_{l_0+i}}}(n \, \overline{s_{\al_{l_0+1}} \cdots s_{\al_{l_0+i}}}) =
\Delta^{\omega_{\al_{l_0+i}}}(m^{-1} \ow^{\, -1} g t^{-1}  \, \overline{s_{\al_{l_0+1}} \cdots s_{\al_{l_0+i}}})\\
& = t^{-s_{\al_{l_0+1}} \cdots s_{\al_{l_0+i}} \omega_{\al_{l_0+i}}} \Delta_{w \omega_{\al_{l_0+i}}, \,s_{\al_{l_0+1}} \cdots s_{\al_{l_0+i}}}(g).
\end{align*}
Writing $s_{\al_{l_0+1}} \cdots s_{\al_{l_0+i}} \omega_{\al_{l_0+i}}=\sum_{\al \in \Gamma} k_\al \omega_{\al}$, one has
\[
t^{s_{\al_{l_0+1}} \cdots s_{\al_{l_0+i}} \omega_{\al_{l_0+i}}} = \prod_{\al \in \Gamma} \left(\Delta_{w\omega_{\al}, \omega_{\al}}(g)\right)^{k_\al}.
\]
For $i \in [1, j-l_0]$, set $f_i^{\prime\prime} \in \CC[wB^-B/N(v)]$ by
\[
f_i^{\prime\prime}(g_\cdot N(v)) = \frac{\Delta_{w \omega_{\al_{l_0+i}}, \,s_{\al_{l_0+1}} \cdots s_{\al_{l_0+i}}}(g)}
{\prod_{\al \in \Gamma} \left(\Delta_{w\omega_{\al}, \omega_{\al}}(g)\right)^{k_\al}}, \hs g \in \ow N^-NT.
\]
One then has $q^{\prime\prime}(f_1^{\prime\prime}, \ldots, f_{1, j-l_0}^{\prime\prime}) = z_j
p^{\prime\prime}(f_1^{\prime\prime}, \ldots, f_{1, j-l_0}^{\prime\prime})
\in \CC[wB^-B/N(v)]$, which, again by Proposition \ref{pr:minor-pos} and Remark \ref{re:Pos-X}, implies that
$z_j\in\Pos(G/N(v))$.

For $j \in [l+1, l+d]$, since $z_j(g_\cdot N(v))  = \Delta_{w \omega_{j-l}, \omega_{j-l}}(g)$, again $z_j\in \Pos(G/N(v))$.

Let now $Q = B(v)$. Let $\varpi: G/N(v) \to G/B(v)$ be the natural projection. Then
$(\varpi^* (z_1), \ldots, \varpi^*(z_l))$ coincides with the first $l$ Bott-Samelson coordinates on $wB^-B/N(v)$ defined by the same $\br$.
 On the other hand, for any toric chart $\rho^{\sG/\sB(v)}_{(\bw_0, \bfv)}: (\CC^\times)^l \to G/B(v)$ in \eqref{eq:rho-c-Bv},
and for any $j \in [1, l]$, one has
\[
z_j\left(\rho^{\sG/\sB(v)}_{(\bw_0, \bfv)}(c_1, \ldots, c_l)\right) = \varpi^*(z_j)
\left(\rho^{\sG/\sN(v)}_{(\bw_0, \bfv)}(c_1, \ldots, c_l, 1, \ldots, 1)\right) \in \Pos(c_1, \ldots, c_l).
\]
Thus $z_j \in \Pos(G/B(v))$ for each $j \in [1, l]$.
\end{proof}

\begin{remark}\label{re:SL2}
{\rm
As we remarked in the paragraph after the statement of Theorem A in $\S$\ref{ssec:statements},
Bott-Samelson charts on $G/Q$ are not to be confused with toric charts in $\cP^{\sG/\sQ}_{\rm Lusztig}$.
Consider the case of $G/Q = G = SL(2, \CC)$:
there are two Bott-Samelson charts on $G$, corresponding to $w = e$ and $w = s_1=s$, respectively
given by
\begin{align*}
\bsigma^{(s, \emptyset, s)}:\; \CC^2 \times \CC^\times \longrightarrow B^-B, \;&
(\xi_1, \xi_2, \xi_3) \longmapsto \left(\begin{array}{cc} 1 & 0 \\ \xi_1 & 1\end{array}\right)\!\!
\left(\begin{array}{cc} 1 & \xi_2 \\ 0 & 1\end{array}\right)\left(\begin{array}{cc} \xi_3 & 0\\ 0 &\xi_3^{-1}\end{array}\!\right)\!\!\\
&\hs \hs \hs= \left(\!\begin{array}{cc} \xi_3 & \xi_2\xi_3^{-1} \\ \xi_1\xi_3 & (\xi_1\xi_2+1)\xi_3^{-1}\end{array}\!\!\!\right),\\
\bsigma^{(\emptyset, s, s)}: \; \CC^2 \times \CC^\times \longrightarrow sB^-B, \; &(z_1, z_2, z_3) \longmapsto
\left(\begin{array}{cc} z_1 & -1 \\ 1 & 0\end{array}\right)\!\!
\left(\begin{array}{cc} 1 & z_2 \\ 0 & 1\end{array}\right)\!\!\left(\begin{array}{cc} z_3 & 0\\ 0 &z_3^{-1}\end{array}\!\right)\\
& \hs \hs \hs =\left(\!\begin{array}{cc} z_1z_3 & (z_1z_2-1)z_3^{-1} \\ z_3 & z_2z_3^{-1}\end{array}\!\!\!\right).
\end{align*}
The two sets of coordinates are related by
\[
\xi_1 = z_1^{-1}, \;\;\; \xi_2 = z_1(z_1z_2-1), \;\;\; \xi_3 = z_1z_3, \;\;\;
z_1 = \xi_1^{-1}, \;\;\; z_2 = \xi_1(\xi_1\xi_2+1), \;\;\; z_3 = \xi_1\xi_3.
\]
It is a coincidence  that in this case the toric chart $\rho^\sG_{(s, s)}: (\CC^\times)^3 \to SL(2, \CC)$ in the Lusztig positive structure
is equal to
$\bsigma^{(s, \emptyset, s)}|_{(\CC^\times)^3}$. Note that
while each $z_j$ has a subtraction-free expression in $(\xi_1, \xi_2, \xi_3)$, $\xi_2$ does not have a subtraction-free expression in $(z_1, z_2, z_3)$:
indeed, otherwise the elements $g \in sB^-B$ with $(z_1, z_2, z_3)$-coordinates satisfying $z_1=z_2=1$ and $z_3 > 0$ would lie in
$G_{>0}$ which is not the case.
\hfill $\diamond$
}
\end{remark}

\begin{example}\label{ex:Sp2}
{\rm
In this example we take $G/Q = G = \SP(2, \CC) = \{g \in {\rm GL}(4, \CC): g^t J_{2} g = J_{2}\}$ where $J_{2} =
\left(\begin{array}{cc} 0 & I_2\\ - I_2 & 0\end{array}\right)$, and $I_2$ is the identity matrix of size $2$. Choose the Cartan subalgebra
$\h$ of the Lie algebra $\sp(2, \CC)$ to be the set of all diagonal elements in $\sp(2, \CC)$ and write elements in $\h$ as
$x = {\rm diag}(x_1, x_2, -x_1, -x_2)$.
Choose the simple roots as $\al_1 = x_1-x_2$ and $\al_2 = 2x_2$, and let $s_1 = s_{\al_1}$ and $s_2 = s_{\al_2}$, and choose root vectors
\[
e_{\al_1} = \left(\begin{array}{cccc} 0 & 1 & 0 & 0\\ 0 & 0 & 0 & 0\\ 0 & 0 & 0 & 0 \\ 0 & 0 & -1 & 0\end{array}\right),
\hs e_{\al_2} = \left(\begin{array}{cccc} 0 & 0 & 0 & 0\\ 0 & 0 & 0 & 1\\ 0 & 0 & 0 & 0 \\ 0 & 0 & 0 & 0\end{array}\right).
\]
Then the fundamental weights are $\omega_{\al_1} = x_1$ and $\omega_{\al_2} = x_1+x_2$.
Write $g \in \SP(2, \CC)$ as $g=(a_{ij})_{i, j = 1, \ldots, 4}$, and set $\Delta_{ij, kj} =
{\rm det}\left(\begin{array}{cc} a_{ik} & a_{il}\\ a_{jk} & a_{jl}\end{array}\right)$ for $i<j$ and $k < l$. Let
\begin{align*}
{\bf \Delta}^{(1)} & = \left(
\begin{array}{cccc}
    \Delta_{\omega_1,\omega_1} & \Delta_{\omega_1,s_1\omega_1} & \Delta_{\omega_1,s_1s_2s_1\omega_1} & \Delta_{\omega_1,s_2s_1\omega_1}  \\
    \Delta_{s_1\omega_1,\omega_1} &  \Delta_{s_1\omega_1,s_1\omega_1} & \Delta_{s_1\omega_1,s_1s_2s_1\omega_1}
		& \Delta_{s_1\omega_1,s_2s_1\omega_1} \\
    \Delta_{s_1s_2s_1\omega_1,\omega_1}  & \Delta_{s_1s_2s_1\omega_1,s_1\omega_1} & \Delta_{s_1s_2s_1\omega_1,s_1s_2s_1\omega_1}
		& \Delta_{s_1s_2s_1\omega_1,s_2s_1\omega_1} \\
    \Delta_{s_2s_1\omega_1,\omega_1}  & \Delta_{s_2s_1\omega_1,s_1\omega_1} & \Delta_{s_2s_1\omega_1,s_1s_2s_1\omega_1}
		& \Delta_{s_2s_1\omega_1,s_2s_1\omega_1} \\
   \end{array}\right),\\
{\bf \Delta}^{(1)} & = \left(
\begin{array}{cccc}
    \Delta_{\omega_2,\omega_2} & \Delta_{\omega_2,s_2\omega_2} & \Delta_{\omega_2,s_1s_2\omega_2} & \Delta_{\omega_2,s_2s_1s_2\omega_2} \\
    \Delta_{s_2\omega_2,\omega_2} & \Delta_{s_2\omega_2,s_2\omega_2} & \Delta_{s_2\omega_2,s_1s_2\omega_2}
		& \Delta_{s_2\omega_2,s_2s_1s_2\omega_2} \\
    \Delta_{s_1s_2\omega_2,\omega_2}  & \Delta_{s_1s_2\omega_2,s_2\omega_2} & \Delta_{s_1s_2\omega_2,s_1s_2\omega_2}
		& \Delta_{s_1s_2\omega_2,s_2s_1s_2\omega_2} \\
    \Delta_{s_2s_1s_2\omega_2,\omega_2}  & \Delta_{s_2s_1s_2\omega_2,s_2\omega_2} & \Delta_{s_2s_1s_2\omega_2,s_1s_2\omega_2}
		& \Delta_{s_2s_1s_2\omega_2,s_2s_1s_2\omega_2}
   \end{array}
    \right).
\end{align*}
Then it is straightforward to check that
\begin{align*}
{\bf \Delta}^{(1)} & = \left(
\begin{array}{cccc}
     a_{11} & a_{12} & -a_{13} & a_{14}  \\a_{21} &  a_{22} & -a_{23}  & a_{24} \\
    -a_{31}  & -a_{32} & a_{33}   & -a_{34}  \\a_{41}  & a_{42} & -a_{43}  & a_{44}
   \end{array}\right),\\
{\bf \Delta}^{(2)} & =\left(
\begin{array}{cccc}
     \Delta_{12,12} &\Delta_{12,14} & -\Delta_{12,23} & \Delta_{12,34} \\
    \Delta_{14,12} &  \Delta_{14,14} & -\Delta_{14,23} & \Delta_{14,34} \\
    -\Delta_{23,12}  & -\Delta_{23,14} & \Delta_{23,23}   & -\Delta_{23,34} \\
    \Delta_{34,12}  & \Delta_{34,14} & -\Delta_{34,23}  & \Delta_{34,34} \end{array}\right),
\end{align*}
and that the entries of ${\bf \Delta}^{(1)}$ and ${\bf \Delta}^{(2)}$ are all the generalized minors of $\SP(2, \CC)$.

Consider now the Bott-Samelson chart on $G/Q = G = \SP(2, \CC)$ for
$w = e$ and
\[
\br=((s_1, s_2, s_1, s_2), \emptyset, (s_2, s_1, s_2, s_1)) \in \cR(w_0) \times \cR(e) \times \cR(w_0).
\]
A direct calculation gives the Bott-Samelson coordinates $(z_1, \ldots, z_{10})$ on $B^-B$ to be
\begin{align*}
&z_1=\frac{a_{21}}{a_{11}},\;\;\;  z_2=\frac{-\Delta_{23,12}}{\Delta_{12,12}},\;\;\;
z_3=-\frac{\Delta_{13,12}}{\Delta_{12,12}}=\frac{(-\Delta_{23,12})a_{11}+(-a_{31})\Delta_{12,12}}{a_{21}\Delta_{12,12}}, \\
&z_4=\frac{\Delta_{14,12}}{\Delta_{12,12}},\;\;\; z_5=\frac{\Delta_{12,12}\Delta_{12,14}}{a^2_{11}},\;\;\;
z_6=\frac{a_{14}\Delta_{12,12}}{a_{11}},\\
&z_7=a_{11}a_{13}+a_{12}a_{14}=\frac{a_{11}^2 \Delta_{12, 34} + a_{14}^2 \Delta_{12,12}}{\Delta_{12, 14}},\\
&z_8=\frac{a_{11}a_{12}}{\Delta_{12,12}},\;\;\;  z_9=a_{11},\;\;\;  z_{10}=\Delta_{12,12}.
\end{align*}
Note that the above formulas express the Bott-Samelson coordinates $(z_1, \ldots, z_{10})$ both as regular functions on $B^-B$, given by
$a_{11}\Delta_{12,12} \neq 0$,  and as
rational functions in the generalized minors of $\SP(2, \CC)$ in a subtraction-free manner. In the case of $z_3$,
note that $-\Delta_{13,12}$ is not a generalized minor for $\SP(2, \CC)$, while in the case of $z_7$,
$a_{13} = -\Delta_{\omega_1, s_1s_2s_1\omega_1}$. The second identity for $z_3$ comes from the Pl${\rm \ddot{u}}$cker relation for ${\rm GL}(4, \CC)$:
\[
a_{21} \Delta_{13, 12} = a_{11} \Delta_{23, 12} + a_{31} \Delta_{12, 12},
\]
and the  second identity for $z_7$ is the result of the following Pl${\rm \ddot{u}}$cker relation for $\SP(2, \CC)$
(the second identity of \cite[(2) of Theorem 1.16]{FZ:double} applied to $u = v = e, i = 1$ and $j = 2$):
\[
a_{12}a_{14} \Delta_{12,14} = a_{14}^2 \Delta_{12, 12} + a_{11}(a_{11} \Delta_{12, 34} +(-a_{13}) \Delta_{12, 14}).
\]
\hfill $\diamond$
}
\end{example}

\section{Symmetric Poisson CGL extensions and Poisson-Ore varieties}\label{sec:review-CGL}

\subsection{Definitions}\label{ssec:de-CGL}
We review from \cite{GY:PNAS, GY:Poi} the definitions on symmetric Poisson CGL extensions over the field $\CC$, although the general theory of
\cite{GY:PNAS, GY:Poi} works for any field of characteristic $0$.

In what follows, let $\TT$ be a split $\CC$-torus. The (additive) group of algebraic characters of $\TT$
is denoted by $X(\TT)$, and we also regard an element in $X(\TT)$ as  in $\t^*$ through its differential at the identity element of $\TT$, where
$\t$ is the Lie algebra of $\TT$.

Recall that a $\TT$-Poisson algebra is a Poisson algebra $R$
over $\CC$ with a rational $\TT$-action by Poisson isomorphisms. For a $\TT$-Poisson algebra $R$, denote the induced action of $h \in \t$ on $R$ by
$h(r)$ for $r \in R$. More specifically, if $r \in R$ is a $\TT$-weight vector with $t\cdot r = t^\chi r$ for all $t \in \TT$, then
$h(r) = \chi(h)r$ for all $h \in \t$. For a $\TT$-Poisson $\CC$-algebra $(R, \{\, , \, \})$ that is also an integral domain,
define a {\it $\TT$-homogeneous Poisson prime element}
to be a prime element $r \in R$ that is a $\TT$-weight vector and is such that $\{a, R\} \subset aR$.

The following definition of symmetric Poisson CGL extensions is a paraphrase of that given in \cite{GY:PNAS, GY:Poi}.

\begin{definition}\label{de:Poi-CGL} \cite{GY:PNAS, GY:Poi}
{\rm
1) A symmetric $\TT$-Poisson CGL extension of length $n$ is a pair $\cE = (R, \cD)$, where
$R = \left(\CC[z_1, \ldots, z_n],  \{\, , \, \}\right)$ is a polynomial Poisson algebra and
\begin{equation}\label{eq:cD}
\cD =(\chi_1, \ldots, \chi_n, \; h_1, \ldots, h_n, \; h_1^\prime, \ldots, h_n^\prime),
\end{equation}
with $\chi_j \in X(\TT)$ and $h_j, h_j^\prime \in \t$ for $j \in [1, n]$, satisfying
\begin{equation}\label{eq:xkk}
\chi_j(h_j) \neq 0 \hs \mbox{and} \hs \chi_j(h_j^\prime) \neq 0, \hs \forall \; j \in [1, n],
\end{equation}
and such that the Poisson bracket $\{\, , \, \}$ for $\CC[z_1, \ldots, z_n]$ takes the special form
\begin{equation}
\label{eq:xjk}
\{z_i, \, z_j\} = -\chi_i(h_j) z_iz_j - f_{i,j} =\chi_j(h_i^\prime)z_iz_j-f_{i,j}\;\mbox{for some}\;f_{i,j} \in \CC[z_{i+1}, \ldots, z_{j-1}]
\end{equation}
for all $1 \leq i < j \leq n$,
and that the  $\TT$-action on $\CC[z_1, \ldots, z_n]$ by associative algebra isomorphisms via
$t\cdot z_j = t^{\chi_j}z_j$, where $t \in T$ and $j \in [1, n]$,
preserves the Poisson bracket $\{\, , \, \}$. For $1 \leq i < j \leq n$,
\begin{equation}\label{eq:log-can}
\{z_i, z_j\}_{\rm log-can} :=-\chi_i(h_j) z_iz_j
\end{equation}
 is called the
{\it log-canonical term} of $\{z_i, z_j\}$.
We will also refer to the collection $\cD$ in \eqref{eq:cD}
as the {\it $\TT$-action data} on $R$, and to the $z_j$'s as the {\it CGL generators} of $\cE$.

2) A {\it localized symmetric $\TT$-Poisson CGL extension} is a pair $\cE = (R, \cD)$, where
\[
R = (\CC[z_1, \ldots, z_n][y_1^{-1}, \ldots,y_k^{-1}],  \{\, , \, \}),
\]
$((\CC[z_1, \ldots, z_n],  \{\, , \, \}), \cD)$ is a symmetric $\TT$-Poisson CGL extension, and $y_1, \ldots,y_k$ are
$\TT$-homogeneous Poisson prime elements in $\CC[z_1, \ldots, z_n]$.
\hfill $\diamond$
}
\end{definition}

\begin{remark}\label{re:ij}
{\rm
Given a symmetric $\TT$-Poisson CGL extension $(R, \cD)$ as in Definition \ref{de:Poi-CGL}, for any $1 \leq i < j \leq n$,
the pair $\cE_{[i, j]} = (R_{[i,j]}, \cD_{[i,j]})$ is also a symmetric $\TT$-Poisson CGL extension, where
$R_{[i, j]} = \CC[x_i, x_{i+1}, \ldots, x_j]$ is a Poisson subalgebra of $R$, and
\[
\cD_{[i,j]} =  (\chi_i, \ldots, \chi_j, \; h_i, \ldots, h_j, \; h_i^\prime, \ldots, h_j^\prime).
\]
In  particular, the Poisson algebra $R$ is a symmetric iterated Poisson-Ore extension in the sense explained in $\S$\ref{ssec:intro}.
\hfill $\diamond$
}
\end{remark}

\begin{definition}\label{de:presen}
{\rm Given a  $\TT$-Poisson algebra $P$ and a $\CC$-split torus $\TT^\prime$, by a {\it presentation of $P$ as a localized symmetric $\TT^\prime$-Poisson CGL
extension} we mean a triple
\[
\cP = (\cE, \, E,\, I),
\]
 where $\cE = (R, \cD)$ is a localized symmetric $\TT^\prime$-Poisson CGL
extension, $E: \TT \to \TT^\prime$ is an embedding of $\CC$-split tori, and $I: P \to R$ is an isomorphism of $\TT$-Poisson algebras, where
$\TT$ acts on $R$ through the embedding $E$ and the $\TT^\prime$-action on $R$.
We drop the adjective ``localized" when $(R, \cD)$ is a symmetric $\TT^\prime$-Poisson CGL extension,
and we often suppress the mention of
$\TT$ and $\TT^\prime$ and speak of the symmetric Poisson CGL extensions and presentations.
\hfill $\diamond$
}
\end{definition}

We will see in this paper (see, for example, Remark \ref{rk:many})
that a given $\TT$-Poisson algebra may have many different presentations as Poisson CGL extensions.

\begin{example}\label{ex:T-0}
{\rm
Consider the $\TT$-Poisson algebra $(\CC[\TT], 0)$, where $0$ stands for the zero Poisson bracket and $\TT$ acts on $\CC[\T]$ by translation.
Let $(\omega_1, \ldots, \omega_d)$ be
any basis of the character group $X(\T)$, and let $(\omega_1^*, \ldots, \omega_d^*)$ be
the basis of $\t$ dual to $(\omega_1, \ldots, \omega_d)$. Then
$((\CC[\omega_1, \ldots, \omega_d], 0), \cD)$ is a symmetric $\T$-Poisson CGL extension, where
\[
\cD = (\omega_1, \ldots, \omega_d, \; \omega_1^*, \ldots, \omega_d^*, \; -\omega_1^*, \ldots, -\omega_d^*),
\]
and $(\cE, {\rm Id}_\sTT, I)$ is a presentation of  $(\CC[\T], 0)$ as
a localized symmetric $\T$-Poisson CGL extension, where $\cE = ((\CC[\omega_1^{\pm 1}, \ldots, \omega_d^{\pm 1}], 0), \cD)$ and
$I: \CC[\T]\cong \CC[\omega_1^{\pm 1}, \ldots, \omega_d^{\pm 1}]$.
\hfill $\diamond$
}
\end{example}

Recall that an affine $\TT$-Poisson variety is a Poisson variety $(X, \piX)$ together with an algebraic action by $\TT$ preserving $\piX$.  In such a case,
$(\CC[X], \piX)$ becomes a $\TT$-Poisson algebra with the induced $\TT$-action defined at the end of Notation
\ref{nota:intro}.

\begin{definition}\label{de:Poi-Ore}
{\rm Let $(X, \piX)$ be an $n$-dimensional irreducible $\TT$-Poisson variety.

1) If $X^\prime \subset X$ is a $\TT$-invariant Zariski open subset and $\rho: Z \to X^\prime$ is a parametrization of $X^\prime$ by a
Zariski open subset $Z$
of $\CC^n$,  we say that
$\piX$ is {\it presented as a localized symmetric CGL extension in the coordinate chart $\rho$ (or via the parametrization $\rho: Z \to X^\prime$)}
if there exist a $\CC$-split torus $\TT^\prime$, an embedding $E: \TT \to \TT^\prime$, and $\TT^\prime$-action data $\cD$,
 such that $((\CC[Z], \cD), E, (\rho^*:\CC[X^\prime] \to \CC[Z]))$
is  a presentation of the $\TT$-Poisson algebra $(\CC[X^\prime], \piX)$ as a localized symmetric $\TT^\prime$-Poisson CGL extension.

2) By a {\it $\TT$-Poisson-Ore atlas} for $(X, \piX)$ we mean an atlas $\cA_\sX$ on $X$, consisting of $\TT$-invariant coordinate charts parametrized by
Zariski open subsets of $\CC^n$,  such that
$\piX$ is presented as a localized symmetric Poisson CGL extension in each of the coordinate charts
of $\cA_\sX$.
\hfill $\diamond$
}
\end{definition}

\begin{definition}\label{de:PoiOre-va}
{\rm
By a {\it $\TT$-Poisson-Ore variety}, we mean a triple $(X, \pi_\sX, \cA_\sX)$, where
$(X, \piX)$ is an irreducible rational $\TT$-Poisson variety,  and $\cA_{\sX}$ is a
{\it $\TT$-Poisson-Ore atlas} for $(X, \piX)$. A $\TT$-Poisson-Ore variety for some split $\CC$-torus $\TT$
is also simply referred to as a Poisson-Ore variety.
\hfill $\diamond$
}
\end{definition}

\subsection{Mixed products of symmetric Poisson CGL extensions}\label{ssec:mixed}
Given $\TT_i$-Poisson algebras $(R_i, \{\,, \, \}_i)$ for $i = 1, 2$, and given any
\[
\nu = \sum_q a_q \otimes b_q \in \t_1 \otimes \t_2,
\]
where $\t_1$ and $\t_2$ are the respective Lie algebras of $\TT_1$ and $\TT_2$,
define the Poisson bracket $\{\,, \, \}_\nu$ on the tensor product algebra $R_1 \otimes_\CC R_2$ via
\begin{align}\nonumber
\{r_1, \; r_1^\prime\}_\nu  &= \{r_1, \, r_1^\prime\}_1, \hs \{r_2, \; r_2^\prime\}_\nu  =
\{r_2, \, r_2^\prime\}_2, \hs r_1, r_1^\prime \in R_1, \; r_2, r_2^\prime \in R_2,\\
\label{eq:mixed-mu}
\{r_1, \; r_2\}_\nu & = -\sum_q a_q(r_1) b_q(r_2), \hs r_1 \in R_1, \; r_2 \in R_2.
\end{align}
Then $(R_1 \otimes_\CC R_2, \{\,, \, \}_\nu)$ is a $(\TT_1 \times \TT_2)$-Poisson algebra with the product $(\TT_1 \times \TT_2)$-action.
Suppose now that
\[
\cE_1 = ((\CC[z_1, \ldots, z_n], \{\,, \, \}_1), \,\cD_1) \hs \mbox{and} \hs
\cE_2 = ((\CC[z_{n+1}, \ldots, z_{n+m}], \{\,, \, \}_2), \,\cD_2)
\]
are symmetric $\TT_1$- and $\TT_2$-Poisson CGL extensions with respective action data
\begin{align*}\label{eq:A12}
\cD_1 &= (\chi_1, \ldots, \chi_n, \; h_1, \ldots, h_n,\; \;  h_1^\prime, \ldots, h_n^\prime),\\
 \cD_2 & = (\chi_{n+1}, \ldots, \chi_{n+m}, h_{n+1}, \ldots, h_{n+m}, \; h_{n+1}^\prime, \ldots, h_{n+m}^\prime).
\end{align*}
Identify
$\CC[z_1, \ldots, z_n] \otimes_\CC  \CC[z_{n+1}, \ldots, z_{n+m}] \cong \CC[z_1, \ldots, z_{n+m}]$,  and  for $j \in [1, n+m]$, define
$\hat{\chi}_j\in X(\TT_1 \times \TT_2)$ and  $\hat{h}_j, \hat{h}_j^\prime \in \t_1 \oplus \t_2$  by
\begin{align*}
&\hat{\chi}_j = (\chi_j, \; 0) \;\;\mbox{for} \;\; j \in [1, n] \;\;\mbox{and} \;\; \hat{\chi}_j = (0, \, \chi_j) \;\; \mbox{for}\;\; j \in [n+1, n+m],\\
&\hat{h}_j = (h_j, \; \nu^\#(\chi_j)), \hs \hat{h}_j^\prime = (h_j^\prime, \; -\nu^\#(\chi_j)), \hs j \in [1, n],\\
&\hat{h}_j = ((\nu^{21})^\#(\chi_j), \; h_j) , \hs  \hat{h}_j^\prime = (-(\nu^{21})^\#(\chi_j), \; h_j^\prime), \hs j \in [n+1, n+m],
\end{align*}
where $\nu^\#: \t_1^* \to \t_2$ and $(\nu^{21})^\#: \t_2^* \to \t_1$ are respectively defined by
\[
\nu^\#(\xi_1) = \sum_q \xi_1(a_q) b_q, \hs (\nu^{21})^\#(\xi_2) = \sum_q \xi_2(b_q) a_q, \hs \xi_1 \in \t_1^*,
\; \xi_2 \in \t_2^*.
\]
The proof of the following proposition is straightforward and is omitted.

\begin{lemma}\label{le:product}
The pair $\cE_1 \otimes_\nu \cE_2 := ((\CC[z_1, \ldots, z_{n+m}],\, \{\, ,\, \}_\nu), \cD)$, where
\[
\cD = \left(\hat{\chi}_1, \ldots, \hat{\chi}_{n+m}, \; \hat{h}_1, \ldots, \hat{h}_{n+m}, \; \hat{h}^\prime_1, \ldots, \hat{h}^\prime_{n+m}\right)
\]
is a symmetric $(\TT_1 \times \TT_2)$-Poisson CGL extension. We call $\cE_1 \otimes_\nu \cE_2$
the {\it mixed product} of $\cE_1$ and $\cE_2$ defined by $\nu$.
\end{lemma}

Rephrasing in geometrical terms, if for $i = 1, 2$, $(X_i, \pi_i)$ is a $\TT_i$-Poisson variety, then
each $\nu = \sum_q a_q \otimes b_q \in \t_1 \otimes \t_2$   gives rise to
the Poisson bi-vector field
\begin{equation}\label{eq:pinu}
\pi_\nu = (\pi_1, 0) + (0, \pi_2) -\sum_q (\rho_1(a_q), 0) \wedge (0,  \rho_2(b_q))
\end{equation}
on $X_1 \times X_2$,  where $\rho_i$, for $i = 1, 2$, is the induced Lie algebra action of $\t_i$ on $X_i$, i.e., for $a \in \t_i$,
$\rho_i(a)$ is the vector field on $X$ given by $\rho_i(a)(x) = \frac{d}{d\lambda}|_{\lambda = 0} \exp (\lambda a)x$ for $x \in X$.

\begin{definition}\label{de:pi-product}
{\rm The Poisson structure $\pi_\nu$ on $X_1 \times X_2$ will be called the
{\it mixed product} of the Poisson structures $\pi_1$ and $\pi_2$ defined by $\nu$.
\hfill $\diamond$
}
\end{definition}

See \cite{Lu-Mou:mixed} for
a more general definition and construction of
mixed product Poisson structures associated to Poisson Lie groups.

\subsection{Completeness of Hamiltonian flows of symmetric CGL generators}\label{ssec:hamil}
We start by recalling a definition from \cite{Lu-Mi:Kostant}.
Let $\calQ$ be the algebra of all quasi-polynomials in one complex variable \cite[$\S$26]{Arnold:ODE}, i.e.,
all holomorphic functions on $\CC$ of the form
\begin{equation}\label{eq:gamma-t}
\gamma(c) = \sum_{k=1}^N q_k(c)e^{a_kc}, \hs c \in \CC,
\end{equation}
 where each $q_k(c) \in \CC[c]$ and the $a_k$'s pairwise distinct complex numbers. A holomorphic map $\gamma: \CC \to \CC^n$ is said to have {\it Property
$\calQ$} if each of its components is in $\calQ$.

For a smooth affine complex variety $X$ and a holomorphic curve $\gamma: \CC \to X$ in $X$, we say that $\gamma$ has
{\it Property $\calQ$} if there exists an embedding $I: X \hookrightarrow \CC^n$ of $X$ as an affine subvariety of $\CC^n$
such that $I \circ \gamma: \CC \to\CC^n$ has Property $\calQ$.
It is proved in
\cite[Lemma 1.5]{Lu-Mi:Kostant} that if a holomorphic curve $\gamma$ in $X$ has Property $\calQ$, then
$I^\prime \circ \gamma: \CC \to \CC^{n'}$ has Property $\calQ$ for all affine embeddings $I^\prime: X \hookrightarrow \CC^{n'}$.

\begin{definition}\label{de:calQ}\cite{Lu-Mi:Kostant}
{\rm
For a smooth affine complex Poisson variety $(X, \piX)$ and  $f \in \CC[X]$,  we say that $f$ has  {\it complete Hamiltonian flow with Property $\calQ$}
if all the integral curves of the Hamiltonian vector field of $f$ are defined over $\CC$ and have Property $\calQ$.
\hfill $\diamond$
}
\end{definition}

Using the elementary fact that all solution curves of the ODE $dx/dt = a_0 x + b(t)$, where $a_0 \in \CC$ and $b(t) \in \calQ$,
are in $\calQ$,  the following facts are proved in \cite{Lu-Mi:Kostant}.

\begin{lemma}\label{le:hamil}\cite[Lemma 1.2 and Lemma 1.4]{Lu-Mi:Kostant}.
Let $\pi$ be an algebraic Poisson structure on $\CC^n$ and $f \in \CC[\CC^n]$. If there exist coordinates
$(x_1, \ldots, x_n)$ on $\CC^n$ such that for each $j \in [1, n]$ there are $a_j \in \CC$ and $b_j \in \CC[x_1, \ldots, x_{j-1}]$ such that
\[
\{f, x_j\} = a_j x_j f + b_j,
\]
then $f$ has complete Hamiltonian flows in $\CC^n$ with Property $\calQ$. Furthermore, if
a non-zero $g \in \CC[\CC^n]$ is such that $\{f, g\} = afg$ for some
$a \in \CC$, then $f$ has complete Hamiltonian flows in $X_g:=\{x \in \CC^n: g(x) \neq 0\}$.
\end{lemma}

\begin{proposition}\label{pr:hamil}
Let $\cE = (R, \cD)$ be a localized symmetric Poisson CGL extension as in Definition \ref{de:Poi-CGL}, where
$R = (\CC[z_1, \ldots, z_n][y_1^{-1}, \ldots,y_k^{-1}],  \{\, , \, \})$.
Let $X\subset \CC^n$ be defined by $y_1y_2 \cdots y_k \neq 0$. Then for each $j \in [1, n]$,  $z_j$  has complete Hamiltonian flow with Property $\calQ$ in
$X$.
\end{proposition}

\begin{proof}
Fix $j \in [1, n]$.  Taking $f = z_j$ and $(x_1, \ldots, x_n) = (z_{j-1}, \ldots, z_1, z_{j+1}, z_{j+2}, \ldots, z_n, z_j)$ as the new set of coordinates for
$\CC^n$,
it follows from \eqref{eq:xjk} that the assumptions in Lemma \ref{le:hamil} on $f$ and on $(x_1, \ldots, x_n)$ are satisfied. Moreover,
by \cite[Corollary 5.10]{GY:Poi}, each $y_i$ for $i \in [1, k]$, being a
homogeneous Poisson prime element, has log-canonical Poisson bracket with $z_j$. Thus $g = y_1y_2 \cdots y_k$ has log-canonical Poisson bracket
with $z_j$.  By Lemma \ref{le:hamil} that $z_j$ has complete Hamiltonian flow with Property $\calQ$ in
$X$.
\end{proof}

The following theorem is now an immediate consequence of Proposition \ref{pr:hamil} and the definition of Poisson-Ore varieties.

\begin{theorem}\label{th:hamil}
For any Poisson-Ore variety
$(X, \piX, \cA_\sX)$, all the coordinate functions in any coordinate chart $\rho: Z \to \rho(Z) \subset X$ in
$\cA_\sX$ have complete Hamiltonian flows with Property $\calQ$ in $\rho(Z)$.
\end{theorem}

\subsection{Symmetric Poisson CGL extensions from generalized Bruhat cells}\label{ssec:GBcs}
Continuing with the set-up in Notation \ref{nota:intro0} and
 Notation \ref{nota:intro}, we review in this section the standard multiplicative Poisson structure $\pist$ on $G$, the standard Poisson
structures on generalized Bruhat cells, and the associated symmetric Poisson CGL extensions .

\begin{notation}\label{nota:3}
{\rm
Recall that $\g$ and $\h$ are the respective Lie algebras of $G$ and $T$.
We fix a
symmetric and non-degenerate invariant bilinear form $\lara_\g$ on $\g$, and
denote by $\lara$ both the restriction of $\lara_\g$ to $\h$ and the induced bilinear form on $\h^*$.
Let again $d = \dim \h$, and let $\{H_q\}_{q=1}^d$ be any basis of $\h$ that is orthonormal with respect to $\lara$.
Let $\Delta^+ \subset \h^*$ be the set of all positive roots.
\hfill $\diamond$
}
\end{notation}

Recall from Notation \ref{nota:intro} that for each simple root $\alpha$, we have fixed root vectors $e_\alpha$ for $\al$ and $e_{-\alpha}$ for
$-\alpha$ such that $\alpha([e_\al, e_{-\al}]) = 2$. Extend the choices of such root vectors $e_\al$ and $e_{-\al}$ for all positive roots $\alpha$.
It is then easy to see that $\la e_\al, e_{-\al}\ra_\g = \frac{2}{\la \al, \al \ra}$ for $\al \in \Delta^+$.
The {\it standard quasi-triangular $r$-matrix} on $\g$  is given by \cite{dr:quantum}
\begin{equation}\label{eq:rst}
r_{\rm st} = \sum_{q=1}^d H_q \otimes H_q + \sum_{\al \in \Delta^+} \la \al, \al\ra(e_{-\al} \otimes e_\al) \in \g \otimes \g,
\end{equation}
which depends only on the choice of the triple $(B, T, \lara_\g)$ and not on that of the root vectors.
Let
\[
\Lambda_{\rm st} =  \sum_{\alpha \in \Delta^+} \frac{\la \al, \al \ra}{2} (e_{-\al} \otimes e_\al - e_\al \otimes e_{-\al}) \in \wedge^2 \g
\]
be the skew-symmetric part of $r_{\rm st}$.
Then the  {\it standard multiplicative
Poisson structure} $\pist$ on $G$ is  defined to be the Poisson bi-vector field on $G$ given by
\begin{equation}\label{eq:pist}
\pist= \Lambda_{\rm st}^L- \Lambda_{\rm st}^R,
\end{equation}
where $\Lambda_{\rm st}^L$ and $\Lambda_{\rm st}^R$ respectively denote the left and right invariant
bi-vector fields on $G$ with value $\Lambda_{\rm st}$ at the identity element of $G$. The Poisson Lie group $(G, \pist)$ is
the semi-classical limit of the quantum group $\CC_q[G]$ (see \cite{chari-pressley, etingof-schiffmann}).

It follows from the
definition that $\pist$  is invariant under both left and right translations by elements in $T$, and it
is well-known (see, for example, \cite{hodges, reshe-4, Kogan-Z}) that the $T$-orbits of symplectic leaves of $\pist$ are precisely the double Bruhat cells
\[
G^{u, v} =BuB \cap B^-vB^-, \hs u, v \in W.
\]
It follows that every $BwB$, where $w \in W$, is a Poisson submanifold of $(G, \pist)$.

For any integer $r \geq 1$, recall now from
\eqref{eq:Fn} the quotient variety $F_r$ of $G^r$ by $B^r$. Let again
$\varpi_r: G^r \to F_r$ be the projection.
It is shown in \cite[$\S$1.3]{Lu-Mou:flags} and \cite[$\S$7.1]{Lu-Mou:mixed} that
\begin{equation}\label{eq:pi-r}
\pi_r \; \stackrel{{\rm def}}{=} \; \varpi_r(\pi_{\rm st}^r)
\end{equation}
is a well-defined Poisson structure on  $F_r$, and that for any $\bfu = (u_1, \ldots, u_r) \in W^r$,
the generalized Bruhat cell $\O^\bfu \subset F_r$ is a Poisson submanifold with respect to the Poisson structure $\pi_r$. The restriction of $\pi_r$ to $\O^\bfu$
will still be denoted as $\pi_r$ and is called \cite{Elek-Lu:BS} the {\it standard Poisson structure} on $\O^\bfu$.
It is also clear that the $T$-action on $F_r$ defined in \eqref{eq:T-Fr} preserves the Poisson structure
$\pi_r$.

Let $\bfu = (u_1, \ldots, u_r) \in W^r$, and let $l = l(u_1) + \cdots + l(u_r)$.
Recall that associated to each
$\tilde{\bfu} = (\bfu_1,  \ldots, \bfu_r)  \in \cR(u_1) \times \cdots \times \cR(u_r)$ one has
the Bott-Samelson parametrization $\bbeta^{\tilde{\bfu}}: \CC^l \to \O^\bfu$ in \eqref{eq:beta-Ou}.
Let $\{\, ,\, \}_{\tilde{\bfu}}$ be the Poisson bracket on $\CC[z_1, \ldots, z_l]$ such that
\begin{equation}\label{eq:bracket-u}
(\bbeta^{\tilde{\bfu}})^*: \;\; (\CC[\O^\bfu], \;\pi_r) \longrightarrow (\CC[z_1, \ldots, z_l], \, \{\, , \, \}_{\tilde{\bfu}})
\end{equation}
is an isomorphism of Poisson algebras. Regard $(\CC[\O^\bfu], \pi_r)$ as a $T$-Poisson algebra,
where $T$ acts on $\CC[\O^\bfu]$ through  the
$T$-action on  $\O^\bfu$ in \eqref{eq:T-Fr}, i.e.
\[
t\cdot z_j = t^{s_{\al_1} s_{\al_2} \cdots s_{\al_{j-1}}(\al_j)} z_j, \hs t \in T, \; j \in [1, l],
\]
as in \eqref{eq:T-xi}.
For any $\chi \in \h^*$, let $\chi^\#$ be the unique element in $\h$ satisfying
\begin{equation}\label{eq:sharp}
\chi^\prime(\chi^\#) = \la \chi,  \; \chi^\prime \ra, \hs \chi^\prime \in \h^*.
\end{equation}
Set $\cD_{\tilde{\bfu}} = (\chi_1, \ldots, \chi_l, \; h_1, \ldots, h_l, \; h_1^\prime, \ldots, h_l^\prime)$, where for $j \in [1, l]$,
\begin{equation}\label{eq:chi-k-BS}
\chi_j = s_{\al_1} s_{\al_2} \cdots s_{\al_{j-1}}(\al_j) \in \h^*\hs \mbox{and} \hs
h_j = -h_j^\prime = s_{\al_1} s_{\al_2} \cdots s_{\al_{j-1}}(\al_j^\#) \in \h.
\end{equation}

\begin{theorem}\label{th:Elek-Lu} \cite{Elek-Lu:BS}
For any $\bfu = (u_1, \ldots, u_r) \in W^r$ and $\tilde{\bfu} \in \cR(u_1) \times \cdots \times \cR(u_r)$,
\begin{equation}\label{eq:E-bfu}
\cE_{\tilde{\bfu}} := \left((\CC[z_1, \ldots, z_l], \, \{\,, \, \}_{\tilde{\bfu}}), \, \cD_{\tilde{\bfu}}\right)
\end{equation}
is a symmetric $T$-Poisson CGL extension. In particular, for $1 \leq i < j \leq l$, the log-canonical term
$\{z_i, z_j\}_{\tilde{\bfu}, {\rm log-can}}$ of $\{z_i, z_j\}_{\tilde{\bfu}}$ is given by (see \eqref{eq:log-can})
\[
\{z_i, z_j\}_{\tilde{\bfu}, {\rm log-can}} = -\chi_i(h_j) z_iz_j = -\la s_{\al_1} \cdots s_{\al_{i-1}}(\al_i), \, s_{\al_1} \cdots s_{\al_{j-1}}(\al_j)\ra z_iz_j.
\]
\end{theorem}

\begin{definition}\label{de:cRu}
{\rm For $\bfu = (u_1, \ldots, u_r) \in W^r$ and  $\tilde{\bfu} \in \cR(u_1) \times \cdots \times \cR(u_r)$, set
\[
\cP_{\tilde{\bfu}} = (\cE_{\tilde{\bfu}}, \, {\rm Id}_\sT, \, (\bbeta^{\tilde{\bfu}})^*).
\]
We call $\cP_{\tilde{\bfu}}$ the {\it symmetric $T$-Poisson CGL presentation of
$(\CC[\O^\bfu], \pi_r)$ (in the Bott-Samelson coordinates or via the Bott-Samelson parametrization) defined by $\tilde{\bfu}$}.
\hfill $\diamond$
}
\end{definition}

\begin{remark}\label{rk:many}
{\rm
As an element in $W$ may have more then one reduced words, the $T$-Poisson algebra $(\CC[\O^\bfu], \pi_r)$ for $\bfu \in W^r$ in general has more then
one presentation as a symmetric Poisson CGL extension.
\hfill $\diamond$
}
\end{remark}

\begin{remark}\label{re:fjk-explicit}
{\rm By Theorem \ref{th:Elek-Lu}, the Poisson bracket $\{\,, \, \}_{\tilde{\bfu}}$ has the form
\[
\{z_i, z_j\}_{\tilde{\bfu}} = -\chi_i(h_j)z_iz_j -f_{i,j}, \hs 1 \leq i < j \leq l,
\]
where $f_{i,j} \in \CC[z_{i+1}, \ldots, z_{j-1}]$ for all $1 \leq i < j \leq l$.
Explicit formulas for the polynomials $f_{i,j}$
are given in \cite{Elek-Lu:BS} in terms of root strings and the structure constants
of the Lie algebra $\g$ in any Chevalley basis. We refer to \cite{Elek-Lu:BS} for detail.
When $r = 1$, Theorem \ref{th:Elek-Lu} is the Poisson analog of the
Levendorskii-Soibelman straightening law for the quantum Schubert cell corresponding to $u_1 \in W$ (see \cite[$\S$9.2]{GY:AMSbook}
and  \cite[I.6.10]{Brown-Goodearl:book}).
\hfill $\diamond$
}
\end{remark}


\section{The Bott-Samelson atlas is a Poisson-Ore atlas}\label{sec:BG-CGL}
Continuing with the setting from $\S$\ref{ssec:GBcs},  we first
review in $\S$\ref{ssec:piGQ} the definition of the Poisson structure $\pi_{\sG/\sQ}$ for $Q = B(v)$ or $N(v)$ with $v \in W$.
We then prove in $\S$\ref{ssec:JwQ-poi} that for each $w \in W$, the decomposition $J^w_\sQ$ of $wB^-B/Q$, given in
\eqref{eq:Jw-Bv-0} and \eqref{eq:Jw-Nv-0}, identifies the restriction of $\pi_{\sG/\sQ}$ to $wB^-B/Q$ with a mixed product of $\pi_1$
on $\O^{w_0w^{-1}}$ and $\pi_2$ on  $\O^{(w, v)}$ (and the zero Poisson structure on $T$ when $Q = N(v)$).
Using the presentations of $\pi_1$ and $\pi_2$ as symmetric Poisson CGL extensions via Bott-Samelson parametrizations and the
mixed product construction in Lemma \ref{le:product}, we obtain presentations of $\pi_{\sG/\sQ}$ as symmetric (or localized symmetric)
Poisson CGL extensions
in every Bott-Samelson coordinate chart on $G/Q$,
thereby proving the
 Theorem B stated in $\S$\ref{ssec:intro}. Details on the (localized)
symmetric Poisson CGL presentations of $\pi_{\sG/\sQ}$ in the Bott-Samelson  coordinate charts
are given in Theorem \ref{th:CGL-Bv} for $Q = B(v)$
and in Theorem \ref{th:CGL-Nv} for $Q = N(v)$.

\subsection{The Poisson structure $\pi_{\sG/\sQ}$}\label{ssec:piGQ}
Given a Poisson Lie group $(L, \pi)$, recall that a Lie subgroup $M$ of $L$ is called a {\it coisotropic subgroup} of $(L, \pi)$ if
$M$ is also a coisotropic submanifold of $L$ with respect to $\pi$,
i.e., if $\pi(x) \in T_x L \otimes T_xM + T_xM \otimes T_xL$ for all $x \in M$.  It is easy to see from the definition that
when $M$ is a closed coisotropic subgroup of a Poisson Lie group $(L, \pi)$, the Poisson structure $\pi$ on $L$ projects to a well-defined
Poisson structure on $L/M$, called the quotient Poisson structure of $\pi$.

Returning to the Poisson Lie group $(G, \pist)$ with $\pist$ defined in \eqref{eq:pist},
an argument similarly to that used in the proof of
\cite[Lemma 10]{Lu-Mou:groupoid} shows that for any $v \in W$,  $N(v)$
is a coisotropic subgroup of $(G, \pist)$.
Since the Poisson structure  $\pist$ is invariant under translation by elements in $T$, $B(v)$
is also a coisotropic subgroup of $(G, \pist)$.

\begin{notation}\label{nota:pi-GQ}
{\rm
For $v \in W$ and $Q = B(v)$ or $N(v)$, denote by $\pi_{\sG/\sQ}$ the Poisson structure on $G/Q$ which is the projection to $G/Q$ of the
Poisson structure $\pist$ on $G$. The restriction of $\pi_{\sG/\sQ}$ to each shifted big cell $wB^-B/Q$ is also denoted by $\pi_{\sG/\sQ}$.
Note that $\pi_{\sG/\sB} = \pi_1$ in the notation of \eqref{eq:pi-r}.
}
\end{notation}

\subsection{The decompositions $J^w_\sQ$ as Poisson maps}\label{ssec:JwQ-poi}
For $v, w \in W$, recall from $\S$\ref{ssec:ABS} the isomorphisms
\begin{align*}
J^w_{\sB(v)}:\;\; & \; wB^-B/B(v) \; \stackrel{\sim}{\longrightarrow}\; \O^{w_0w^{-1}} \times \O^{(w, v)},\\
J^w_{\sN(v)}: \;\;& \; wB^-B/N(v) \; \stackrel{\sim}{\longrightarrow}\; \O^{w_0w^{-1}} \times \O^{(w, v)} \times T.
\end{align*}
We now identify the respective Poisson structures $J^w_{\sB(v)}(\pi_{\sG/\sB(v)})$ and $J^w_{\sN(v)}(\pi_{\sG/\sN(v)})$ on
$\O^{w_0w^{-1}} \times \O^{(w, v)}$ and on $\O^{w_0w^{-1}} \times \O^{(w, v)} \times T$ as mixed products of the standard Poisson structures
$\pi_1$ on $\O^{w_0w^{-1}}$ and $\pi_2$ on $\O^{(w, v)}$ (and the zero Poisson structure on $T$ when $Q = N(v)$).
To this end,
for $x \in \h$, denote by $\rho_r(x)$ the vector field on $F_r$ generated by the $T$-action on $F_r$ in \eqref{eq:T-Fr} in the direction of $x$, i.e.,
\begin{equation}\label{eq:rho-r}
\rho_r(x)(p) = \frac{d}{d\lambda}|_{\lambda=0} \exp(\lambda x) \cdot p, \hs p \in F_r.
\end{equation}
The restriction of $\rho_r(x)$ to $\O^\bfu$, for $x \in \h$,
will also be denoted by $\rho_r(x)$.
For $x \in \h$, let $x^R$ be the right-invariant  (also left-invariant) vector field on $T$ defined by $x$.

\begin{notation}\label{nota:pi}
{\rm
Let again $d = \dim_\CC T$ and recall that
$\{H_q: q \in [1, d]\}$ is
a basis of $\h$ orthonormal with respect to the
bilinear form $\lara$. Define  the bi-vector fields
\begin{align}\label{eq:pi-12}
&\pi_{1, 2} = (\pi_1,  \,0) + (0, \, \pi_2) + \mu \;\;\; \mbox{on} \;\;\;\O^{w_0w^{-1}} \times \O^{(w, v)}, \;\;\;\mbox{and}\\
\label{eq:pi-120}
&\pi_{1, 2, 0} = (\pi_1, \, 0, \,0) + (0, \, \pi_2, \, 0) + \mu_{12} + \mu_{13} + \mu_{23} \;\;\;\mbox{on} \;\;\;\O^{w_0w^{-1}} \times \O^{(w, v)} \times T,
\end{align}
where the {\it mixed terms} $\mu$ and $\mu_{12}, \mu_{13}, \mu_{23}$ are respectively given by
\begin{align}\label{eq:mu-de-0}
\mu & = \sum_{q=1}^d (\rho_1(w_0w^{-1}(H_q)), \; 0) \wedge (0,  \rho_2(H_q)),\\
\label{eq:mu-de-1}
\mu_{12} & =\sum_{q=1}^d (\rho_1(w_0w^{-1}(H_q)),  0,  0) \wedge (0,  \rho_2(H_q),  0),\\
\label{eq:mu-de-2}
\mu_{13} & =\sum_{q=1}^d (\rho_1(w_0(H_q)),  0,  0) \wedge (0,  0, H_q^R),\\
\label{eq:mu-de-3}
\mu_{23} & =-\sum_{q=1}^d (0, \rho_2(w(H_q)),  0) \wedge (0,  0, H_q^R).
\end{align}
}
\end{notation}

 \begin{theorem}\label{th:JwQ-poi}
For any $w, v \in W$, the maps
\begin{align*}
J^w_{\sB(v)}: \;\;& \; (wB^-B/B(v), \; \pi_{\sG/\sB(v)}) \longrightarrow \left(\O^{w_0w^{-1}} \times \O^{(w, v)}, \;\pi_{1,2}\right),\\
J^w_{\sN(v)}: \;\;& \; (wB^-B/N(v), \; \pi_{\sG/\sN(v)}) \longrightarrow \left(\O^{w_0w^{-1}} \times \O^{(w, v)} \times T, \;
\pi_{1,2,0}\right),
\end{align*}
are Poisson isomorphisms.
\end{theorem}

Since the proof of Theorem \ref{th:JwQ-poi} is self-contained and can be
read independently of the rest of the paper, we present the proof in the Appendix in order not to disrupt the flow of the paper.

\begin{remark}\label{re:pi123-mix}
{\rm
Regard $(\O^{w_0w^{-1}}, \pi_1)$ and $(\O^{(w, v)}, \pi_2)$ as $T$-Poisson varieties, where
again $T$ act on $\O^{w_0w^{-1}} \subset F_1$ and on $\O^{(w, v)}\subset F_2$ via \eqref{eq:T-Fr}. Theorem \ref{th:JwQ-poi} then says that
$\pi_{1, 2}$ is the mixed product  (in the sense of Definition \ref{de:pi-product}) of $\pi_1$ and $\pi_2$ defined by
\[
\nu = -\sum_{q=1}^d w_0w^{-1}(H_q) \otimes H_q \in \h \otimes \h.
\]
Similarly, regard $(\O^{w_0w^{-1}} \times \O^{(w, v)},\pi_{1, 2})$ as a $(T \times T)$-Poisson  variety with the product $(T \times T)$-action, and
regard $(T, 0)$ as a $T$-variety with the left $T$-action by translation. Then $\pi_{1, 2, 0}$ is the mixed product of $\pi_{1, 2}$ and the zero Poisson structure on
$T$ defined by
\[
\nu^\prime = -\sum_{q=1}^d (w_0(H_q), \, -w(H_q)) \otimes H_q \in (\h \oplus \h) \otimes \h.
\]
\hfill $\diamond$
}
\end{remark}

\subsection{Symmetric Poisson CGL presentations of $\pi_{\sG/\sQ}$ in Bott-Samelson coordinate charts}\label{ssec:proof-main}
Let $v \in W$ and $Q = B(v)$ or $N(v)$.  For each $w \in W$,  we
regard $\left(\CC[wB^-B/Q], \,  \pi_{\sG/\sQ}\right)$ as a $T$-Poisson algebra with the $T$-action given by
\begin{equation}\label{eq:T-Q}
(t \cdot \phi)(g_\cdot Q) = \phi(tg_\cdot Q), \hs t \in T, \, \phi \in \CC[wB^-B/Q], \, g \in wB^-B.
\end{equation}
We now describe  symmetric Poisson CGL presentations of
$\left(\CC[wB^-B/Q], \,  \pi_{\sG/\sQ}\right)$ in the Bott-Samelson coordinates on $wB^-B/Q$.
For the convenience of the reader, we recall part of Notation \ref{nota:br}.

\begin{notation}\label{nota:bra1}
{\rm
Let $l_0 = l(w_0)$, and for $v, w \in W$, let $k = l_0-l(w)$ and  $l = l_0 + l(v)$. Write an element
$\br  \in \cR(w_0w^{-1}) \times \cR(w) \times \cR(v)$ as $\br= (\bfw^0, \bfw, \bfv)$ with
\[
\bfw^0 = (s_{\al_1}, \, \ldots, \, s_{\al_k}), \;\;\; \bfw = (s_{\al_{k+1}}, \, \ldots, s_{\al_{l_0}}), \;\;\; \bfv = (s_{\al_{l_0+1}}, \, \ldots, \, s_{\al_{l}}).
\]
For each  $j \in [1, l]$, set
\begin{align}\label{eq:chi-hjl1}
&\chi_j = s_{\al_1}  \cdots s_{\al_{j-1}}(\al_j), \;\; \;\; h_j = -h_j^\prime =  s_{\al_1}  \cdots s_{\al_{j-1}}(\al_j^\#),\;\;\mbox{if}\;\; j \in [1, k],\\
\label{eq:chi-hjl2}
&\chi_j = s_{\al_{k+1}}  \cdots s_{\al_{j-1}}(\al_j), \;\; \;\;
h_j = -h_j^\prime = s_{\al_{k+1}}  \cdots s_{\al_{j-1}}(\al_j^\#), \;\; \mbox{if}\;\; j \in [k+1, l].
\end{align}
\hfill $\diamond$
}
\end{notation}

\noindent
{\bf The case of $B(v)$}. Let $w, v \in W$ and let $\br =  (\bfw^0, \bfw, \bfv) \in \cR(w_0w^{-1}) \times \cR(w) \times \cR(v)$ be as in Notation \ref{nota:bra1}.
Recall from \eqref{eq:bga-1} and \eqref{eq:bga-Bv} the parametrizations
\begin{align*}
&\bsigma^\br: \;\; \CC^l \longrightarrow \O^{w_0w^{-1}} \times \O^{(w, v)}, \\
&\bsigma^{\br}_{\sB(v)} = (J^w_{\sB(v)})^{-1} \circ \bsigma^{\br}: \;\; \CC^l \longrightarrow wB^-B/B(v).
\end{align*}
Let $\{\, , \, \}_{(\bfw^0|\bfw, \bfv)}$ be the Poisson bracket on $\CC[z_1,  \ldots, z_l]$ such that
\begin{equation}\label{eq:brhos}
\left(\bsigma^\br\right)^*:\;\; \left(\CC[\O^{w_0w^{-1}} \times \O^{(w, v)}], \,\pi_{1, 2}\right) \longrightarrow
\left(\CC[z_1, \ldots, z_l], \, \{\, , \, \}_{(\bfw^0|\bfw, \bfv)}\right)
\end{equation}
is a Poisson isomorphism. By Theorem \ref{th:JwQ-poi}, we have the Poisson isomorphism
\begin{equation}\label{eq:brhos-Bv}
\left(\bsigma^\br_{\sB(v)}\right)^*:\;\; \left(\CC[wB^-B/B(v)], \pi_{\sG/\sB(v)}\right) \longrightarrow
\left(\CC[z_1, \ldots, z_l], \, \{\, , \, \}_{(\bfw^0|\bfw, \bfv)}\right).
\end{equation}
Set $\cD_{(\bfw^0|\bfw, \bfv)} = (\hat{\chi}_1, \ldots, \hat{\chi}_l, \; \hat{h}_1, \ldots, \hat{h}_l, \; \hat{h}_1^\prime, \ldots, \hat{h}_l^\prime)$, where
\begin{equation}\label{eq:D-Bv}
\hat{\chi}_j = \begin{cases}(\chi_j, \, 0), \; j \in [1, k],\\  (0, \, \chi_j), \; j \in [k+1, l], \end{cases}
\hat{h}_j =-\hat{h}_j^\prime =  \begin{cases}(h_j,  -ww_0(h_j)),\; j \in [1, k],\\
(-w_0w^{-1}(h_j), h_j), \; j \in [k+1, l],\end{cases}
\end{equation}
and $\chi_j \in \h^*$ and $h_j \in \h$ for $j \in [1, l]$ are given in \eqref{eq:chi-hjl1} and \eqref{eq:chi-hjl2}, and let
\[
\cE_{(\bfw^0|\bfw, \bfv)} :=\left(\left(\CC[z_1,  \ldots, z_l],  \{\, , \, \}_{(\bfw^0|\bfw, \bfv)}\right),  \cD_{(\bfw^0|\bfw, \bfv)}\right).
\]
Recall from Theorem \ref{th:Elek-Lu} that one has the symmetric $T$-Poisson CGL extensions
\[
\cE_{\bw^0}  =\left((\CC[z_1, \ldots, z_k], \{\, \, \}_{\bw^0}), \cD_{\bw^0}\right)\;\;\mbox{and} \;\;
\cE_{(\bw, \bfv)} = \left((\CC[z_{k+1}, \ldots, z_l], \{\, \, \}_{(\bw, \bfv)}), \cD_{(\bw, \bfv)}\right).
\]
One checks directly that $\cE_{(\bfw^0|\bfw, \bfv)}$ is the mixed product (see Lemma \ref{le:product})
of the symmetric $T$-Poisson CGL extensions $\cE_{\bfw^0}$ and $\cE_{(\bfw, \bfv)}$ defined by
\[
\nu = -\sum_{q=1}^d w_0w^{-1}(H_q) \otimes H_q \in \h \otimes \h.
\]

\begin{theorem}\label{th:CGL-Bv}
For any $w, v \in W$ and  $\br = (\bfw^0, \bfw, \bfv) \in \cR(w_0w^{-1}) \times \cR(w) \times \cR(v)$ as in Notation \ref{nota:bra1},
the triple
\[
\cP_{\sB(v)}^\br  := \left(\cE_{(\bfw^0|\bfw, \bfv)}, \, E_w, \; (\bsigma^\br_{\sB(v)})^*\right)
\]
is a presentation of the $T$-Poisson algebra $\left(\CC[wB^-B/B(v)], \;\pi_{\sG/\sB(v)}\right)$ as a symmetric $(T \times T)$-Poisson CGL extension, where
$E_w: T \to T \times T,
t \to (t^{ww_0}, t)$ for $t \in T$.
\end{theorem}

\begin{proof}
Equip $\O^{w_0w^{-1}} \times \O^{(w, v)}$ with the product $(T \times T)$-action, where $T$ acts on $\O^{w_0w^{-1}} \subset F_1$ and
on $\O^{(w, v)}\subset F_2$ via \eqref{eq:T-Fr},
and equip
$\CC[\O^{w_0w^{-1}} \times \O^{(w, v)}]$ with the induced $(T \times T)$-action (see end of Notation \ref{nota:intro}).
Let $T \times T$ act on $\CC[z_1,  \ldots, z_l]$ such that for $j \in [1, l]$, $z_j$  has $(T \times T)$-weight $\hat{\chi}_j$ as given in \eqref{eq:D-Bv},
and let $T$ act on  $\CC[z_1,  \ldots, z_l]$ through the embedding
$E_w: T \to T \times T$.
By \eqref{eq:T-xi},
$\left(\bsigma^\br\right)^*$ in \eqref{eq:brhos} is an isomorphism of $(T \times T)$-Poisson algebras, and
by Lemma \ref{le:JwQ-T} on the $T$-equivariance of the isomorphism $J^w_{\sB(v)}$,
$\left(\bsigma^\br_{\sB(v)}\right)^*$ in \eqref{eq:brhos-Bv} is an isomorphism of $T$-Poisson algebras.
As
\[
\left(\bsigma^\br_{\sB(v)}\right)^*=  (\bsigma^\br)^* \circ \left((J^w_{\sB(v)})^{-1}\right)^*,
\]
the assertion on $\cP_{\sB(v)}^\br$ now follows Theorem \ref{th:JwQ-poi}.
\end{proof}

\begin{remark}\label{re:Ou-cut}
{\rm For each $\br = (\bfw^0, \bfw, \bfv) \in \cR(w_0w^{-1}) \times \cR(w) \times \cR(v)$, recall from
\eqref{eq:bracket-u} the Poisson bracket
$\{\,, \,\}_{(\bfw^0,\bfw, \bfv)}$ on $\CC[z_1, \ldots, z_l] \cong \CC[\O^{(w_0w^{-1}, w, v)}]$. Using the definition of the Poisson
structure $\pi_{1, 2}$ and Theorem \ref{th:Elek-Lu}, it is easy to see that
\[
\{z_i, \; z_j\}_{(\bfw^0|\bfw, \bfv)} = \{z_i, \; z_j\}_{(\bfw^0,\bfw, \bfv)}  \hs \mbox{if} \hs i,j \in [1, k] \hs \mbox{or} \hs i, j \in [k+1, l],
\]
but for $i \in [1, k]$ and $j \in [k+1, l] $, $\{z_i,  z_j\}_{(\bfw^0|\bfw, \bfv)}$ is the {\it negative} of the log-canonical term of
$\{z_i,  z_j\}_{(\bfw^0,\bfw, \bfv)}$. For this reason, we may call $\{\,, \,\}_{(\bfw^0|\bfw, \bfv)}$ the {\it log-canonical cut} of
$\{\,, \,\}_{(\bfw^0,\bfw, \bfv)}$ at the $k$'s place with coefficient $-1$.
\hfill $\diamond$
}
\end{remark}

\medskip
\noindent
{\bf The case of $N(v)$}.
We now turn to the $T$-Poisson algebra $\left(\CC[wB^-B/N(v)], \pi_{\sG/\sN(v)}\right)$.
Recall again that $\omega_1, \ldots, \omega_d$ is a listing of the set of all fundamental weights and that
$z_j =\omega_{j-l} \in \CC[T]$ for $j \in [l+1, l+d]$.

Let $\{\,, \, \}_{(\bfw^0|\bfw, \bfv)}^{\bowtie 0}$ be the Poisson bracket on
$\CC[z_1, \ldots, z_l, z_{l+1}^{\pm 1}, \ldots, z_{l+d}^{\pm 1}]$ such that
\[
(\bsigma^\br \times \sigma)^*:  \left(\CC[\O^{w_0w^{-1}} \!\!\times\!\! \O^{(w, v)} \!\!\times T],  \pi_{1, 2,0}\right)\rightarrow
\left(\CC[z_1, \ldots, z_l, z_{l+1}^{\pm 1}, \ldots, z_{l+d}^{\pm 1}],  \{\,, \, \}_{(\bfw^0|\bfw, \bfv)}^{\bowtie 0}\right)
\]
is a Poisson algebra isomorphism. We first express $\{\,, \, \}_{(\bfw^0|\bfw, \bfv)}^{\bowtie 0}$ in terms of
$\{\,, \, \}_{(\bfw^0|\bfw, \bfv)}$.

\begin{lemma}\label{le:bowtie} One has $\{z_i, \, z_j\}_{(\bfw^0|\bfw, \bfv)}^{\bowtie 0} = \{z_i, \, z_j\}_{(\bfw^0|\bfw, \bfv)}$ for all $i, j \in [1, l]$, and
\[
\{z_i, \, z_j\}_{(\bfw^0|\bfw, \bfv)}^{\bowtie 0} = \begin{cases}
(w_0\omega_{j-l})(h_i) z_iz_j, & \hs i \in [1, k], \, j \in [l+1, l+d],\\
-(w_0\omega_{j-l})(h_i)z_iz_j, & \hs i \in [k+1, l], \, j \in [l+1, l+d],\\
0, & \hs i, j \in [l+1, l+d].\end{cases}
\]
\end{lemma}

\begin{proof}
Directly follows from the definition of $\pi_{1, 2,0}$ given in Theorem \ref{th:JwQ-poi}.
\end{proof}

In particular, $\{\, , \, \}_{(\bfw^0|\bfw, \bfv)}^{\bowtie 0}$ restricts to a Poisson bracket on $\CC[z_1,  \ldots, z_l, z_{l+1}, \ldots, z_{l+d}]$.
By Theorem \ref{th:JwQ-poi}, one has an isomorphism of Poisson algebras
\[
(\bsigma^\br_{\sN(v)})^*:\,
\left(\CC[wB^-B/N(v)],  \pi_{\sG/\sN(v)}\right) \longrightarrow
\left(\!\CC[z_1, \ldots, z_l, z_{l+1}^{\pm 1}, \ldots, z_{l+d}^{\pm 1}], \, \{\, , \, \}_{(\bfw^0|\bfw, \bfv)}^{\bowtie 0}\!\right).
\]
Let $\{\omega_1^*, \ldots, \omega_d^*\}$ be the basis of $\h$ dual to $\{\omega_1, \ldots, \omega_d\}$. With
notation as in Notation \ref{nota:bra1} and in particular $\chi_j \in \h^*$ and $h_j \in \h$ for $j \in [1, l]$ given in \eqref{eq:chi-hjl1} and \eqref{eq:chi-hjl2},
let $\cD_{((\bfw^0|\bfw, \bfv), \sT)} = (\tilde{\chi}_1, \ldots, \tilde{\chi}_{l+d},  \tilde{h}_1, \ldots, \tilde{h}_{l+d},
\tilde{h}_1^\prime, \ldots, \tilde{h}_{l+d}^\prime)$, where
\begin{align}\label{eq:D-Nv1}
&\tilde{\chi}_j = \begin{cases}(\chi_j,  0, 0), \;\; \;j \in [1, k], \\
(0,  \chi_j, 0), \;\; \;j \in [k+1, l],\\
(0, 0, \omega_{j-l}), \;\;\;j \in [l+1, l+d],\end{cases}\\
\label{eq:D-Nv2}
&\tilde{h}_j = -\tilde{h}_j^\prime = \begin{cases} (h_j,\,  -ww_0(h_j), \,-w_0(h_j)),  \;\;\; j \in [1, k], \\
(-w_0w^{-1}(h_j), \, h_j,\, w^{-1}( h_j)), \;\;\; j \in [k+1, l],\\
 \left(-w_0 (\omega_{j-l}^\#), \, w(\omega_{j-l}^\#),\, \omega_{j-l}^*\right), \;\;\; j \in [l+1, l+d].\end{cases}
\end{align}

\begin{theorem}\label{th:CGL-Nv}
For any $w, v \in W$ and  $\br = (\bfw^0, \bfw, \bfv) \in \cR(w_0w^{-1}) \times \cR(w) \times \cR(v)$,
\[
\cE_{((\bfw^0|\bfw, \bfv), \sT)}:=\left(\left(\CC[z_1,  \ldots, z_l, z_{l+1}, \ldots, z_{l+d}],\;  \{\, , \, \}_{(\bfw^0|\bfw, \bfv)}^{\bowtie 0}\right), \; \cD_{((\bfw^0|\bfw, \bfv), \sT)}\right).
\]
is a symmetric $(T \times T \times T)$-Poisson CGL extension, and
\[
\cP^\br_{\sN(v)} := \left(\cE_{((\bfw^0|\bfw, \bfv), \sT)}^\prime, \, E_{w}^\prime, \;  (\bsigma^\br_{\sN(v)})^*\right)
\]
is a presentation of the $T$-Poisson algebra $\left(\CC[wB^-B/N(v)], \;\pi_{\sG/\sN(v)}\right)$ as a localized symmetric $(T \times T \times T)$-Poisson CGL
extension, where
\[
\cE_{((\bfw^0|\bfw, \bfv), \sT)}^\prime =\left(\left(\CC[z_1,  \ldots, z_l, z_{l+1}^{\pm 1}, \ldots, z_{l+d}^{\pm 1}],\;
\{\, , \, \}_{(\bfw^0|\bfw, \bfv)}^{\bowtie 0}\right), \; \cD_{((\bfw^0|\bfw, \bfv), \sT)}\right)
\]
and $E_{w}^\prime$ is the embedding $T \to T \times T \times T$ given by
$t \to (t^{ww_0}, t, t^w)$ for $t \in T$.
\end{theorem}

\begin{proof} Let $T \times T\times T$ act on $\CC[z_1, \ldots, z_l, z_{l+1}^{\pm 1}, \ldots, z_{l+d}^{\pm 1}]$ such that $z_j$ has $(T \times T \times T)$-weight
$\tilde{\chi}_j$ as given in \eqref{eq:D-Nv1}.
By the $T$-equivariance of the isomorphism $J^w_{\sN(v)}$ given in Lemma \ref{le:JwQ-T} and by \eqref{eq:T-xi},
the isomorphism $(\bsigma^\br_{\sN(v)})^*$ is
$T$-equivariant, where $t \in T$ acts on $\CC[z_1, \ldots, z_l, z_{l+1}^{\pm 1}, \ldots, z_{l+d}^{\pm 1}]$
via the embedding $E_w^\prime: T \to T \times T\times T$.
One checks directly that
$\cE_{((\bfw^0|\bfw, \bfv), \sT)}$ is the  mixed product (see Definition \ref{de:pi-product}) defined by
\[
\nu^\prime = -\sum_{q=1}^d (w_0(H_q) , \, -w(H_q))\otimes H_q \in (\h\oplus \h) \otimes \h
\]
of the symmetric $(T \times T)$-Poisson CGL extension $\cE_{(\bfw^0|\bfw, \bfv)}$
 and the symmetric $T$-Poisson CGL extension $((\CC[z_{l+1}, \ldots, \omega_{l+d}], 0), \cD)$  in Example \ref{ex:T-0}.
The assertion on $\cP_{\sN(v)}^\br$ now follows from Theorem \ref{th:JwQ-poi}.
\end{proof}

\medskip
\begin{example}\label{ex:SL4-GB}
{\rm Let $G = SL(4, \CC)$ and choose again the Borel subgroups $B^-$ and $B$ to consist of lower triangular and upper triangular matrices respectively.
Let $\al_1, \al_2, \al_3$ be the standard listing of the simple roots and set $s_i = s_{\al_i}$, $i = 1, 2, 3$.
We consider two Bott-Samelson coordinate charts on $G/B$. The Poisson brackets in the coordinates  are computed using the computer program
in GAP language written by Blazs Elek.

Let first $w = e$ and $\br_1 = ((s_3,s_2,s_1,s_3,s_2,s_3), \emptyset, \emptyset)$.
 The corresponding Bott-Samelson parametrization
$\bsigma^{\br_1}:=\bsigma^{\br_1}_{\sG/\sB}: \CC^6 \to B^-B/B$ is given by
\[
\bsigma^{\br_1} (\xi_1, \ldots, \xi_6) =  \left(\begin{array}{cccc} 1 & 0 & 0 & 0\\ \xi_1 & 1 & 0 & 0\\
 \xi_1\xi_4-\xi_2 & \xi_4 & 1 & 0\\ \Delta_1 & \xi_4\xi_6-\xi_5 & \xi_6 & 1\end{array}\right). B \in B^-B/B,
\]
where $\Delta_1 = \xi_1\xi_4\xi_6-\xi_1\xi_5 -\xi_2\xi_6 + \xi_3$. In the
$(\xi_1, \ldots, \xi_6)$ coordinates, the Poisson structure $\pi_{\sG/\sB}$ is given by
\begin{align*}
&\{\xi_1,\xi_2\}=-\xi_1\xi_2,\;\;\;\;\; \{\xi_1,\xi_3\}=-\xi_1\xi_3,\;\;\;\;\; \{\xi_1,\xi_4\}=\xi_1\xi_4-2\xi_2,\\
& \{\xi_1,\xi_5\}=\xi_1\xi_5-2\xi_3,\;\;\;\;\; \{\xi_1,\xi_6\}=0,\;\;\;\;\; \{\xi_2,\xi_3\}=-\xi_2\xi_3,\\
&\{\xi_2,\xi_4\}=-\xi_2\xi_4,\;\;\;\;\; \{\xi_2,\xi_5\}=2\xi_3\xi_4,\;\;\;\;\; \{\xi_2,\xi_6\}=\xi_2\xi_6-2\xi_3,\\
&\{\xi_3,\xi_4\}=0,\;\;\;\;\; \{\xi_3,\xi_5\}=-\xi_3\xi_5,\;\;\;\;\; \{\xi_3,\xi_6\}=-\xi_3\xi_6,\\
&\{\xi_4,\xi_5\}=-\xi_4\xi_5,\;\;\;\;\; \{\xi_4,\xi_6\}=\xi_4\xi_6-2\xi_5,\;\;\;\;\; \{\xi_5,\xi_6\}=-2\xi_5\xi_6.
\end{align*}

 For another example,  take $w = w_0 =s_1s_2s_3s_1s_2s_1$ and $\br_2 = (\emptyset, (s_1, s_3, s_2, s_3, s_1, s_2), \emptyset)$.
The corresponding Bott-Samelson parametrization $\bsigma^{\br_2}:=\bsigma^{\br_2}_{\sG/\sB}: \CC^6 \to w_0B^-B/B = Bw_0B/B$
is given by
\[
\bsigma^{\br_2}(z_1, \ldots, z_6) =
\left(\begin{array}{cccc}
z_1z_5-z_3 & z_4-z_1z_6 & z_1 & -1\\ z_5 & -z_6 & 1 & 0 \\ z_2 & -1 & 0 & 0\\
1 & 0 & 0 &0\end{array}\right){}_\cdot B \in w_0B^-B/B.
\]
In the $(z_1, \ldots, z_6)$ coordinates, the Poisson structure $\pi_{\sG/\sB}$ is given by
\begin{align*}
&\{z_{1},z_{2}\} = 0, \hs
\{z_{1},z_{3}\} = -z_{1}z_{3}, \hs
\{z_{1},z_{4}\} = -z_{1}z_{4}, \\
&\{z_{1},z_{5}\} = z_{1}z_{5}-2z_{3},\hs
\{z_{1},z_{6}\} = z_{1}z_{6}-2z_{4}, \hs
\{z_{2},z_{3}\} = -z_{2}z_{3}, \\
&\{z_{2},z_{4}\} = z_{2}z_{4}-2z_{3}, \hs
\{z_{2},z_{5}\} = -z_{2}z_{5},\hs
\{z_{2},z_{6}\}= z_{2}z_{6}-2z_{5}, \\
&\{z_{3},z_{4}\} = -z_{3}z_{4}, \hs
\{z_{3},z_{5}\} = -z_{3}z_{5}, \hs
\{z_{3},z_{6}\} = -2z_{4}z_{5},\\
&\{z_{4},z_{5}\} = 0, \hs
\{z_{4},z_{6}\} = -z_{4}z_{6}, \hs
\{z_{5},z_{6}\} = -z_{5}z_{6}.\\
\end{align*}
The coordinates $(\xi_1, \ldots, \xi_6)$ on $B^-B/B$ and $(z_1, \ldots, z_6)$ on $\overline{w_0}B^-B/B$
are related by
\begin{align*}
&\xi_1 = \frac{z_5}{z_1z_5-z_3}, \hs \xi_2 = \frac{z_2z_6-z_5}{z_3z_6-z_4z_5}, \hs \xi_3 = \frac{1}{z_2z_4-z_3},\\
&\xi_4 = \frac{z_1z_2z_6-z_2z_4-z_1z_5+z_3}{z_3z_6-z_4z_5}, \hs
\xi_5 = \frac{z_1}{z_2z_4-z_3}, \hs \xi_6 = \frac{z_4}{z_2z_4-z_3},\\
&z_1 = \frac{\xi_5}{\xi_3}, \hs z_2 = \frac{\xi_1\xi_4-\xi_2}{ \xi_1\xi_4\xi_6-\xi_1\xi_5 -\xi_2\xi_6 + \xi_3}, \hs
z_3 = \frac{\xi_1\xi_5-\xi_3}{ \xi_3(\xi_1\xi_4\xi_6-\xi_1\xi_5 -\xi_2\xi_6 + \xi_3)}, \\
&z_4 = \frac{\xi_6}{\xi_3}, \hs z_5 = \frac{\xi_1}{ \xi_1\xi_4\xi_6-\xi_1\xi_5 -\xi_2\xi_6 + \xi_3}, \hs z_6 = \frac{\xi_2\xi_6-\xi_3}{\xi_2\xi_5-\xi_3\xi_4}.
\end{align*}
It is remarkable that such non-trivial changes of coordinates transform one symmetric Poisson CGL extension to another.
\hfill $\diamond$
}
\end{example}

\subsection{Proof of Theorem B and Theorem C}\label{ssec:proof-C}

\medskip
Theorem B stated in $\S$\ref{ssec:intro} is the combination of Theorem \ref{th:CGL-Bv} and
Theorem \ref{th:CGL-Bv}.

Theorem C stated in $\S$\ref{ssec:intro} follows immediately from Theorem B and Theorem \ref{th:hamil}.

\appendix
\addtocontents{toc}{\protect\setcounter{tocdepth}{1}}
\section{Proof of Theorem \ref{th:JwQ-poi} and $T$-leaves of $(G/Q, \pi_{\sG/\sQ})$}\label{sec:proof-JwQ-poi}
In this appendix, we first prove Theorem \ref{th:JwQ-poi} which says that for any $v, w \in W$,
\begin{align}\label{eq:JwBv-poi}
J^w_{\sB(v)}: \;\;& \; (wB^-B/B(v), \; \pi_{\sG/\sB(v)}) \longrightarrow \left(\O^{w_0w^{-1}} \times \O^{(w, v)}, \; \pi_{1, 2}\right),\\
\label{eq:JwNv-poi}
J^w_{\sN(v)}: \;\;& \; (wB^-B/N(v), \; \pi_{\sG/\sN(v)}) \longrightarrow \left(\O^{w_0w^{-1}} \times \O^{(w, v)} \times T, \;
\pi_{1, 2, 0}\right)
\end{align}
are Poisson isomorphisms, where $J^w_{\sB(v)}$ and $J^w_{\sN(v)}$ are given in \eqref{eq:Jw-Bv} and \eqref{eq:Jw-Nv}, and the Poisson structures
$\pi_{1, 2}$ and $\pi_{1, 2,0}$ are given in \eqref{eq:pi-12}  and \eqref{eq:pi-120}.
Here recall that for $Q = B(v)$ or $N(v)$, $\pi_{\sG/\sQ}$ is the projection to $G/Q$ of the standard  Poisson
structure $\pist$ on $G$.
In $\S$\ref{ssec:somefacts-pist}, we review some facts on the Poisson Lie group $(G, \pist)$.
In $\S$\ref{ssec:auxiliary-poi}, we prove certain maps involved in the definitions of $J^w_{\sB(v)}$ and $J^w_{\sN(v)}$ are Poisson,  and we use
these facts to prove Theorem \ref{th:JwQ-poi} in $\S$\ref{ssec:proof-JwQ-poi}. In $\S$\ref{ssec:D}, we determine the $T$-leaves of
$(G/Q, \pi_{\sG/\sQ})$.

\subsection{Some facts on the Poisson Lie group $(G, \pist)$}\label{ssec:somefacts-pist}
We first recall (see, for example, \cite{etingof-schiffmann, Lu-Mou:mixed}) that
given a Poisson Lie group $(L, \pi)$ and a Poisson manifold $(X, \piX)$, a left action
of $L$ on $X$ is said to be Poisson if the action map $(L, \pi) \times (X, \piX) \to (X, \piX)$ is Poisson.  A Poisson manifold $(X, \piX)$ is called a Poisson homogeneous space \cite{dr:homog} of a Poisson Lie group $(L, \pi)$ if
$(X, \piX)$ has a Poisson action by $(L, \pi)$ which is also transitive.

\begin{example}\label{ex:quotient}
 {\em
 If  $M$ is a closed coisotropic subgroup (see $\S$\ref{ssec:piGQ})
 of a Poisson Lie group $(L, \pi)$,
the action of $L$ on $L/M$ by left translation makes $(L/M, \pi_{\sL/\sM})$
a  Poisson
 homogeneous space of  $(L, \pi)$, where $\pi_{\sL/\sM}$ is the projection to $L/M$ of the Poisson structure $\pi$ on $L$.
 Note that
 as $\pi(e) = 0$, $\pi_{\sL/\sM}$ vanishes at $e_\cdot M \in L/M$.  In general,
 it is easy to see from the definitions that if $(X, \piX)$ is a Poisson homogeneous space of $(L, \pi)$ and if $x \in X$ is such that $\piX(x)=0$, then
 the stabilizer subgroup $L_x$ of $L$ at $x$ is a coisotropic subgroup of $(L, \pi)$, and the map
 \[
 (L/L_x, \pi_{\sL/{\sL}_x}) \longrightarrow (X, \piX)\;\; l \longmapsto lx, \hs l \in L,
 \]
 is a Poisson isomorphism.
 \hfill $\diamond$
 }
 \end{example}

Returning to the Poisson Lie group $(G, \pist)$, where $\pist$ is given in \eqref{eq:pist},  for any $v \in W$ and $Q = B(v)$ or $N(v)$,  the Poisson
manifold $(G/Q, \pi_{\sG/\sQ})$ is then a Poisson homogeneous space of $(G, \pist)$.

We now recall a Drinfeld double of the Poisson Lie group $(G, \pist)$:
associated to the standard quasi-triangular $r$-matrix $r_{\rm st} \in \g \otimes \g$ in \eqref{eq:rst},  one has the quasi-triangular
$r$-matrix $r_{\rm st}^{(2)}\in (\g \oplus \g)^{\otimes 2}$ for the direct product Lie algebra $\g \oplus \g$, given
\cite[$\S$6.1]{Lu-Mou:mixed} as
\begin{equation}\label{eq:rst-2}
r_{\rm st}^{(2)} = (r_{\rm st}, \; 0) -(0, \, r_{\rm st}^{21})
- \sum_{q=1}^d (H_q, 0) \wedge (0, H_q) -\sum_{\al\in \Delta^+} \la \al, \al \ra (e_\al, 0) \wedge (0, e_{-\al}),
\end{equation}
where $r_{\rm st}^{21} = \tau(r_{\rm st})$ with
$\tau(x \otimes y) = y \otimes x$ for $x, y \in \g$, and for any vector space $V$ and $u, v \in V$, we use the convention that
\[
u \wedge v = u\otimes v - v\otimes u \in V \otimes V.
\]
Let  $\Lambda_{\rm st}^{(2)} \in \wedge^2 (\g \oplus \g)$ be the skew-symmetric part of $r_{\rm st}^{(2)}$, and let
$\Pist$ be the multiplicative Poisson structure on the product group $G \times G$ given by
\begin{equation}\label{eq:de-PIst}
\Pist = \left(r_{\rm st}^{(2)}\right)^L - \left(r_{\rm st}^{(2)}\right)^R= \left(\Lambda_{\rm st}^{(2)}\right)^L- \left(\Lambda_{\rm st}^{(2)}\right)^R.
\end{equation}
Here, for $A \in  \g^{\otimes k}$, $A^L$ and $A^R$ respectively  denote the left and right invariant tensor fields on $G$ with value
$A$ at the identity of $G$.
It follows from the definitions that
\begin{equation}\label{eq:Pist-mu}
\Pist = (\pist, \, 0) + (0, \, \pist) + \mu_1 + \mu_2,
\end{equation}
where
\begin{align}\label{eq:mu-1}
\mu_1 &= \sum_{q=1}^d (H_q^R, \; 0) \wedge (0, \; H_q^R) +\sum_{\al\in \Delta^+} \la \al, \al\ra(e_\al^R, \, 0) \wedge (0, e_{-\al}^R),\\
\label{eq:mu-2}
\mu_2 & =-\sum_{q=1}^d (H_q^L, \; 0) \wedge (0, \; H_q^L) -\sum_{\al\in \Delta^+} \la \al, \al\ra(e_\al^L, \, 0) \wedge (0, e_{-\al}^L).
\end{align}
The Poisson Lie group $(G \times G, \, \Pist)$ is a Drinfeld double of the Poisson Lie group $(G, \pist)$ (see, for example,
\cite[paragraph after Example 6.11]{Lu-Mou:mixed}). In particular, the embedding
\[
(G, \, \pist) \hookrightarrow (G \times G, \; \Pist), \;\; g \longmapsto (g, g), \hs g \in G,
\]
is Poisson, and the projections $(G \times G, \Pist) \to (G, \pist)$ to both factors are Poisson.

It follows from \eqref{eq:Pist-mu}, \eqref{eq:mu-1} and \eqref{eq:mu-2} that $B \times B$ is a coisotropic subgroup of the Poisson Lie
group $(G \times G, \, \Pist)$. Let $\varpi$ be the projection
\[
\varpi: \;\; G \times G \longrightarrow (G \times G)/(B \times B) \cong G/B \times G/B,
\]
and let  $\Pi = \varpi(\Pist)$.  It follows from \eqref{eq:de-PIst} and  \eqref{eq:Pist-mu} that
\[
\Pi =- \varpi\left(\left(\Lambda_{\rm st}^{(2)}\right)^R\right) = (\pi_{\sG/\sB}, \, 0) + (0, \; \pi_{\sG/\sB})+ \varpi(\mu_1).
\]

Let $G_{\rm diag} = \{(g, g): g \in G\} \subset G \times G$ and consider the $G_{\rm diag}$-orbits in $G/B \times G/B$,
which are precisely of the form
\[
G_{\rm diag} (v) \; \stackrel{{\rm def}}{=}\;  G_{\rm diag}(e_\cdot B, \, \ov_\cdot B)\subset G/B \times G/B, \hs v \in W.
\]
Note that for $v \in W$, the stabilizer subgroup of $G \cong G_{\rm diag}$ at $(e_\cdot B, \ov_\cdot B) \in G/B \times G/B$ is precisely
$B(v) = B \cap \ov B \ov^{\, -1}$.

\begin{lemma}\label{le:Gv}
For each $v \in W$, $G_{\rm diag} (v)$ is a Poisson submanifold of $G/B \times G/B$ with respect to $\Pi$, and the $G$-equivariant map
\begin{equation}\label{eq:GBvG}
(G/B(v), \; \pi_{\sG/{\sB}(v)}) \longrightarrow (G_{\rm diag} (v), \, \Pi),\;\; g_\cdot B(v) \longmapsto (g_\cdot B, \; g\ov_\cdot B), \hs g \in G,
\end{equation}
is a Poisson isomorphism.
\end{lemma}

\begin{proof} Let $v \in W$.
By \cite[Theorem 2.3]{LY:DQ},  $G_{\rm diag} (v)$ is a Poisson submanifold of $G/B \times G/B$ with respect to $\Pi$, and,
as a $G \cong G_{\rm diag}$-orbit, $(G_{\rm diag} (v), \Pi)$ is a Poisson
homogeneous space of $(G, \pist)$.
It is also easy to see that $\pi_{\sG/\sB}(\ov_\cdot B) = 0$ and $\varpi(\mu_1)(e_\cdot B, \ov_\cdot B) = 0$. It follows that
$\Pi(e_\cdot B, \ov_\cdot B) = 0$. By Example \ref{ex:quotient}, the map in \eqref{eq:GBvG} is an isomorphism of Poisson homogeneous spaces of
$(G, \pist)$.
\end{proof}

Recall the Poisson manifold $(F_2, \pi_2)$ from $\S$\ref{ssec:GBcs}.
By \cite[Theorem 7.8]{Lu-Mou:mixed} (see also \cite[Proposition 5.6]{Lu-Mou:flags}), the map
\begin{equation}\label{eq:J2}
J_2: \;\; (F_2, \, \pi_2) \longrightarrow (G/B \times G/B, \; \Pi), \;\; [g_1, g_2]_{F_2} \longmapsto ({g_1}_\cdot B, \, {g_1g_2}_\cdot B), \;\; g_1, g_2 \in G,
\end{equation}
is a Poisson isomorphism, with
\begin{equation}\label{eq:J-2-inverse}
J_2^{-1}: \;\; (G/B \times G/B, \; \Pi) \longrightarrow (F_2, \pi_2), \;\; ({h_1}_\cdot B, \, {h_2}_\cdot B) \longmapsto [h_1, \, h_1^{-1}h_2]_{\sF_2}.
\end{equation}
Note that $J_2$ is $G$-equivariant if $F_2$ is given the $G$-action by
\begin{equation}\label{eq:G-F2}
g_\cdot [g_1, \, g_2]_{F_2} = [gg_1, \; g_2]_{F_2}, \hs g, g_1, g_2 \in G.
\end{equation}
By Lemma \ref{le:Gv}, we have the following interpretation of the Poisson homogeneous space $(G/B(v), \pi_{\sG/{\sB(v)}})$
as a Poisson submanifold in $(F_2, \pi_2)$.

\begin{lemma}\label{le:Gv-F2}
For any $v \in W$, the map
\[
E_v:\;\; (G/B(v), \; \pi_{\sG/{\sB(v)}}) \longrightarrow (F_2, \, \pi_2), \;\;\; g_\cdot B(v) \longmapsto [g, \, \ov]_{F_2},
\hs g \in G,
\]
is a Poisson embedding.
\end{lemma}

\subsection{Some auxiliary Poisson morphisms}\label{ssec:auxiliary-poi}
Recall that for any $w \in W$, $B^-wB/B$ is a Poisson submanifold of $(G/B, \pi_{\sG/\sB})$
(see \cite[Theorem 1.5]{GY:GP}), and recall that $BwB$ is a Poisson submanifold of
$(G, \pist)$.  Recall also from $\S$\ref{ssec:ABS} that using the product decomposition $\ow N^- \ow^{\, 1} = N_wN_w^-=N_w^-N_w$, where again
\[
N_w = N \cap (\ow N^- \ow^{\, -1}) \hs \mbox{and} \hs N_w^- = N^- \cap (\ow N^- \ow^{\, -1}),
\]
every $x \in \ow N^- \ow^{\, -1}$ can be uniquely written as
\begin{equation}\label{eq:ggg}
x = x_+ x_- = x^\prime_- x^\prime_+ \hs \mbox{with} \hs x_+, x^\prime_+ \in N_w, \; x_-, x^\prime_- \in N_w^-.
\end{equation}

\begin{lemma}\label{le:Kw-Poi-1} Let $w \in W$.
With  $x \in \ow N^- \ow^{\, -1}$ decomposed as in \eqref{eq:ggg}, the maps
\begin{align*}
&p^w_1: \;\; (wB^-B, \, \pist) \longrightarrow (B^-wB/B, \, \pi_{\sG/\sB}), \;\; x \ow b \longmapsto x_-\ow_\cdot B, \hs
x \in \ow N^- \ow^{\, -1}, \; b \in B,\\
&p^w_2: \;\; (wB^-B, \, \pist) \longrightarrow (BwB, \pist), \;\; x\ow b \longmapsto x^\prime_+\ow b, \hs
x \in \ow N^- \ow^{\, -1}, \; b \in B,
\end{align*}
are both Poisson. With the zero Poisson structure on $T$, the map
\[
j_w \;\; (wB^-B, \, \pist) \longrightarrow (T, 0), \;\; x\ow nt \longmapsto t, \hs x \in \ow N^- \ow^{\, -1},  n \in N, \, t \in N,
\]
is also Poisson.
\end{lemma}

\begin{proof}
Since the projection $(G, \pist)$ to $(G/B, \pi_{\sG/\sB})$ is Poisson, to prove $p^w_1$ is Poisson, it suffices to show that
\[
\tilde{p}^w_1: \;\; (wB^-B/B, \; \pi_{\sG/\sB}) \longrightarrow (B^-wB/B, \, \pi_{\sG/\sB}), \;\; x \ow_\cdot B \longmapsto x_-\ow_\cdot B,
\; x \in \ow N^- \ow^{\, -1}
\]
is Poisson, where again $x \in \ow N^- \ow^{\, -1}$ is decomposed as in \eqref{eq:ggg}. For $g,h \in G$, let
\begin{align*}
&\sigma_{g}:\;\;\; G/B \longrightarrow G/B, \;\; \;g'_\cdot B  \longmapsto gg'_\cdot B, \hs g'\in G,\\
&\sigma_{h_\cdot \sB}:\;\; \;G \longrightarrow G/B, \;\;\; h'\longrightarrow h'h_\cdot B, \hs h' \in G.
\end{align*}
Since the left action of $(G, \pist)$ on $(G/B, \, \pi_{\sG/\sB})$  is Poisson, one has
\[
\tilde{p}^w_1(x \ow_\cdot B) = \tilde{p}^w_1(x_+x_- \ow_\cdot B)= (\widetilde{p}^w_1\circ \sigma_{x_+}) (\pi_{\sG/\sB}(x_-\ow_\cdot B))
+ (\tilde{p}^w_1 \circ \sigma_{x_-\ow_\cdot \sB}) (\pist(x_+)).
\]
Since $B^-wB/B$ is a Poisson submanifold of $(G/B, \pi_{\sG/\sB})$, one has
\[
(\widetilde{p}^w_1\circ \sigma_{x_+}) (\pi_{\sG/\sB}(x_-\ow_\cdot B)) =\pi_{\sG/\sB}(x_-\ow_\cdot B).
\]
Since $N_w$ is a coisotropic subgroup of $(G, \pist)$, one has $(\tilde{p}^w_1 \circ \sigma_{x_-\ow_\cdot \sB}) (\pist(x_+))=0$. Thus
\[
\tilde{p}^w_1(x \ow_\cdot B) = \pi_{\sG/\sB}(x_-\ow_\cdot B).
\]
Hence, $\tilde{p}^w_1$ is Poisson.
Similarly, using the
multiplicativity of $\pist$ and the fact that $N_w^-$ is a coisotropic subgroup of $(G, \pist)$ (which can be proved
using a similar argument as that in the proof of \cite[Lemma 10]{Lu-Mou:groupoid}), one shows that
$p_2^w$ is Poisson.

To show that $j_w$ is Poisson, using the fact that $p_2^w$ is Poisson, it suffices to show that
\[
j_w^\prime: (BwB, \pist) \longrightarrow (T, 0),\;\; j_w^\prime(g' t) = t, \hs g' \in N_w \ow N, \, t \in T,
\]
is Poisson. By \cite[Lemma 10]{Lu-Mou:groupoid},  both $N_w\ow$ and $N$ are coisotropic submanifolds of $(G, \pist)$. By the
multiplicativity of $\pist$, $N_w \ow N$ is also coisotropic with respect to $\pist$. Writing $g \in BwB$ uniquely as $g = g't$,
where $g' \in N_w \ow N$ and $t \in T$, one has
$\pist(g) = r_t \pist(g')$. Hence, $j_w^\prime(\pist(g)) = 0$.
\end{proof}

For $w \in W$ and $Q = B(v)$ or $N(v)$, recall from \eqref{eq:IwQ} the isomorphism
\begin{align*}
I^w_\sQ:& \;\; wB^-B/Q \longrightarrow (B^-wB/B) \times (BwB/Q) \subset (G/B) \times (G/Q), \\
\nonumber &\;\; \;\;\;\; x\ow b_\cdot Q \longmapsto (x_-\ow_\cdot B, \; x^\prime_+ \ow b_\cdot Q), \hs x \in \ow N^- \ow^{\, -1}, \, b \in B,
\end{align*}
where, again, $x \in \ow N^- \ow^{\, -1}$ is decomposed as in \eqref{eq:ggg}.
Note that $I^w_\sQ$ is $T$-equivariant, where $T$ acts on both $G/B$ and $G/Q$ by left translation and on $G/B \times G/Q$ diagonally.
For $\xi \in \h = {\rm Lie}(T)$, let $\rho_{\sG/\sQ}(\xi)$ be the vector field on $G/Q$ given by
\begin{equation}
\rho_{\sG/\sQ}(\xi)(g_\cdot Q) = \frac{d}{dt}|_{t = 0} \exp(t\xi)g_\cdot Q, \hs g \in G.
\end{equation}
 Let again  $\{H_q\}_{q=1}^d$  be
a basis of $\h$ that is orthonormal with respect to $\lara$.
Introduce the mixed product Poisson structure $\pi_{\sG/\sB} \bowtie_{\mu_0} \pi_{\sG/\sQ}$
(see $\S$\ref{ssec:mixed}) on $(G/B) \times (G/Q)$ by
\begin{align*}
&\pi_{\sG/\sB} \bowtie_{\mu_0} \pi_{\sG/\sQ} = (\pi_{\sG/\sB}, \; 0) + (0, \; \pi_{\sG/\sQ}) +\mu_0, \hs \mbox{where}\\
&\mu_0 = \sum_{q=1}^d (\rho_{\sG/\sB}(H_q), \; 0) \wedge (0, \;\rho_{\sG/\sQ}(H_q)).
\end{align*}
It follows from the definition of $\pi_{\sG/\sB} \bowtie_{\mu_0} \pi_{\sG/\sQ}$ that
$(B^-wB/B) \times (BwB/Q)$ is a Poisson submanifold of $((G/B) \times (G/Q), \, \pi_{\sG/\sB} \bowtie_{\mu_0} \pi_{\sG/\sQ})$.

\begin{lemma}\label{le:IwQ-pi}
For every $w \in W$, the
map
\begin{equation}\label{eq:IwQ-pi}
I_\sQ^w: \;\; (wB^-B/Q, \; \pi_{\sG/\sQ}) \longrightarrow ((B^-wB/B) \times (BwB/Q), \;   \pi_{\sG/\sB} \bowtie_{\mu_0} \pi_{\sG/\sQ})
\end{equation}
is a Poisson isomorphism.
\end{lemma}

\begin{proof}
Since $\pi_{\sG/\sQ}$ is a quotient Poisson structure, $I^w_\sQ$ is Poisson as long as $I^w_{\{e\}}$
is Poisson.
Assume thus $Q = \{e\}$ and note that in this case $G/Q = G$ so $\pi_{\sG/\sQ} = \pist$.
Consider the open submanifold $(wB^-B) \times (wB^-B)$ of $G \times G$ and the map
\[
D_w: \; (wB^-B) \times (wB^-B) \longrightarrow (B^-wB/B) \times (BwB), \;\; (x\ow b_1, \; y\ow b_2) \longmapsto (x_-\ow_\cdot B, \; y^\prime_+\ow b_2 ),
\]
where $x, y \in \ow N^- \ow^{\, -1}$, $b_1, b_2\in B$, and $x = x_+x_-$ and $y = y^\prime_- y^\prime_+$ with
$x_+, y^\prime_+ \in N_w$ and $x_-, y^\prime_- \in N_w^-$. We  first prove that
\begin{equation}\label{eq:Dw-Poi}
D_w: \; \left((wB^-B) \times (wB^-B), \; \Pist \right) \longrightarrow \left( (B^-wB/B) \times (BwB), \; \pi_{\sG/\sB} \bowtie_{\mu_0} \pist\right)
\end{equation}
is Poisson.
Indeed, recall that $\Pist = (\pist, 0) + (0, \pist) + \mu_1 + \mu_2$, where $\mu_1$ and $\mu_2$
are respectively given in \eqref{eq:mu-1} and \eqref{eq:mu-2}. By Lemma \ref{le:Kw-Poi-1}, one has
\[
D_w((\pist, 0) + (0, \pist)) = (\pi_{\sG/\sB}, \; 0) + (0, \; \pist).
\]
By the definition of $D_w$,  $D_w(x^L, 0) = 0$ for all $x \in \b={\rm Lie}(B)$.  Thus  $D_w(\mu_2) = 0$.
 It also follows from the definition of $D_w$ that for $\alpha \in \Delta^+$, if
$w^{-1}(\al) \in -\Delta^+$, then $D_w((e_\alpha^R, 0)) = 0$, and if $w^{-1}(\alpha) \in \Delta^+$, then
$D_w((0, e_{-\al}^R)) = 0$. Consequently,
\[
D_w(\mu_1) = D_w\left(\sum_{q}^d  (H^R_q, 0) \wedge (0, H^R_q)\right) = \mu_0.
\]
This shows that $D_w$ in \eqref{eq:Dw-Poi} is Poisson.
As the diagonal embedding
$(wB^-B, \pist) \to\left((wB^-B) \times (wB^-B),  \Pist\right)$
is Poisson, we see that $I^w_{\{e\}}$ is Poisson.
\end{proof}

We now turn to the isomorphism $\zeta^w: B^-wB/B\to \O^{w_0w^{-1}}$, for $w \in W$, given in \eqref{eq:zeta-w}.
Recall also that $t^v = \ov^{\, -1} t \ov \in T$ for $t \in T$ and $v \in W$.

\begin{lemma}\label{le:NNw}
For any $w \in W$, the map
\[
\zeta^w: \;(B^-wB/B, \, \pi_{\sG/\sB}) \longrightarrow (\O^{w_0w^{-1}}, \, \pi_{\sG/\sB}),\;\; m\ow_\cdot B \longmapsto
(\overline{w_0w^{-1}}\, m^{-1})_\cdot B, \;\; m \in N_w^-,
\]
is a $T$-equivariant Poisson isomorphism, where $t \in T$ acts on $B^-wB/B$ by left translation by $t$ and on $\O^{w_0w^{-1}}$
by left translation by $t^{ww_0}$.
\end{lemma}

\begin{proof}
Let $u = w_0w^{-1}$ so that $\ou = \overline{w_0} \,\ow^{\, -1}$. It follows from
$N_u = \ou \,N_w^{-} \,\ou^{\, -1}$ that
$\zeta^w$ is a well-defined $T$-equivariant isomorphism with the $T$-actions on both sides as described.
To show that $\zeta^w$ is a Poisson isomorphism, consider the two Poisson isomorphisms
\begin{align*}
&\zeta_1: \;\; \; (G/B, \, \pi_{\sG/\sB}) \longrightarrow (G/B^-, \, \pi_{\sG/{\sB}^-}), \;\;\; \zeta_1(g_\cdot B) =
g\overline{w_0}_\cdot B^-, \hs g \in G,\\
&\zeta_2: \;\;\; (G/B^-, \, \pi_{\sG/{\sB}^-}) \longrightarrow (B^-\backslash G, \, -\pi_{{\sB}^-\backslash \sG}), \;\;\;
\zeta_2(g_\cdot B^-) = {B^-}_\cdot g^{-1}, \hs g \in G,
\end{align*}
where $\pi_{\sG/{\sB}^-}$ and $\pi_{{\sB}^-\backslash \sG}$ respectively denote the Poisson structures on
$G/B^-$ and $B^-\backslash G$ that are projections of the Poisson structure $\pist$ on $G$. The restriction of composition
$\zeta_2 \circ \zeta_1$ to $(B^-wB/B, \pi_{\sG/\sB}) \subset (G/B, \pi_{\sG/\sB})$ gives the Poisson isomorphism
\begin{align*}
\zeta_3: &\;\; (B^-wB/B, \pi_{\sG/\sB})\longrightarrow (B^-\backslash B^-uB^-,
\, -\pi_{{\sB}^-\backslash \sG}),\\
&\;\;\zeta_3(m\ow_\cdot B) = {B^-}_\cdot \, \overline{w_0} (m\ow)^{-1} = {B^-}_\cdot (\ou \, m^{-1}), \hs m \in N_w^-.
\end{align*}
By \cite[Lemma 14]{Lu-Mou:groupoid}, the map
\[
\zeta_{\ou}: (B^-\backslash B^-uB^-, \, -\pi_{{\sB}^-\backslash \sG})\longrightarrow (BuB/B, \, \pi_{\sG/\sB}), \;\;
\zeta_{\ou}({B^-}_\cdot  x) = x_\cdot B, \;\;x \in N\ou \cap \ou N^-,
\]
is a Poisson isomorphism. As $\zeta^w = \zeta_{\ou} \circ \zeta_3$, we see that $\zeta^w$ is a Poisson isomorphism.
\end{proof}

Let $v, w \in W$. We now relate $(BwB/N(v), \pi_{\sG/\sN(v)})$ and $(BwB/B(v), \pi_{\sG/\sB(v)})$.  Define
\begin{equation}\label{eq:Kwv}
K^w_v: BwB/N(v) \longrightarrow (BwB/B(v)) \times T, \;\;  n_1\ow n_2 t_\cdot N(v) \longmapsto ({n_1\ow n_2}_\cdot B(v), \; t),
\end{equation}
where $n_1 \in N_w, \, n_2 \in N_v$, and $t \in T$. It is clear that $K^w_v$ is a $T$-equivariant isomorphism, where $t_1 \in T$ acts on
$BwB/N(v)$ by left translation by $t_1$ and on $(BwB/B(v)) \times T$ by
\[
t_1 \cdot (g_\cdot B(v), \, t) = (t_1 g_\cdot B(v), \; t_1^wt), \hs t_1, \, t \in T, \, g \in G.
\]
Define the bi-vector field $\pi_{\sG/\sB(v)} \bowtie_{\mu^\prime} 0$ on $(BwB/B(v)) \times T$ by
\begin{align}\label{eq:mu-prime}
&\pi_{\sG/\sB(v)} \bowtie_{\mu^\prime} 0 = (\pi_{\sG/\sB(v)}, \; 0) + \mu^\prime, \hs \mbox{where} \;\\
\nonumber
&\mu^\prime = -\sum_{q=1}^d (\rho_{\sG/\sB(v)}(w(H_q)), \; 0) \wedge (0, \; H_q^R) .
\end{align}

\begin{lemma}\label{le:Kwv-poi}
For any $w, v \in W$,
\[
K^w_v: \; (BwB/N(v), \; \pi_{\sG/\sN(v)}) \longrightarrow ((BwB/B(v)) \times T, \; \pi_{\sG/\sB(v)} \bowtie_{\mu^\prime} 0)
\]
is a Poisson isomorphism.
\end{lemma}

\begin{proof}
Since the projections
\begin{align*}
&(BwB,\; \pist) \longrightarrow (BwB/N(v),  \;\pi_{\sG/\sN(v)}),\\
&(BwB/T, \;\pi_{\sG/\sT}) \longrightarrow (BwB/B(v), \;\pi_{\sG/\sB(v)})
\end{align*}
are Poisson, it is enough to prove Lemma \ref{le:Kwv-poi} for $v = {w_0}$, i.e., to prove that
\[
K := K^w_{w_0}:\;(BwB, \,\pist) \longrightarrow ((BwB/T) \times T, \; \pi_{\sG/\sT} \bowtie_{\mu^\prime} 0)
\]
is Poisson. Consider again $G \times G$ with the multiplicative Poisson structure
$\Pist$ from \eqref{eq:de-PIst}. By \eqref{eq:Pist-mu},
$(BwB) \times G$ is a Poisson submanifold of $(G \times G, \Pist)$.
Define
\[
K^\prime: \; BwB \times G \longrightarrow (G/T) \times T, \;\;(g' t,  \; g) \longmapsto(g_\cdot T, \; t),\hs g' \in N\ow N, \, g \in G, \, t \in T.
\]
Since $K$ is the composition of $K^\prime$ with the diagonal Poisson embedding $(BwB, \pist)\hookrightarrow (BwB \times G, \Pist)$,
$K$ is Poisson once we prove that
\[
K^\prime: (BwB \times G, \;\Pist) \longrightarrow  ((G/T) \times T, \; (\pi_{\sG/\sT}, 0) +  \mu^\prime)
\]
 is Poisson.
Recall again that $\Pist = (\pist, 0) + (0, \pist) + \mu_1 + \mu_2$, with $\mu_1$ and $\mu_2$ respectively given in
 \eqref{eq:mu-1} and \eqref{eq:mu-2}.
By Lemma \ref{le:Kw-Poi-1}, $K^\prime(\pist, 0)=0$. By the definition of $\pi_{\sG/\sT}$,
$K^\prime(0, \pist) = (\pi_{\sG/\sT}, 0)$.  Thus
\[
K^\prime((\pist, 0) + (0, \pist)) = (\pi_{\sG/\sT}, 0).
\]
It follows from the definitions that
$K^\prime(e_\al^L, 0) = K^\prime(e_\al^R, 0) = 0$ for all $\al\in \Delta^+$ and that $K^\prime(0, x^L) = 0$ for $x \in \h$. Furthermore,  it again
follows from the definition of $K^\prime$ that
\[
K^\prime(x^R, 0) = (0, \, (w^{-1}(x))^R) \hs \mbox{and} \hs K^\prime(0, x^R) = (\rho_{\sG/\sT}(x), \, 0), \hs x \in \h.
\]
The fact that $K^\prime$ is Poisson now follows from
\begin{align*}
K^\prime(\mu_1+\mu_2) = K^\prime\left(\sum_{q=1}^d(H_q^R, \,0) \wedge (0, \, H_q^R)\right) \!
=\sum_{q=1}^d (0, \, (w^{-1}(H_q))^R) \wedge (\rho_{\sG/\sT}(H_q), \, 0) = \mu^\prime.
\end{align*}
\end{proof}

For $v, w \in W$, recall from \eqref{eq:BBQ-1} and \eqref{eq:BBQ-2} the isomorphisms
\begin{align*}
&\zeta^{(w, v)}_{\sB(v)}: \;\; BwB/B(v) \longrightarrow \O^{(w, v)}, \;\;\; {n_1\ow n_2}_\cdot B(v) \longmapsto [n_1\ow, \; n_2\ov]_{F_2},\\
&{\zeta}^{(w, v)}_{\sN(v)}: \;\; BwB/N(v) \longrightarrow \O^{(w, v)} \times T, \;\;\; n_1\ow n_2 t_\cdot N(v) \longmapsto ([n_1\ow, \; n_2\ov]_{F_2},\, t),
\end{align*}
where $n_1 \in N_w, \, n_2 \in N_v$ and $t \in T$.  Note that $\zeta^{(w, v)}_{\sB(v)}$ and ${\zeta}^{(w, v)}_{\sN(v)}$
are $T$-equivariant, where $t_1 \in T$ acts on
$BwB/B(v)$ and $BwB/N(v)$ by left translation by $t_1$, on $\O^{(w, v)} \subset F_2$ and on $\O^{(w, v)} \times T$ respectively by
\begin{align*}
&t_1 \cdot [n_1\ow, \, n_2\ov]_{\sF_2} = [t_1n_1\ow, \, n_2\ov]_{\sF_2}, \\
&t_1 \cdot ([n_1\ow, \, n_2\ov]_{\sF_2}, \, t) = ([t_1n_1\ow, \, n_2\ov]_{\sF_2}, \, t_1^wt),
\hs n_1 \in N_w, \, n_2 \in N_v, \, t \in T,
\end{align*}
see \eqref{eq:T-Fr}.
Equip
$\O^{(w, v)}$ with the standard Poisson structure $\pi_{2}$ given in \eqref{eq:pi-r},
and let $\mu^{\prime\prime}$ be the bi-vector field on $\O^{(w, v)} \times T$ by
\begin{equation}\label{eq:mu-prime-prime}
\mu^{\prime\prime} = -\sum_{q=1}^d (\rho_2(w(H_q)), \, 0) \wedge (0, \, H_q^R).
\end{equation}

\begin{lemma}\label{le:BwB-Owv}
The maps
\begin{align*}
&\zeta^{(w, v)}_{\sB(v)}: \;\; (BwB/B(v), \;\pi_{\sG/{\sB(v)}}) \longrightarrow (\O^{(w, v)}, \;\pi_{2}),\\
&{\zeta}^{(w, v)}_{\sN(v)}: \;\; (BwB/N(v), \;\pi_{\sG/{\sN(v)}}) \longrightarrow (\O^{(w, v)} \times T, \;(\pi_{2}, \, 0) + \mu^{\prime\prime})
\end{align*}
are Poisson isomorphisms.
\end{lemma}

\begin{proof}
As $BwB/B(v)$ is a Poisson submanifold of $(G/B(v), \pi_{\sG/{\sB(v)}})$,
by Lemma \ref{le:Gv-F2}, one has the Poisson embedding
\[
(BwB/B(v), \; \pi_{\sG/{\sB(v)}}) \longrightarrow (F_2, \, \pi_2), \;\;n\ow b_\cdot B(v)\longmapsto [n\ow b, \, \ov]_{\sF_2} =
[n\ow, \, b\ov]_{\sF_2},
\]
where  $n \in N_w$ and $b \in B$. Since the image of the above embedding is precisely $\O^{(w, v)}$,
we see that the $\zeta^{(w, v)}_{\sB(v)}$ is a Poisson isomorphism. The fact that
${\zeta}^{(w, v)}_{\sN(v)} = (\zeta^{(w, v)}_{\sB(v)} \times {\rm Id}) \circ K^w_v$ is a Poisson isomorphism follows directly from Lemma \ref{le:Kwv-poi}.
\end{proof}

\subsection{Proof of Theorem \ref{th:JwQ-poi}}\label{ssec:proof-JwQ-poi} Let again $w, v \in W$, and let the notation be as in the statement of Theorem
\ref{th:JwQ-poi}.  First let $Q = B(v)$. By
Lemma \ref{le:IwQ-pi}, Lemma \ref{le:NNw}, and Lemma \ref{le:BwB-Owv}, one has the Poisson isomorphisms
\begin{align}\label{eq:comp-1}
(wB^-B/B(v), \; \pi_{\sG/\sB(v)}) &\; \;\stackrel{I^w_{\sB(v)}}{\!\!\!-\!\!\!\longrightarrow}\;
\left((B^-wB/B) \times (BwB/B(v)), \; \pi_{\sG/\sB} \bowtie_{\mu_0} \pi_{\sG/\sB(v)}\right)  \\
\nonumber & \stackrel{\zeta^w \times \zeta^{(w, v)}_{\sB(v)}}{-\!\!\!-\!\!\!-\!\!\!-\!\!\!-\!\!\!\longrightarrow}\;
\left(\O^{w_0w^{-1}} \times \O^{(w, v)}, \; \pi_{1, 2}\right).
\end{align}
Since $J^w_{\sB(v)} = (\zeta^w \times \zeta^{(w, v)}_{\sB(v)}) \circ I^w_{\sB(v)}$, one sees that
\[
J^w_{\sB(v)}: \;\;  (wB^-B/B(v), \; \pi_{\sG/\sB(v)}) \longrightarrow
\left(\O^{w_0w^{-1}} \times \O^{(w, v)}, \; \pi_{1, 2}\right)
\]
is Poisson.
To see that $J^w_{\sN(v)}$ is Poisson, note that
\[
J^w_{\sN(v)} = \left({\rm Id} \times {\zeta}^{(w, v)}_{\sN(v)}\right) \circ
 (\zeta^w \times {\rm Id}) \circ I^w_{\sN(v)} .
\]
By
Lemma \ref{le:IwQ-pi} and  Lemma \ref{le:NNw},
\[
\left((\zeta^w \times {\rm Id})\circ I^w_{\sN(v)}\right)(\pi_{\sG/\sN(v)}) = (\pi_1, \, 0) + (0, \, \pi_{\sG/\sN(v)}) +\mu_0^\prime,
\]
where $\mu_0^\prime =  \sum_{q=1}^d (\rho_1(w_0w^{-1}(H_q)), \, 0)
\wedge (0, \, \rho_{\sG/\sN(v)}(H_q))$.
By Lemma \ref{le:BwB-Owv} and the $T$-equivariance of ${\zeta}^{(w, v)}_{\sN(v)}$, one has
\begin{align*}
&({\rm Id} \times {\zeta}^{(w, v)}_{\sN(v)})((\pi_1, \, 0) + (0, \, \pi_{\sG/\sN(v)}) +\mu_0^\prime)  = (\pi_1, 0, 0) +(0, \, \pi_2, \, 0) + (0, \, \mu^{\prime\prime}) \\
& \;\;\hs + \sum_{q=1}^d \left((\rho_1(w_0w^{-1}(H_q)), \, 0, \, 0) \wedge ((0, \, \rho_2(H_q), \, 0) + (0, \, 0, \, (w^{-1}(H_q))^R))\right)\\
& \hs =(\pi_1, 0, 0) + (0, \, \pi_2, \, 0)  + \mu_{23} + \mu_{12} + \mu_{13} = \pi_{1, 2, 0}.
\end{align*}
It follows that $J^w_{\sN(v)}: (wB^-B/N(v), \, \pi_{\sG/\sN(v)}) \to (\O^{w_0w^{-1}} \times  \O^{(w, v)} \times T, \, \pi_{1,2,0})$ is Poisson.
This finishes the proof of Theorem \ref{th:JwQ-poi}.

\subsection{$T$-leaves of $(G/Q, \pi_{\sG/\sQ})$}\label{ssec:D}
For any Poisson variety $(X, \piX)$ with an action by a complex torus $\TT$ via Poisson isomorphisms, define (see \cite{Lu-Mou:flags})
the $\TT$-leaf of $\piX$ through $x \in X$ to be
$\TT\Sigma_x = \bigcup_{t \in \TT} t\Sigma_x$, where $\Sigma_x$ is the symplectic leaf of $\piX$ through $x$.
For the $T$-Poisson variety $(G/Q, \pi_{\sG/\sQ})$, where $Q = B(v)$ or $N(v)$ for $v \in W$ and $T$ acts on $G/Q$ by left translation, we now determine the
$T$-leaves of $(G/Q, \pi_{\sG/\sQ})$.

 Let $\varpi_{\sG/\sQ}: G \to G/Q$
be the projection. Recall the monoidal product $\ast$ on $W$ determined by $w \ast s_\al = ws_\al$ if $l(ws_\al) =l (w) + 1$ and
$w \ast s_\al = w$ if $l(ws_\al) = l(w) - 1$.

\begin{theorem}\label{th:D}
For $v \in W$ and $Q = B(v)$ or $N(v)$,  the $T$-leaves of $(G/Q, \pi_{\sG/\sQ})$ are precisely the subvarieties
\[
L_{w, y}^{\sG/\sQ} \, \stackrel{\rm def}{=} \,\varpi_{\sG/\sQ} \left((BwB) \cap B^-yB \ov^{\, -1}\right) \subset G/Q,
\]
where $y, w \in W$ and $y \leq w \ast v$.
\end{theorem}

\begin{proof}
Consider first the case of $Q = B(v)$ and recall from Lemma \ref{le:Gv-F2} the $T$-equivariant
Poisson embedding
\[
E_v:\;\; (G/B(v), \; \pi_{\sG/{\sB(v)}}) \longrightarrow (F_2, \, \pi_2), \;\;\; g_\cdot B(v) \longmapsto [g, \, \ov]_{F_2},
\hs g \in G,
\]
where $T$ acts on $F_2$ by \eqref{eq:T-Fr}.
It follows from the Bruhat decomposition $G = \bigsqcup_{w \in W} BwB$ that the image of $E_v$ is given by
\[
E_v(G/B(v)) = \bigsqcup_{w \in W} \O^{(w, v)} \subset F_2.
\]
The $T$-leaves of $(F_r, \pi_r)$, for any $r \geq 1$, are determined in \cite[Theorem 1.1]{Lu-Mou:flags}. For the case of $r = 2$ at hand,
let
\[
\mu_{\sF_2}: \;\; F_2 \longrightarrow G/B, \;\; [g_1, g_2]_{\sF_2} \longmapsto {g_1g_2}_\cdot B.
\]
By \cite[Theorem 1.1]{Lu-Mou:flags}, the $T$-leaves of $(F_2, \pi_2)$ are precisely the intersections
\[
R^{(w, x)}_y \, \stackrel{\rm def}{=} \, \O^{(w, x)} \cap \mu_{\sF_2}^{-1}(B^-yB/B),
\]
where $w, x, y \in W$ and $y \leq w\ast x$. Thus, for each $w \in W$, $\O^{(w, v)}\subset F_2$ is a union of $T$-leaves $R^{(w, v)}_y$
with $y \leq w \ast v$. It is straightforward to see that
$L_{w, y}^{\sG/\sQ} = E_v^{-1}(R^{(w, v)}_y)$ for all $w, y \in W$ with $y \leq w\ast v$. Thus the $L_{w, y}^{\sG/\sQ}$'s are precisely all the $T$-leaves of
$(G/B(v), \pi_{\sG/\sB(v)})$.

Consider now $Q = N(v)$ and the decomposition $G/N(v) = \bigsqcup_{w \in W} BwB/N(v)$, where note that each $BwB/N(v)$ is a $T$-invariant
Poisson submanifold with respect to $\pi_{\sG/\sN(v)}$. Let $w \in W$ and recall from Lemma \ref{le:Kwv-poi} the
$T$-equivariant Poisson isomorphism
\[
K^w_v: \; (BwB/N(v), \; \pi_{\sG/\sN(v)}) \longrightarrow ((BwB/B(v)) \times T, \; \pi_{\sG/\sB(v)} \bowtie_{\mu^\prime} 0),
\]
where $t_1 \in T$ acts on $(BwB/B(v)) \times T$ by
\[
t_1 \cdot (g_\cdot B(v), \, t) = (t_1 g_\cdot B(v), \; t_1^wt), \hs t_1, \, t \in T, \, g \in G.
\]
By \cite[Lemma 2.23]{Lu-Mi:Kostant}, the $T$-leaves of $((BwB/B(v)) \times T, \; \pi_{\sG/\sB(v)} \bowtie_{\mu^\prime} 0)$ are precisely of
the form $L \times T$, where $L$ is a $T$-leaf of $(BwB/B(v), \pi_{\sG/\sB(v)})$. Applying the $T$-equivariant Poisson isomorphism, one sees that
the $T$-leaves of $(BwB/N(v),  \pi_{\sG/\sN(v)})$ are precisely $L_{w, y}^{\sG/\sN(v)}$, where $y \in W$ and $y \leq w \ast v$.
It follows that $L_{w, y}^{\sG/\sN(v)}$, where $w, y \in W$ and $y \leq w \ast v$, are all the $T$-leaves of $(G/N(v), \pi_{\sG/\sN(v)})$.
\end{proof}

\begin{example}\label{ex:Tleaves-GB}
{\rm
When $v = w_0$, so that $G/N(v) = G$, one has $w\ast w_0 = w_0$ for every $w \in W$, so the condition $y \leq w\ast w_0$ is satisfied
for every $y \in W$, and one has $B^-yB\overline{w_0}^{-1} = B^-yw_0B^-$.
In this case, Theorem \ref{th:D} recovers the well-know result \cite{hodges, reshe-4, Kogan-Z}
that the $T$-leaves (for the $T$-action on $G$ by left translation) of $(G, \pist)$
are precisely all the double Bruhat cells $G^{w, u} = BwB \cap B^-uB^-$, where $w,u \in W$.

When $v = e$, so that $G/B(v) = G/B$, Theorem \ref{th:D} recovers the well-know result from \cite{GY:GP} that the $T$-leaves of
$(G/B, \pi_{\sG/\sB})$ are precisely the open Richardson varieties $(BwB/B) \cap (B^-yB/B)$, where $w, y \in W$ and $y \leq w$.
\hfill $\diamond$
}
\end{example}

The next examples shows that for any $w, v \in W$, the double Bruhat cell $G^{w, v^{-1}}$, as a $T$-leaf of $(G, \pist)$,
 can also be embedding into $G/N(v)$ as a $T$-leaf of $(G/N(v), \pi_{\sG/\sN(v)})$.

\begin{example}\label{ex:Tleaves-G}
{\rm
For an arbitrary $v \in W$, consider the $T$-leaf
\[
L_{w, e}^{\sG/\sN(v)} =\varpi_{\sG/\sN(v)} \left((BwB) \cap B^-B \ov^{\, -1}\right)
\]
of $(G/N(v), \pi_{\sG/\sN(v)})$. Recall from \eqref{eq:BBN} and Lemma \ref{le:nn-v} that every $g \in B^-B\ov^{\, -1}$ is uniquely written
as $g =g_1 n$ with $g_1 \in B^-v^{-1}B^-$ and $n \in N(v)$, and that
we have the embedding
\[
\delta_v: \;\; B^-v^{-1}B^- \longrightarrow G/N(v), \;\; g \longmapsto g_\cdot N(v).
\]
For $w \in W$, denote the restriction of $\delta_v$ to $G^{w, v^{-1}} = BwB \cap B^-v^{-1}B^-$ by
\begin{equation}\label{eq:delta-wv}
\delta_{w, v} = \delta_v|_{\sG^{w, v^{-1}}}:\;\;\; G^{w, v^{-1}} \longrightarrow G/N(v).
\end{equation}
It then follows that the image of $\delta_{w, v}$ is precisely the $T$-leaf $L_{w, e}^{\sG/\sN(v)}$ of $(G/N(v), \pi_{\sG/\sN(v)})$.
As $G^{w, v^{-1}}$ is a $T$-leaf of $(G, \pist)$, we conclude that
\[
\delta_{w, v}:\;\; (G^{w, v^{-1}}, \pist) \stackrel{\sim}{\longrightarrow} (L_{w, e}^{\sG/\sN(v)}, \pi_{\sG/\sN(v)})
\]
is a Poisson isomorphism of $T$-leaves.
\hfill $\diamond$
}
\end{example}

\end{document}